\documentclass{amsart}
\usepackage{amssymb,latexsym, amscd}
\usepackage{color}
\numberwithin{equation}{section}

\newtheorem{theorem}{Theorem}[section]
\newtheorem{proposition}[theorem]{Proposition}
\newtheorem{conjecture}[theorem]{Conjecture}
\newtheorem{corollary}[theorem]{Corollary}

\newtheorem{lemma}[theorem]{Lemma}
\newtheorem{problem}[theorem]{Problem}
\newtheorem{maintheorem}[theorem]{Main Theorem}
\theoremstyle{definition}

\newtheorem{remark}[theorem]{Remark}
\newtheorem{example}[theorem]{Example}
\newtheorem{definition}[theorem]{Definition}

\newcommand{\act}{\mathop{\triangleright}}

\def\endproof{\hfill$\square$\medskip}

\def\ZZ{\mathbb{Z}}

\def\CC{\mathbb{C}}

\def\kk{\Bbbk}

\advance\oddsidemargin by-0.5in
\advance\evensidemargin by-0.5in
\advance\textwidth by 1in

\begin{document}
\author{Arkady Berenstein}
\address{Department of Mathematics, University of Oregon,
Eugene, OR 97403, USA} \email{arkadiy@math.uoregon.edu}

\author{David Kazhdan}
\address{\noindent Department of Mathematics, Heberew University, Jerusalem, Israel}
\email{kazhdan@math.huji.ac.il}


\thanks{The authors were partially supported
by the BSF grant no.~2012365,  
the NSF grant~DMS-1403527 (A.~B.), the ERC grant no.~247049 (D.K.).
}

\makeatletter
\renewcommand{\@evenhead}{\tiny \thepage \hfill  A.~BERENSTEIN and  D.~KAZHDAN \hfill}

\renewcommand{\@oddhead}{\tiny \hfill Hecke-Hopf algebras
 \hfill \thepage}
\makeatother

\title{Hecke-Hopf algebras}


\begin{abstract} Let $W$ be a Coxeter group. The goal of the paper is to construct new Hopf algebras that contain Hecke algebras $H_{\bf q}(W)$ as (left) coideal subalgebras. Our {\it Hecke-Hopf algebras}\footnotemark[1] ${\bf H}(W)$ have a number of applications. In particular they provide new solutions of quantum Yang-Baxter equation and lead to a construction of a new family of endo-functors of the category of $H_{\bf q}(W)$-modules. Hecke-Hopf algebras for the symmetric group are related to Fomin-Kirillov algebras; for an arbitrary Coxeter group $W$ the ``Demazure" part of ${\bf H}(W)$ is being acted upon by generalized braided derivatives which generate the corresponding (generalized) Nichols algebra.   
\end{abstract}

\maketitle

\tableofcontents

\section{Introduction and main results}
\label{sect:intro}

It\footnotetext[1]{In a recent preprint \texttt{arXiv:1608.07509} the term {\it Hopf-Hecke algebras} was used in different context} is well-known that Hecke algebras $H_q(W)$ of Coxeter groups $W$ do not have interesting Hopf algebra structures since the only available one emerges via a complicated isomorphism with the group algebra of $W$ and, moreover  this would  make $H_q(W)$ into  yet another cocommutative Hopf algebra. The goal of this paper is to show how to extend a Hecke algebra $H_q(W)$ to a (non-cocommutative) 
Hopf algebra ${\bf H}(W)$ that contains $H_q(W)$ as a left coideal subalgebra.

We start with the simplest case when $W$ is  the symmetric group $S_n$ generated by $s_i$, $i=1,\ldots,n-1$ subject to the usual Coxeter relations.

\begin{definition}
\label{def:hat H}
For $n\ge 2$ denote by $ {\bf H}(S_n)$ the $\ZZ$-algebra  generated by $s_i$ and $D_i$, $i=1,\ldots,n-1$ subject to relations:

$\bullet$  
$s_i^2=1$,  
$s_iD_i+D_is_i=s_i-1$, $D_i^2=D_i$, $i=1,\ldots,n-1$.

$\bullet$  
$s_js_i=s_is_j$, 
$D_js_i=s_iD_j$, $D_jD_i=D_iD_j$ if $|i-j|>1$.

$\bullet$  
$s_js_is_j=s_is_js_i$, 
$D_is_js_i=s_js_iD_j$, $D_js_iD_j=s_iD_jD_i+D_iD_js_i+s_iD_js_i$ if $|i-j|=1$.

\end{definition}

\begin{remark} We will leave as an exercise to the reader to show that the braid relations $D_iD_{i+1}D_i=D_{i+1}D_iD_{i+1}$ and Yang-Baxter relations $D_is_iD_{i+1}s_iD_{i+1}=
D_{i+1}s_iD_{i+1}s_iD_i$ hold in ${\bf H}(S_n)$.
\end{remark}


\begin{theorem} 
\label{th:hopf hat intro}
For any $n\ge 2$, ${\bf H}(S_n)$ is a Hopf algebra over $\ZZ$ with the
coproduct $\Delta$,  the counit $\varepsilon$, and antipode anti-automorphism $S$ given respectively by (for $i=1,\ldots,n-1$):
$$\Delta(s_i)=s_i\otimes s_i,~\Delta(D_i)=D_i\otimes 1+s_i\otimes D_i,~\varepsilon(s_i)=1,~\varepsilon(D_i)=0,~S(s_i)=s_i,~S(D_i)=-s_iD_i\ .$$

%
%
%
%
%

\end{theorem}

We prove Theorem \ref{th:hopf hat intro} along with its generalization,  Theorem \ref{th:hopf hat intro W}, 
in Section \ref{subsec:Simply-laced}.

\medskip

\begin{remark} In fact, the Hopf agebras ${\bf H}(S_n)$, $n=3,4,5$ were studied in \cite[Section 3.3]{AG} (equations (14)-(18) with $\lambda_1=\lambda_2=0$, $\lambda_3=\frac{1}{4}$) in the context of classification of finite-dimensional pointed Hopf algebras, with the presentation similar to that in Remark \ref{rem:nilpotent presentation} below. It would be interesting to see how would the Hopf algebras ${\bf H}(S_n)$, $n\ge 6$ (as well as ${\bf H}(W)$, where $W$ is any Coxeter group, see below) fit the classification program of pointed Hopf algebras started in \cite{AS} and, conversely, how would a rich theory of pointed Hopf algebras enhance the study of ${\bf H}(S_n)$ and their representations.

\end{remark}

The algebra ${\bf H}(S_n)$ has some additional symmetries.

\begin{theorem} 
\label{th:symmetries Sn}
(a) The assignments $s_i\mapsto -s_i$, $D_i\mapsto 1-D_i$ define an automorphism of ${\bf H}(S_n)$.

(b) The assignments $s_i\mapsto -s_i$, $D_i\mapsto s_i+D_i$ define an automorphism of ${\bf H}(S_n)$.

(c) The assignments $s_i\mapsto s_i$, $D_i\mapsto D_i$ define an anti-automorphism of ${\bf H}(S_n)$.
 
\end{theorem}

We prove Theorem \ref{th:symmetries Sn} along with its generalization, Theorem \ref{th:symmetris D(W)} 
in Section \ref{subsec:proof of Theorem symmetris D(W)}.

Define a family of elements $D_{ij}\in {\bf H}(S_n)$,  $1\le i<j\le n$ by $D_{i,i+1}=D_i$ and $wD_{ij}w^{-1}=D_{w(i),w(j)}$ for any permutation $w\in S_n$ such that $w(i)<w(j)$ (it follows from Definition \ref{def:hat H} that the elements $D_{ij}$ are well-defined). 
Denote by ${\bf D}(S_n)$ the subalgebra of ${\bf H}(S_n)$ generated by all $D_{ij}$. 

\begin{proposition} 
\label{pr:HS_n factorization} For all $n\ge 2$, ${\bf H}(S_n)$ factors as
 ${\bf H}(S_n)={\bf D}(S_n)\cdot \ZZ S_n$ over $\ZZ$, i.e., the multiplication map defines an isomorphism of $\ZZ$-modules
${\bf D}(S_n)\otimes\ZZ S_n\widetilde \longrightarrow {\bf H}(S_n)$.
\end{proposition}

We prove Proposition \ref{pr:HS_n factorization} in Section \ref{subsec:Simply-laced}. 
The algebra ${\bf D}(S_n)$ is can be viewed as a deformed Fomin-Kirillov algebra because of the following result (see also Remark \ref{rem:FK} for more details).

\begin{proposition} 
\label{pr:D presentation Sn}
For $n\ge 2$ the algebra ${\bf D}(S_n)$ is generated by $D_{ij}$, $1\le i<j\le n$ subject to:

$\bullet$ $D_{ij}^2=D_{ij}$ for all $1\le i<j\le n$.

$\bullet$ $D_{ij}D_{k\ell}=D_{k\ell}D_{ij}$ whenever $\{i,j\}\cap \{k,\ell\}=\emptyset$.

$\bullet$ $D_{ij}D_{jk}=D_{ik}D_{ij}+D_{jk}D_{ik}-D_{ik}$, $D_{jk}D_{ij}=D_{ij}D_{ik}+D_{ik}D_{jk}-D_{ik}$
for $1\le i<j<k\le n$.

\end{proposition}

We prove Proposition \ref{pr:D presentation Sn} along with its generalization, Proposition \ref{pr:presentation simply-laced D},  in Section \ref{subsec:Simply-laced}.

\begin{remark} In Section \ref{sec:bosonized} we construct a (Hopf) algebra of symmetries of ${\bf D}(S_n)$ and of its generalizations to arbitrary groups.
These Hopf algebras can be viewed as generalizations of Nichols algebras.
\end{remark}

\noindent Recall that Hecke algebra $H_q(S_n)$ is generated over $\ZZ[q,q^{-1}]$ by $T_1,\ldots,T_{n-1}$ subject to relations:

$\bullet$ Braid relations $T_iT_jT_i =T_jT_iT_j$ if $|i-j|=1$ and $T_iT_j=T_jT_i$ if $|i-j|>1$.

$\bullet$ Quadratic relations $T_i^2=(1-q)T_i+q$.

\begin{theorem} 
\label{th:hecke hopf intro}
For any $n\ge 2$ the assignment $T_i\mapsto s_i+(1-q)D_i$,  
$i=1,\ldots,n-1$ defines an injective homomorphism of $\ZZ[q,q^{-1}]$-algebras $\varphi:H_q(S_n)\hookrightarrow  
{\bf H}(S_n)\otimes \ZZ[q,q^{-1}]$. 

\end{theorem}
We prove Theorem \ref{th:hecke hopf intro} 
in Section \ref{subsec:Simply-laced}.

Thus, it is natural to call ${\bf H}(S_n)$ the {\it Hecke-Hopf algebra} of $S_n$. 

Theorem \ref{th:hecke hopf intro} implies that any ${\bf H}(S_n)\otimes \ZZ[q,q^{-1}]$-module is automatically an $H_q(S_n)$-module. That is, 
the tensor category ${\bf H}(S_n)\otimes \ZZ[q,q^{-1}]-Mod$ of  ${\bf H}(S_n)\otimes \ZZ[q,q^{-1}]$-modules
is equivalent to a sub-category of the (non-tensor) category $H_q(S_n)$-Mod.  We can strengthen this  by noting that the relations $\Delta(\varphi(T_i))=s_i\otimes \varphi(T_i) + D_i\otimes (1-q)$ 
for $i=1,\ldots,n-1$  imply the following result.

\begin{corollary} 
\label{cor:coaction}
In the notation of Theorem \ref{th:hopf hat intro}, the image $\varphi(H_q(S_n))\cong H_q(S_n)$ is a left coideal subalgebra in $ {\bf H}(S_n)$, in particular, the assignment
$T_i\mapsto s_i\otimes T_i+ D_i\otimes (1-q)$, $i=1,\ldots,n-1$,
is a (coassociative and counital) homomorphism of $\ZZ[q,q^{-1}]$-algebras:
\begin{equation}
\label{eq:Hecke-Hopf coaction intro} 
H_q(S_n)\to  {\bf H}(S_n)\otimes  H_q(S_n)\ .
\end{equation}

\end{corollary}

In turn, the coaction \eqref{eq:Hecke-Hopf coaction intro} defines a large family of conservative endo-functors of the category $H_q(S_n)-Mod$.

\begin{corollary} \label{cor:F_U}
For any ${\bf H}(S_n)$-module $M$ the assignments $V\mapsto F_M(V):=M\otimes  V$ 
define a family of endo-functors  on $H_q(S_n)-Mod$ so that $F_{M\otimes N}=F_M\circ F_N$ for all $M,N\in {\bf H}(S_n)-Mod$.
%

\end{corollary}

\begin{remark} If $q=1$, then $\CC S_n$ is a Hopf subalgebra of ${\bf H}(S_n)\otimes \CC$. Of course, this gives a ``classical" analogue $\underline F_M:\CC S_n-Mod\to \CC S_n-Mod$ of the functors $F_M$. However, we do not expect that, under the equivalence of ${\bf H}_q(S_n)-Mod$ with $\CC S_n-Mod$, for a generic $q\in \CC$, the functors $F_M$ will identify with $\underline F_M$. 
\end{remark}

The following result shows the existence of a large 
number of finite-dimensional ${\bf H}(S_n)$-modules.

\begin{proposition} 
\label{pr:demazure}
For any $n\ge 2$, the polynomial algebra $\ZZ[x_1,\ldots,x_n]$ is an ${\bf H}(S_n)$-module algebra via the natural permutation action of $S_n$ and
$$D_i\mapsto \frac{1}{1-x_ix_{i+1}^{-1}}(1-s_i)\ ,$$
the $i$-th Demazure operator. In particular, any graded component of $\ZZ[x_1,\ldots,x_n]$ is an ${\bf H}(S_n)$-submodule. 
\end{proposition}

We prove Proposition \ref{pr:demazure} in Section 
\ref{subsec:Action on Laurent polynomials and verification of Conjecture}.

As an application, for any quadratic solution of QYBE we construct infinitely many new quadratic solutions of QYBE (Section \ref{sec:QYBE}).

Now we generalize the above constructions to arbitrary Coxeter groups $W$.

Recall that a Coxeter group $W$ is generated by $s_i,i\in I$ subject to relations $(s_is_j)^{m_{ij}}=1$, 
where $m_{ij}=m_{ji}\in \ZZ_{\ge 0}$  are such that $m_{ij}=1$ iff $i=j$.

 
\begin{definition} 
\label{def:Hecke-Hopf algebra W}
For any Coxeter group $W=\langle s_i|i\in I\rangle$  we define   
$\hat {\bf H}(W)$ as the $\ZZ$-algebra generated by 
$s_i,D_i$, $i\in I$ subject to relations:

(i) Rank $1$ relations: $s_i^2=1$,  $D_i^2=D_i$, $s_iD_i+D_is_i=s_i-1$ for $i\in I$.


(ii) Coxeter relations: $(s_is_j)^{m_{ij}}=1$ 

(iii) Linear braid relations: $\underbrace{D_is_js_i\cdots s_{j'}}_{m_{ij}} =\underbrace{s_j\cdots s_{i'}s_{j'}D_{i'}}_{m_{ij}}$
for all distinct $i,j\in I$ with $m_{ij}\ne 0$, where $i'=\begin{cases} 
i & \text{if $m_{ij}$ is even}\\
j & \text{if $m_{ij}$ is odd}\\
\end{cases}$ and  $\{i',j'\}=\{i,j\}$.

\end{definition}

\begin{example} The linear braid relation for $W=S_3$ is $D_1s_2s_1=s_2s_1D_2$ and linear braid relations for the dihedral group $W$ of order $8$ are $D_1s_2s_1s_2=s_2s_1s_2D_1$ and $D_2s_1s_2s_1=s_1s_2s_1D_2$.

\end{example}


\begin{theorem} 
\label{th:hopf hat W}
For any Coxeter group  the algebra $\hat {\bf H}(W)$ is a Hopf algebra with the coproduct $\Delta$,  the counit $\varepsilon$, and antipode anti-automorphism $S$ given respectively by (for $i\in I$):
$$\Delta(s_i)=s_i\otimes s_i,~\Delta(D_i)=D_i\otimes 1+s_i\otimes D_i,~\varepsilon(s_i)=1,~\varepsilon(D_i)=0,~S(s_i)=s_i,~S(D_i)=-s_iD_i\ .$$

\end{theorem}

We prove Theorem \ref{th:hopf hat W}  with its generalization to other groups, Theorem \ref{th:hopf hat W gen}, in Section \ref{subsec:proof of  Theorem hopf hat}.

Define ${\mathcal S}:=\{ws_iw^{-1}\,|\,w\in W, i\in I\}$. This is the set of all reflections in $W$. It is easy to see that linear braid relations in 
$\hat {\bf H}(W)$ imply that for any $s\in {\mathcal S}$ there is a unique element $D_s\in \hat {\bf H}(W)$ such that $D_{s_i}=D_i$ for $i\in I$ and $D_{s_iss_i}=s_iD_ss_i$ 
for any $i\in I$, $s\in {\mathcal S}\setminus \{s_i\}$ (Lemma \ref{le:conjugation D}).

Let $\hat {\bf D}(W)$ be the subalgebra of $\hat {\bf H}(W)$ generated by all $D_s$, $s\in {\mathcal S}$ and 
${\bf K}(W):=\bigcap\limits_{w\in W} w\hat {\bf D}(W)w^{-1}$.

By definition, ${\bf K}(W)$ is a  subalgebra of $\hat {\bf D}(W)$ and 
$w{\bf K}(W)w^{-1}={\bf K}(W)$ for all $w\in W$.

\begin{theorem} 
\label{th:Coideal Hopf} 
For any Coxeter group $W$
the ideal ${\bf J}(W)$ generated by ${\bf K}(W)\cap Ker~\varepsilon$
is a Hopf ideal, therefore, the quotient algebra $\underline {\bf H}(W)=\hat {\bf H}(W)/{\bf J}(W)$is a Hopf algebra.

\end{theorem}

We prove Theorem \ref{th:Coideal Hopf} in Section \ref{subsec:proof of Theorem Coideal Hopf}.

\begin{remark} In Section \ref{sect:Generalization to other groups} we  generalize Theorem \ref{th:Coideal Hopf} to arbitrary groups $W$ (Theorem \ref{th:Coideal Hopf gen}) and in Section \ref{sect:Generalization to Hopf algebras} we generalize it even further -- to the case when $W$ is replaced by an arbitrary Hopf algebra $H$ (Theorem \ref{th:general hopf ideal}). 

\end{remark}

We refer to  $\underline {\bf H}(W)$ as the {\it lower Hecke-Hopf algebra} of $W$. 

\begin{definition}
\label{def:hecke algebra}
Given a  Coxeter group $W$, a commutative unital ring $\kk$,  and ${\bf q}=(q_i)\in \kk^I$ such that $q_i=q_j$ whenever $m_{ij}$ is odd, a (generalized) Hecke algebra $H_{\bf q}(W)$ is a $\kk$-algebra generated 
by $T_i$, $i\in I$, subject to relations:

$\bullet$ quadratic relations: $T_i^2=(1-q_i)T_i+q_i$ 
for $i\in I$.

$\bullet$ braid relations: $\underbrace{T_iT_j\cdots}_{m_{ij}} =\underbrace{T_jT_i\cdots}_{m_{ij}}$
for all distinct $i,j\in I$.

\end{definition}

%


\begin{maintheorem} 
\label{th:Hecke in Hecke-Hopf}
For any commutative unital ring $\kk$ the assignments 
$T_i\mapsto s_i+(1-q_i)D_i$,  
$i\in I$, define an injective homomorphism of $\kk$-algebras 
$\underline \varphi_W:H_{\bf q}(W)\to \underline {\bf H}(W)\otimes  \kk$ (whose image is a left coideal subalgebra in $\underline {\bf H}(W)\otimes  \kk$).

\end{maintheorem}

We prove Theorem \ref{th:Hecke in Hecke-Hopf}  in Section \ref{subsec:proof of Theorem Hecke in Hecke-Hopf}.

The following is a corollary from the proof of Theorem \ref{th:Hecke in Hecke-Hopf} (in the case $q_i$ are integer powers of $q$, it was proved in \cite[Section 3.1]{L}).

\begin{corollary} 
\label{cor:Tw form a basis} For any commutative unital ring $\kk$ the Hecke algebra $H_{\bf q}(W)$ is a free $\kk$-module, moreover, the elements $T_w$, $w\in W$ form a $\kk$-basis in $H_{\bf q}(W)$. 

\end{corollary}

Now we will construct a ``Hopf cover" ${\bf H}(W)$ of $\underline {\bf H}(W)$ with an easier to control using the following important structural result.

\begin{theorem}
\label{th:ZWcross and free D} 
For any Coxeter group $W$, the algebra $\hat {\bf D}(W)$  is generated by all $D_s$, $s\in {\mathcal S}$ subject to  relations $D_s^2=D_s$, $s\in {\mathcal S}$. Furthermore,  $\hat {\bf H}(W)$ factors as $\hat {\bf H}(W)=\hat {\bf D}(W)\cdot\ZZ W$, i.e., 
the multiplication map defines an isomorphism of $\ZZ$-modules
$\hat {\bf D}(W)\otimes\ZZ W\widetilde \longrightarrow \hat {\bf H}(W)$.

\end{theorem}

We prove Theorem \ref{th:ZWcross and free D} in Section \ref{subsec:proof of Theorem ZWcross and free D}.

\begin{remark} In Section \ref{sect:Generalization to other groups} we  extend this factorization result to arbitrary groups (Theorem \ref{th:ZWcross and free D gen H}) and in Section \ref{sect:Generalization to Hopf algebras} we generalize it even further  (Lemma \ref{le:general factored quotient H}). 

\end{remark}

Using Theorem \ref{th:ZWcross and free D}, 
we identify $\hat {\bf D}(W_J)$ with a subalgebra of $\hat {\bf D}(W)$ for any $J\subset I$ by claiming that $\hat {\bf D}(W_J)$ is generated by all  $D_s$ with $s\in {\mathcal S}\cap W_J$.

For distinct $i,j\in I$ denote by ${\bf K}_{ij}(W)$ the set of all elements in ${\bf K}(W_{\{i,j\}})\cap Ker~\varepsilon\subset \hat {\bf D}(W_{\,i,j})$ having degree at most $m_{ij}$, 
where we view the free algebra $\hat {\bf D}(W)$ as naturally filtered by $\deg D_s=1$ for $s\in {\mathcal S}$ (clearly, ${\bf K}_{ij}(W)=\{0\}$ if $m_{ij}=0$).

\begin{theorem} 
\label{th:Coideal Hopf ij}
For any Coxeter group $W$ the ideal $\underline {\bf J}(W)$ generated by  by all ${\bf K}_{ij}(W)$, $i,j\in I$, $i\ne j$, is  a Hopf ideal, therefore, the quotient algebra ${\bf H}(W)=\hat {\bf H}(W)/\underline {\bf J}(W)$ is a Hopf algebra.
\end{theorem}

We prove Theorem \ref{th:Coideal Hopf ij} in Section \ref{subsec:proof of Theorem Coideal Hopf}.

We call ${\bf H}(W)$ the {\it Hecke-Hopf algebra} of $W$. 

When $W$  is {\it simply-laced}, i.e., $m_{ij}\in \{0,2,3\}$ for all distinct $i,j\in I$,  
we find the presentation of  ${\bf H}(W)$, thus generalizing that for $S_n$ in Definition \ref{def:hat H}.

\begin{theorem} 
\label{th:hopf hat intro W}
Suppose that $W$ is simply-laced. Then the Hecke-Hopf algebra ${\bf H}(W)$ is generated by $s_i,D_i$, $i\in I$ subject to relations:

$\bullet$  
$s_i^2=1$,  
$s_iD_i+D_is_i=s_i-1$, $D_i^2=D_i$ for $i\in I$.

$\bullet$  
$s_js_i=s_is_j$, 
$D_js_i=s_iD_j$, $D_jD_i=D_iD_j$ if $m_{ij}=2$.

$\bullet$  
$s_js_is_j=s_is_js_i$, 
$D_is_js_i=s_js_iD_j$, $D_js_iD_j=s_iD_jD_i+D_iD_js_i+s_iD_js_i$ if $m_{ij}=3$.

\end{theorem}

We prove Theorem \ref{th:hopf hat intro W} in Section \ref{subsec:Simply-laced}.

\begin{remark} In Section \ref{sec:bosonized} we show that by ``homogenizing" the relations in Theorem \ref{th:hopf hat intro W}, one obtains a Hopf algebra ${\bf H}_0(W)$ (Definition \ref{def:tilde H}) which acts on  ${\bf D}(W)$ via braided derivatives and thus is closely related to the corresponding Nichols algebra.
\end{remark}

\begin{remark} 
\label{rem:nilpotent presentation}
It follows from Theorem \ref{th:hopf hat intro W} that the algebra ${\bf H}(W)\otimes \frac{1}{2}\ZZ$ for simply-laced $W$ has the following presentation in generators $s_i$ and $d_i=D_i+\frac{1}{2}(s_i-1)$:

$\bullet$  
$s_i^2=1$,  
$s_id_i+d_is_i=0$, $d_i^2=0$ for $i\in I$.

$\bullet$  
$s_js_i=s_is_j$, 
$d_js_i=s_id_j$, $d_jd_i=d_id_j$ if $m_{ij}=2$.

$\bullet$  
$s_js_is_j=s_is_js_i$, 
$d_is_js_i=s_js_id_j$, $d_js_id_j=s_id_jd_i+d_id_js_i+\frac{1}{4}(s_i-s_is_js_i)$ if $m_{ij}=3$. 

\end{remark}

Actually, both  $\underline {\bf H}(W)$ and ${\bf H}(W)$ can be factored in the sense of Theorem \ref{th:ZWcross and free D} as follows.

\begin{theorem} 
\label{th:HW factorization}
 $\underline {\bf H}(W)=\underline {\bf D}(W)\cdot \ZZ W$,~~${\bf H}(W)={\bf D}(W)\cdot \ZZ W$ for all Coxeter groups $W$,
where $\underline {\bf D}(W)$, ${\bf D}(W)$ are respectively the images of $\hat {\bf D}(W)$ under the projections $\hat {\bf H}(W)\twoheadrightarrow \underline {\bf H}(W)$, $\hat {\bf H}(W)\twoheadrightarrow {\bf H}(W)$.

\end{theorem} 

We prove Theorem \ref{th:HW factorization} in Section \ref{subsec:proof of Theorem Coideal Hopf}.

\begin{remark} It is natural to ask whether   $\underline {\bf D}(W)$ and ${\bf D}(W)$ are free as $\ZZ$-modules.

\end{remark} 

We  extend Proposition \ref{pr:D presentation Sn} and  provide an explicit description of  ${\bf D}(W)$ for an arbitrary simply-laced Coxeter group $W$.

\begin{definition}
 Given a Coxeter group $W$, we say that a pair $(s,s')$ of distinct reflections is {\it compatible} if there are $i,j\in I$ and $w\in W$ such that  $s=ws_iw^{-1}$, $s'=ws_jw^{-1}$, $\ell(ws_i)=\ell(w)+1$, $\ell(ws_j)=\ell(w)+1$. 
\end{definition}

For $w,w'\in W$ denote by $m_{w,w'}\in \ZZ_{\ge 0}$ the order of $ww'$ in $W$ (if it is infinite, we set $m_{w,w'}=0$).


\begin{proposition} 
\label{pr:presentation simply-laced D}
In the assumptions of Theorem \ref{th:hopf hat intro W}, the algebra ${\bf D}(W)$ is generated by $D_s$, $s\in {\mathcal S}$ subject to relations:
 
$\bullet$ $D_s^2=D_s$ for all $s\in {\mathcal S}$.

$\bullet$ $D_sD_{s'}=D_{s'}D_s$ for all compatible pairs $(s,s')\in {\mathcal S}\times {\mathcal S}$ with $m_{s,s'}=2$.

$\bullet$ $D_sD_{s'}=D_{ss's}D_s+D_{s'}D_{ss's}-D_{ss's}$ for all compatible pairs $(s,s')\in {\mathcal S}\times {\mathcal S}$ with $m_{s,s'}=3$.

\end{proposition}

We prove Proposition \ref{pr:presentation simply-laced D} in Section \ref{subsec:Simply-laced}.

\begin{remark} It would be interesting to find a more explicit characterization of compatible pairs $(s,s')$ with a given $m_{s,s'}$. For instance, we expect that in a simply-laced $W$ each pair $(s,s')$ of reflections with $m_{s,s'}=2$, i.e., $ss'=s's$, is compatible.  

\end{remark}

The following is a refinement of Theorem \ref{th:Hecke in Hecke-Hopf}. 

\begin{theorem} 
\label{th:upper Hecke in Hecke-Hopf}
In the notation of Theorem \ref{th:Hecke in Hecke-Hopf}, the assignments 
$T_i\mapsto s_i+(1-q_i)D_i$, $i\in I$,  define an injective homomorphism of algebras $\varphi_W:H_{\bf q}(W)\hookrightarrow {\bf H}(W)$. Moreover,  $\underline \varphi_W=(\pi_W\otimes 1)\circ \varphi_W$, where $\pi_W:{\bf H}(W)\twoheadrightarrow \underline {\bf H}(W)$ is the canonical surjective homomorphism of Hopf algebras  (which is identity on $\ZZ W$ and 
$\pi_W({\bf D}(W))=\underline {\bf D}(W)$). 

\end{theorem}

We prove Theorem \ref{th:upper Hecke in Hecke-Hopf} in Section \ref{subsec:proof of Theorem Hecke in Hecke-Hopf}.

It follows from Theorem \ref{th:hopf hat intro W} that both $\pi_{S_2\times S_2}$ and $\pi_{S_3}$ are the identity maps and that both definitions of ${\bf H}(S_n)$ agree. One can ask whether  $\pi_{S_n}$ is an isomorphism for $n\ge 4$. 

It follows from  Theorem \ref{th:upper Hecke in Hecke-Hopf} that 
 braid relations 
\begin{equation}
\label{eq:braid relations}
\underbrace{D_iD_j\cdots}_{m_{ij}} =\underbrace{D_jD_i\cdots}_{m_{ij}}
\end{equation}
hold in ${\bf D}(W)$.
In fact, there are other relations in ${\bf D}(W)$.

\begin{theorem}
\label{th:relations rank 2}
Given a Coxeter group $W$, for any distinct $i, j\in I$ with $m:=m_{ij}\ge 2$ and any $w\in W$ such that $\ell(ws_i)=\ell(w)+1,\ell(ws_j)=\ell(w)+1\}$ the following relations hold in ${\bf D}(W)$  for all divisors $n$ of $m$, $r\in [1,n]$ (where we abbreviated $D_k:=D_{w\cdot \underbrace{s_is_j\cdots}_{2k-1}\cdot w^{-1}}\in {\bf D}(W)$, 
$k=1,\ldots,m$):

(a) Quadratic-linear relations (for $1\le p<\frac{m}{2n}$):  
$$\displaystyle{\sum\limits_{0\le a<b< \frac{m}{n}:b-a=\frac{m}{n}-p} D_{r+an} D_{r+bn} = \sum\limits_{0\le a'<b'< \frac{m}{n}:b'-a'=p} D_{r+b'n} D_{r+a'n}-
\sum\limits_{p\le c<\frac{m}{n}-p} D_{r+cn}}\ ,$$
$$\displaystyle{\sum\limits_{0\le a<b< \frac{m}{n}:b-a=\frac{m}{n}-p} D_{r+bn} D_{r+an} = \sum\limits_{0\le a'<b'< \frac{m}{n}:b'-a'=p} D_{r+a'n} D_{r+b'n}-
\sum\limits_{p\le c<\frac{m}{n}-p} D_{r+cn}}\ .$$

(b) Yang-Baxter type relations (for $0\le t\le \frac{m}{n}$): 
$$\overset{\longrightarrow}{\prod\limits_{t\le a\le \frac{m}{n}-1}}(1-D_{r+an})
\overset{\longrightarrow}{\prod\limits_{0\le b\le t-1}} D_{r+bn}
=\overset{\longleftarrow}{\prod\limits_{0\le b\le t-1}} D_{r+bn}\overset{\longleftarrow}{\prod\limits_{t\le a\le \frac{m}{n}-1}}(1-D_{r+an})\ .$$




\end{theorem}

We prove Theorem \ref{th:relations rank 2} in Section \ref{subsec:proof of th:relations rank 2}.



\begin{remark} 
\label{rem:presentation H(W)} After the first version of the present paper was posted to Arxiv, Dr. Weideng Cui informed us that he found a presentation of ${\bf D}(W_{\{i,j\}})$ for $m_{ij}\in \{4,6\}$ in \cite{Cui}. That is, if $m_{ij}=4$, ${\bf D}(W_{\{i,j\}})$ is generated by $D_s$, $s\in {\mathcal S}$ subject to relations $D_s^2=D_s$, $s\in {\mathcal S}$, the braid relations \eqref{eq:braid relations}, and the  
relations from Theorem \ref{th:relations rank 2}. If $m_{ij}=6$, there are more relations than those prescribed by Theorem \ref{th:relations rank 2}. 

\end{remark}

\begin{remark} If $W$ is not crystallographic, i.e.,  $m_{ij}\in \{5\}\sqcup \ZZ_{\ge 7}$ for some distinct $i,j\in I$, we expect even more relations in ${\bf D}(W)$, e.g., for $m_{ij}=5$, one can show that:
\begin{equation}
\label{eq:Qij4}
D_1D_2D_3D_4+(D_1D_2D_4+D_2D_3D_4-D_2D_4)(D_5-1)=D_5D_4D_3D_1+D_5D_3D_2D_1-D_5D_3D_1
\end{equation}
where we abbreviated $D_1:=D_{s_i}$, $D_2=D_{s_is_js_i}$, $D_3=D_{s_js_is_js_is_j}=D_{s_is_js_is_js_i}$, 
$D_4=D_{s_js_is_j}$, $D_5=D_{s_j}$ (as in Theorem \ref{th:relations rank 2}).
We do not expect \eqref{eq:Qij4} to follow from the quadratic relations (Theorem
\ref{th:relations rank 2}(a)). 

\end{remark}

Now we establish a number of symmetries of ${\bf H}(W)$ and ${\bf D}(W)$. 

\begin{theorem} 
\label{th:symmetris D(W)} For any Coxeter group $W$ one has:

\noindent (a) ${\bf H}(W)$ and $\underline {\bf H}(W)$ 
admit an anti-involution 
$\overline{\cdot}$ such that $\overline s_i=s_i, ~\overline D_i=D_i$ for $i\in I$.

\noindent (b) ${\bf D}(W)$ and $\underline {\bf D}(W)$ admit the following symmetries. 
 
(i) A $W$-action by automorphisms via $s_i(D_s)=\begin{cases} 
1-D_{s_i} & \text{if $s=s_i$}\\
D_{s_iss_i} & \text{if $s=s_i$}\\ 
\end{cases}$ for $s\in {\mathcal S}$, $i\in I$.

(ii) An $s$-derivation $d_s$ (i.e., $d_s(xy)=d_s(x)y+s(x)d_s(y)$) such that $d_s(D_{s'})=\delta_{s,s'}$, $s,s'\in {\mathcal S}$.

\noindent (c) ${\bf H}(W)$ 
admits an involution 
$\theta$ such that  $\theta(s_i)=-s_i,~\theta(D_i)= 1-D_i$  for $i\in I$.
 
\end{theorem}

We prove Theorem \ref{th:symmetris D(W)} in Section \ref{subsec:proof of Theorem symmetris D(W)}.

\begin{remark} Proposition \ref{pr:eps theta} implies that for a finite Coxeter group $W$, the algebra $\underline {\bf H}(W)$ 
also admits an involution $\theta$ as in Theorem \ref{th:symmetris D(W)}(c). Based on a more general argument of Proposition \ref{pr:antipode theta}(b), we can conjecture this for all Coxeter groups. In fact, the ``innocently looking" Theorem \ref{th:symmetris D(W)}(c) is highly nontrivial, in particular, applying $\theta$ in the form $\theta=S^{-2}$ (according to Proposition \ref{pr:antipode theta}(a)) to the  the braid relations \eqref{eq:braid relations} in ${\bf D}(W)$ one can obtain a large number of relations in degrees less than $m_{ij}$. 

\end{remark}



%
%


Using Theorem \ref{th:upper Hecke in Hecke-Hopf}, we can extend Corollaries \ref{cor:coaction} and \ref{cor:F_U} to all Coxeter groups.  

\begin{corollary} 
\label{cor:F_U W} Let $W$ be a Coxeter group. Then: 

(a) $H_{\bf q}(W)$ is a left ${\bf H}(W)$-comodule algebra  via
$T_i\mapsto s_i\otimes T_i+D_i\otimes (1-q_i)$ for $i\in I$.

(b) For any  ${\bf H}(W)$-module $M$ the assignments $V\mapsto F_M(V):=M\otimes  V$ 
define a family of (conservative) endo-functors 
on $H_{\bf q}(W)-Mod$ so that $F_{M\otimes N}=F_M\circ F_N$ for all $M,N\in {\bf H}(W)-Mod$.


\end{corollary}

Furthermore, let us extend Proposition \ref{pr:demazure}  to all $W$. Recall from \cite{Kac} that an $I\times I$-matrix $A=(a_{ij})$ is a generalized Cartan matrix if $a_{ii}=2$, $a_{ij}\in \ZZ_{\le 0}$ for $i\ne j$ and $a_{ij}\cdot a_{ji}=0$ implies $a_{ij}=a_{ji}=0$.

The following is a (conjectural) generalization of Proposition \ref{pr:demazure} to {\it crystallographic} Coxeter groups $W$, i.e., such that $m_{ij}\in \{0,2,3,4,6\}$ for all distinct $i,j\in I$.

\begin{conjecture} 
\label{conj:laurent action}
Let $A=(a_{ij})$, $i,j\in I$ be a generalized Cartan matrix. Let $W=W_A$ be the corresponding crystallographic Coxeter group, i.e., $m_{ij}=
\begin{cases}
2+a_{ij}a_{ji} & \text{if $a_{ij}a_{ji}\le 2$}\\
6   & \text{if $a_{ij}a_{ji}=3$}\\
0   & \text{if $a_{ij}a_{ji}>3$}\\
\end{cases}$
for $i,j\in I$, $i\ne j$ and let ${\mathcal L}_I=\ZZ[t_i^{\pm 1}],i\in I$.
Then the assignments 
\begin{equation}
\label{eq:demazure W}
s_i(t_j):=t_i^{-a_{ij}}t_j,~D_i(t_j):=t_j\frac{1-t_i^{-a_{ij}}}{1-t_i}\ , 
\end{equation}
for $i,j\in I$ turn ${\mathcal L}_I$ into an ${\bf H}(W)$-module algebra.
%
%
 %

\end{conjecture}

We verified the conjecture in the simply-laced case, i.e., when $A$ is symmetric (Section \ref{subsec:Action on Laurent polynomials and verification of Conjecture}). We also verified that $\hat {\bf H}(W)$, indeed, acts on ${\mathcal L}_I$ via \eqref{eq:demazure W} (Proposition \ref{pr:hat H action on LI and vanishing of K_ij}) for any $A$. After the first version of the present paper was posted to Arxiv, Dr. Weideng Cui informed us that he proved the conjecture for $m_{ij}\in \{4,6\}$ in 
\cite{Cui} using his presentation of ${\bf H}(W_{\{i,j\}})$ (see Remark \ref{rem:presentation H(W)}), that is, for all crystallographic Coxeter groups.

It would also be interesting to see if ${\bf K}(W)$ annihilates ${\mathcal L}_I$ as well, i.e., 
if the desired action of ${\bf H}(W)$ on ${\mathcal L}_I$ factors through that of $\underline {\bf H}(W)$.

We conclude Introduction with the observation that all results of this section extend to what we call {\it extended Coxeter groups} $\hat W$. Namely, given a Coxeter group $\langle s_i|i\in I\rangle$, we let $\hat W$  be any group generated by $\hat s_i$, $i\in I$ such that 

$\bullet$ $\hat s_i^2$ is central 
for $i\in I$.

$\bullet$ braid relations: $\underbrace{\hat s_i\hat s_j\cdots}_{m_{ij}} =\underbrace{\hat s_j\hat s_i\cdots}_{m_{ij}}$
for all distinct $i,j\in I$.

$\bullet$ The assignments $\hat s_i\mapsto s_i$, define a (surjective) group homomorphism $\hat W\twoheadrightarrow W$.

\smallskip

Clearly, in any extended  Coxeter group $W$ one has a relation $\hat s_i^2=\hat s_j^2$ whenever $m_{ij}$ is odd. In particular, $\hat S_n$ is a central extension of $S_n$ with the cyclic center.

Then Definitions \ref{def:hat H} and \ref{def:Hecke-Hopf algebra W} carry over and give ${\bf H}(\hat S_n)$ and $\hat {\bf H}(\hat W)$ with the only modification: the rank 1 relations $s_iD_i+D_is_i=s_i-1$ are replaced with $\hat s_iD_i+D_i\hat s_i=\hat s_i-\hat s_i^2$ 
because $\hat s_i$ is not necessarily an involution.   Then Theorems \ref{th:hopf hat intro}, \ref{th:symmetries Sn} and Proposition \ref{pr:HS_n factorization} hold for ${\bf H}(\hat S_n)$ with ${\bf D}(\hat S_n)={\bf D}(S_n)$. So do Theorems \ref{th:hopf hat W},~\ref{th:Coideal Hopf},~\ref{th:ZWcross and free D},~\ref{th:Coideal Hopf ij},~\ref{th:HW factorization}, and \ref{th:symmetris D(W)} for $\hat {\bf H}(\hat W)$ with $\hat {\mathcal S}=\{\hat w\hat s_i\hat w^{-1}\,|\,\hat w\in \hat W,i\in I\}$, $\hat {\bf D}(\hat W)=\hat {\bf D}(W)$, ${\bf K}(\hat W)={\bf K}(W)$, and $\underline {\bf K}(\hat W)=\underline {\bf K}(W)$. By the very construction, the canonical homomorphism $\hat W\twoheadrightarrow W$ defines surjective homomorphisms of Hopf algebras $\hat {\bf H}(\hat W)\twoheadrightarrow \hat {\bf H}(W)$, 
$\underline {\bf H}(\hat W)\twoheadrightarrow \underline {\bf H}(W)$, and ${\bf H}(\hat W)\twoheadrightarrow {\bf H}(W)$.   

Finally, to establish analogues of Theorems \ref{th:Hecke in Hecke-Hopf} and \ref{th:upper Hecke in Hecke-Hopf} for a given extended Coxeter group $\hat W$, in the notation of Definition \ref{def:hecke algebra},
define a {\it (generalized) Hecke algebra} $H_{\bf q}(\hat W)$ of $\hat W$, to be generated  over a commutative ring $\kk$ by $T_i$, $z_i$, $i\in I$ subject to relations:

$\bullet$ The assignments $z_i\mapsto \hat s_i^2$ define an injective homomorphism of algebras $\kk[z_i,i\in I]\hookrightarrow \kk \hat W$, where $\kk[z_i,i\in I]$ denotes the subalgebra generated by $z_i$, $i\in I$.

$\bullet$  quadratic relations: $T_i^2=(1-q_i)T_i+q_iz_i$ 
for $i\in I$.

$\bullet$ braid relations: $\underbrace{T_iT_j\cdots}_{m_{ij}} =\underbrace{T_jT_i\cdots}_{m_{ij}}$
for all distinct $i,j\in I$.

Then Theorems \ref{th:Hecke in Hecke-Hopf}, \ref{th:upper Hecke in Hecke-Hopf} and Corollary \ref{cor:Tw form a basis} hold verbatim for $H_{\bf q}(\hat W)$ and $\underline {\bf H}(\hat W)$, ${\bf H}(\hat W)$. 


\subsection*{Acknowledgments}
The first named author gratefully 
acknowledges the support of Hebrew University of Jerusalem where most of this work was done.  
We thank  Pavel Etingof, Jacob Greenstein, and Jianrong Li for stimulating discussions. Special thanks are due to Yuri Bazlov for pointing out that some results of the forthcoming paper \cite{BBHcross} can be of importance in this work. We express our gratitude to Dr. Weideng Cui who helped correcting the proof of Lemma \ref{le:drop from 3 to 2} and shared with us results of his paper \cite{Cui} on presentation of the Hecke-Hopf algebras ${\bf H}(W)$ for dihedral groups of $8$ and $12$ elements respectively, which, in particular, settled Conjecture \ref{conj:laurent action}.

\section{New solutions of QYBE}

\label{sec:QYBE}

We retain the notation of Section \ref{sect:intro}. The following is immediate. 

\begin{lemma} 
\label{le:quadratic braiding}
Let $n\ge 3$. Then for a  $\kk$-module $V$ and a $\kk$-linear map $\Psi:V\otimes V\to V\otimes V$  the following are equivalent.

(i) the assignments $T_i\mapsto \Psi_i=\underbrace{Id_V\otimes \cdots Id_V}_{i-1}\otimes \Psi\otimes \underbrace{Id_V\otimes \cdots \otimes Id_V}_{n-i-1}$
for $i=1,\ldots,n-1$ define a structure of an $H_q(S_n)$-module on $V^{\otimes n}$;


(ii) $\Psi$ satisfies the braid equation on $V^{\otimes 3}$ and the quadratic equation on $V^{\otimes 2}$: 
\begin{equation}
\label{eq:quadratic braiding}
\Psi_1\Psi_2\Psi_1=\Psi_2\Psi_1\Psi_2,~\Psi^2=(1-q)\Psi+q\cdot Id_{V\otimes V}
\end{equation}
(where $\Psi_1:=\Psi\otimes Id_V$, $\Psi_2:=Id_V\otimes \Psi$).

\end{lemma} 

We refer to any $\Psi$ satisfying \eqref{eq:quadratic braiding} as a  {\it quadratic braiding} on $V$. 

In a similar fashion, we obtain the following immediate result for ${\bf H}(S_n)$-modules.

\begin{lemma} 
\label{le:hatHS3 structure}
Let $n\ge 3$.  Then for any $\ZZ$-module $U$ and any pair of $\ZZ$-linear maps  $s,D:U\otimes  U\to U\otimes  U$  the following are equivalent:

(a)  the assignments:
$s_i\mapsto \underbrace{Id_U\otimes \cdots Id_U}_{i-1}\otimes s\otimes \underbrace{Id_U\otimes \cdots \otimes Id_U}_{n-i-1},~D_i\mapsto \underbrace{Id_U\otimes \cdots Id_U}_{i-1}\otimes D\otimes \underbrace{Id_U\otimes \cdots \otimes Id_U}_{n-i-1}$
for $i=1,\ldots,n-1$ define a structure of an ${\bf H}(S_n)$-module on $U^{\otimes n}$;

(b) the assignments 
$s_1\mapsto s\otimes Id_U,~s_2\mapsto Id_U\otimes s,~D_1\mapsto D\otimes Id_U,~D_2\mapsto Id_U\otimes D$
define a structure of an ${\bf H}(S_3)$-module on $U^{\otimes 3}$.
\end{lemma}

We refer to any pair of  $\ZZ$-linear maps 
$s,D:U\otimes  U\to U\otimes  U$ satisfying Lemma \ref{le:hatHS3 structure}(b) as an {\it ${\bf H}(S_3)$-structure} on  $U$.

Furthermore, given an ${\bf H}(S_3)$-structure $s,D:U\otimes  U\to U\otimes  U$ on any  $\ZZ$-module $U$, any $\kk$-module $V$ and any $\kk$-linear map $\Psi:V\otimes V\to V\otimes V$  
define a $\kk$-linear endomorphism $\Psi_U$  of $(U\otimes  V)^{\otimes 2}$ by:
\begin{equation}
\label{eq:PsiU}
\Psi_U=\tau_{23}^{-1}\circ (s\otimes \Psi+(1-q) D\otimes Id_{V\otimes V})\circ \tau_{23}\ ,
\end{equation}
where $\tau_{23}:(U\otimes  V)\otimes (U\otimes  V)\widetilde \to (U\otimes  U)\otimes  (V\otimes V)$ is the permutation of two middle factors.

The following result was the starting point of the entire project. 

\begin{theorem} 
\label{th:H(S3)-structure}
Let $s,D:U\otimes  U\to U\otimes  U$ be any ${\bf H}(S_3)$-structure on a $\ZZ$-module $U$ and let $\Psi:V\otimes V\to V\otimes V$ be a quadratic braiding on a $\kk$-module $V$. Then 

(a) The linear endomorphism
 $\Psi_U$ of $(U\otimes  V)^{\otimes 2}$
is also a quadratic braiding.

(b) The functor  $F_{U^{\otimes n}}$ from Corollary \ref{cor:F_U} satisfies:
$F_{U^{\otimes n}}(V^{\otimes n})\cong (U\otimes  V)^{\otimes n}$,  
where $V^{\otimes n}$ is naturally an $H_q(S_n)$-module by Lemma \ref{le:quadratic braiding} via the quadratic braiding $\Psi$ and $(U\otimes  V)^{\otimes n}$ is naturally an ${\bf H}(S_n)$-module via the quadratic braiding $\Psi_U$.

\end{theorem}

\begin{proof}  
Let  $\Psi:V\otimes V\to V\otimes V$ be a quadratic braiding. By Lemma \ref{le:quadratic braiding}, the assignment $T_i\mapsto \Psi_i$, $i=1,\ldots,n-1$ defines a $\kk$-algebra homomorphism $H_q(S_n)\to End_\kk(V^{\otimes n})$.

Furthermore, let $U$ be a $\ZZ$-module with an ${\bf H}(S_3)$-structure $s,D:U\otimes  U\to U\otimes  U$. Then, clearly, $U^{\otimes n}$ is an ${\bf H}(S_n)$-module by Lemma \ref{le:hatHS3 structure}. 
Tensoring these homomorphisms, we obtain an algebra homomorphism ${\bf H}(S_n)\otimes H_q(S_n)\to End_\ZZ (U^{\otimes n}) \otimes End_{\kk} (V^{\otimes n})\subset End_{\kk} (U^{\otimes n}\otimes V^{\otimes n})$. Composing it with the coaction \eqref{eq:Hecke-Hopf coaction intro}  and naturally 
identifying $U^{\otimes n}\otimes V^{\otimes n}$ with $(U\otimes V)^{\otimes n}$ we obtain a $\kk$-algebra homomorphism $H_q(S_n)\to End_{\kk} ((U\otimes V)^{\otimes n})$ given by $T_i\mapsto (\Psi_U)_i$ for $i=1,\ldots,n-1$. In view of Lemma  \ref{le:quadratic braiding}, $\Psi_U$ is a quadratic braiding. This proves (a). Part (b) also follows.

The theorem is proved.
\end{proof}

\begin{remark} We found a particular case of Theorem \ref{th:H(S3)-structure} in \cite[Formula (4.8)]{GRV} and \cite[Formula (32)]{KMS}, but  the general case seems to be unavailable in the literature.

\end{remark}

The following immediate corollary of Proposition \ref{pr:demazure} provides an example of an ${\bf H}(S_3)$-structure.

\begin{corollary} Let $U=\ZZ[x]$. Then the permutation of factors $s:U\otimes  U\to U\otimes  U$ and the Demazure operator $D=
\frac{1}{1-x_1x_2^{-1}}(1-s)$ on $U\otimes  U=\ZZ[x_1,x_2]$ comprise an ${\bf H}(S_3)$-structure on $U$.
\end{corollary}

\section{Generalization to other groups}
\label{sect:Generalization to other groups}
In this section we generalize the construction of Hecke-Hopf algebras to all groups. Indeed, let $W$ be a group and ${\mathcal S}$ 
be a conjugation-invariant subset of $W\setminus\{1\}$, and let  
$R$ be an integral domain.



For any functions $\chi,\sigma:W\times {\mathcal S}\to R$ 
let $\hat {\bf H}_{\chi,\sigma}(W)$ be an $R$-algebra generated by $W$, as a group, and by $D_s$, $s\in {\mathcal S}$ subject to relations:
\begin{equation}
\label{eq:relations hat H}
wD_sw^{-1}=\chi_{w,s}\cdot D_{wsw^{-1}}+\sigma_{w,s}\cdot (1-wsw^{-1})
\end{equation}
for all $s\in {\mathcal S}$, $w\in W$;
\begin{equation}
\label{eq:relations hat H taft}
{|s| \brack k}_{a_s} D_s(a_sD_s+b_s)(a_s^2D_s+b_s(1+a_s))\cdots (a_s^{k-1}D_s+b_s(1+a_s+\cdots +a_s^{k-2}))=0
\end{equation}
for all  $s\in {\mathcal S}$ 
of finite order $|s|$ and $k=1,\ldots,|s|$, where we abbreviated $a_s:=\chi_{s,s}$, $b_s:=\sigma_{s,s}$ and 
$\displaystyle{{n \brack k}_q=\prod\limits_{i=1}^k \frac{q^{n+1-i}-1}{q^i-1}}\in \ZZ_{\ge 0}[q]$ 
is the $q$-binomial coefficient.


\begin{remark} 
\label{rem:single monic relation D_s}
If $a_s=\chi_{s,s}$ is an primitive $|s|$-th root of unity in $R^\times$, then the relations \eqref{eq:relations hat H taft} simplify:
\begin{equation}
\label{eq:relations hat H taft primitive}
D_s(a_sD_s+b_s)(a_s^2D_s+b_s(1+a_s))\cdots (a_s^{|s|-1}D_s+b_s(1+a_s+\cdots +a_s^{|s|-2}))=0\ .
\end{equation}
Otherwise, if $a_s$ is not an $|s|$-th root of unity and $R$ is a field, then $D_s=0$.
\end{remark}

The following result generalizes Theorem \ref{th:hopf hat W}.

\begin{theorem} 
\label{th:hopf hat W gen} For any group $W$, a conjugation-invariant set ${\mathcal S}\subset W\setminus\{1\}$, and any maps $\chi,\sigma:W\times {\mathcal S}\to R$,  $\hat {\bf H}_{\chi,\sigma}(W)$ is a Hopf algebra with the coproduct $\Delta$, the counit $\varepsilon$, and the antipode anti-automorphism given respectively by (for  $s\in {\mathcal S}$): 
\begin{equation}
\label{eq:hopf H'}
\Delta(s)=s\otimes s,~\Delta(D_s)=D_s\otimes 1+s\otimes D_s,~\varepsilon(s)=1,~\varepsilon(D_s)=0,S(s)=s^{-1},~S(D_s)=-s^{-1}D_s\ .
\end{equation}

\end{theorem}

We prove Theorem  \ref{th:hopf hat W gen} in Section \ref{subsec:proof of  Theorem hopf hat}. 

For any $\chi,\sigma:W\times {\mathcal S}\to R$ denote by $\hat {\bf D}_{\chi,\sigma}(W)$ the $R$-algebra generated by all $D_s$, $s\in {\mathcal S}$ subject to all relations \eqref{eq:relations hat H taft}. By definition, one has an algebra homomorphism $\hat {\bf D}_{\chi,\sigma}(W)\to \hat {\bf H}_{\chi,\sigma}(W)$. This homomorphism is sometimes injective and implies a factorization of $\hat {\bf H}_{\chi,\sigma}(W)$.

\begin{theorem} 
\label{th:ZWcross and free D gen}
In the notation of Theorem \ref{th:hopf hat W gen}, suppose that:

$\bullet$  $\chi$ and $\sigma$ satisfy
\begin{equation}
\label{eq:2 cocycle}
\chi_{w_1w_2,s}=\chi_{w_2,s}\cdot \chi_{w_1,w_2sw_2^{-1}}\in R^\times,~
\sigma_{w_1w_2,s}=\sigma_{w_2,s}+\chi_{w_2,s}\sigma_{w_1,w_2sw_2^{-1}}
\end{equation}
for all $w_1,w_2\in W$, $s\in {\mathcal S}$.
 

$\bullet$ For any $s\in {\mathcal S}$ of finite order and $w\in W$: $\chi_{w,s}^{|s|}=1$ and there exists 
$\kappa_{w,s}\in \ZZ_{\ge 0}$ such that 
\begin{equation}
\label{eq:sigmaws etc}
\sigma_{w,s}=\sigma_{s,s}(1+\chi_{s,s}+\cdots+\chi_{s,s}^{\kappa_{w,s}-1})\ .
\end{equation}

Then  $\hat {\bf H}_{\chi,\sigma}(W)$ factors as $\hat {\bf H}_{\chi,\sigma}(W)=\hat {\bf D}_{\chi,\sigma}(W)\cdot RW$ over $R$ (i.e., 
the multiplication map defines an isomorphism of  $R$-modules
$\hat {\bf D}_{\chi,\sigma}(W)\otimes R W\widetilde \longrightarrow \hat {\bf H}_{\chi,\sigma}(W)$) and is a free $R$-module.

\end{theorem}

We prove Theorem \ref{th:ZWcross and free D gen} in Section \ref{subsec:proof of Theorem ZWcross and free D}.

\begin{remark} Any pair $(\chi,\sigma)$ satisfying \eqref{eq:2 cocycle} defines:

$\bullet$ A $W$-action on  $V=\oplus_{s\in {\mathcal S}} R\cdot D_s$ via  $w(D_s)=\sigma_{w,s}+\chi_{w,s}D_{wsw^{-1}}$ for $w\in W$, 
$s\in {\mathcal S}$ (see also Theorem \ref{th:symmetris Dchisigma(W)}(a) below).

$\bullet$ A function $\gamma\in Hom_R (R W\otimes V, R W)$  given by $\gamma(w\otimes D_s)=\sigma_{w,s}wsw^{-1}$ which is  
a Hochschild $2$-cocycle, i.e., 
$w_1\gamma(w_2\otimes v)w_1^{-1}-\gamma(w_1w_2\otimes v)+\gamma(w_1\otimes w_2(v))=0$ 
for all $w_1,w_2\in W$, $v\in V$ (see also Proposition \ref{pr:mu nu gamma H} with generalization to Hopf algebras).

In particular, for any function $c:W\to R$, the map $\sigma^c:W\times {\mathcal S}\to R$ given by  
$$\sigma^c_{w,s}=\sigma_{w,s}+c_s-\chi_{w,s}c_{wsw^{-1}}$$ 
also satisfies the second condition \eqref{eq:2 cocycle} and thus $\sigma^c$ is cohomological to $\sigma$.

\end{remark}

Denote by $\tilde {\bf D}_{\chi,\sigma}(W)$ the subalgebra of $\hat {\bf H}_{\chi,\sigma}(W)$ generated by all $D_s$, $s\in {\mathcal S}$ 
(by definition, this is a homomorphic image of $\hat {\bf D}_{\chi,\sigma}(W)$ in $\hat {\bf H}_{\chi,\sigma}(W)$) and let
\begin{equation}
\label{eq:KchisigmaW}
{\bf K}_{\chi,\sigma}(W):=\bigcap\limits_{w\in W} w\tilde {\bf D}_{\chi,\sigma}(W)w^{-1}\ .
\end{equation}
 



Denote by ${\bf H}_{\chi,\sigma}(W)$ the quotient algebra of $\hat {\bf H}_{\chi,\sigma}(W)$ by  the ideal generated by ${\bf K}_{\chi,\sigma}(W)\cap Ker~\varepsilon$.

\begin{theorem} 
\label{th:Coideal Hopf gen} In the notation of Theorem \ref{th:hopf hat W gen}, 
suppose that $\hat {\bf H}_{\chi,\sigma}(W)$ is a free $R$-module (e.g., $R$ is a field). 
Then  ${\bf H}_{\chi,\sigma}(W)$ is naturally a Hopf algebra. 

\end{theorem}

We prove Theorem \ref{th:Coideal Hopf gen} in Section \ref{subsec:proof of Theorem Coideal Hopf}. 
We will refer to ${\bf H}_{\chi,\sigma}(W)$  as a  {\it Hopf envelope of $(W,\chi,\sigma)$} (provided that $\hat {\bf H}_{\chi,\sigma}(W)$ is a free $R$-module).

%
%
%
%


Furthermore, in the notation of Theorem \ref{th:hopf hat W gen} denote by ${\bf D}_{\chi,\sigma}(W)$ the quotient of $\tilde {\bf D}_{\chi,\sigma}(W)$ by the ideal generated by 
${\bf K}_{\chi,\sigma}(W)$. By definition, one has an algebra homomorphism  
${\bf D}_{\chi,\sigma}(W)\to {\bf H}_{\chi,\sigma}(W)$. 
Similarly to Theorem \ref{th:ZWcross and free D gen}, this homomorphism is sometimes injective and implies a factorization of  ${\bf H}_{\chi,\sigma}(W)$.


\begin{theorem} 
\label{th:ZWcross and free D gen H}
In the assumptions of Theorem \ref{th:hopf hat W gen}, suppose that 
\begin{equation}
\label{eq:s conj}
\chi_{s,s}~\text{is a primitive $|s|$-th root of unity  $\forall$ $s\in {\mathcal S}$ of finite order $|s|$} \ . 
\end{equation}

Then ${\bf H}_{\chi,\sigma}(W)$ is a Hopf algebra and it factors as ${\bf H}_{\chi,\sigma}(W)={\bf D}_{\chi,\sigma}(W)\cdot RW$.

\end{theorem}

We prove Theorem \ref{th:ZWcross and free D gen H} in Section  \ref{subsec:proof of Theorem Coideal Hopf}. In fact, the lower Hecke-Hopf algebra $\underline {\bf H}(W)$ from Theorem \ref{th:Coideal Hopf} equals ${\bf H}_{\chi,\sigma}(W)$ for a special choice of $\chi,\sigma$ (see Proposition \ref{pr:relations hat H coxeter}) which automatically satisfy \eqref{eq:2 cocycle}, \eqref{eq:sigmaws etc}, and \eqref{eq:s conj}. For some groups $W$, say, complex reflection ones, we may expect an analogue of the Hecke-Hopf algebra ${\bf H}(W)$ as well.

\begin{remark} We believe that classification problem of quadruples $(W,{\mathcal S},\chi,\sigma)$ with any $s\in {\mathcal S}$ of finite order satisfying 
\eqref{eq:2 cocycle}, \eqref{eq:sigmaws etc}, and \eqref{eq:s conj}, is of interest.

\end{remark}

Similarly to Theorem \ref{th:symmetris D(W)}, we can establish some symmetries of ${\bf H}_{\chi,\sigma}(W)$ in general.

\begin{theorem} 
\label{th:symmetris hat D(W) gen}  
In the notation of Theorem \ref{th:hopf hat W gen}, suppose that $\overline{\cdot}$ is an involution on $R$ such that $\overline \chi_{w,s}=\chi_{w,s^{-1}}$, $\overline \sigma_{w,s}=\sigma_{w,s^{-1}}$ for all $w\in W$, $s\in {\mathcal S}$. Then the assignments $\overline w=w^{-1}$, $\overline D_s=D_{s^{-1}}$ for $w\in W$, $s\in {\mathcal S}$ extends to a unique $R$-linear anti-involution of ${\bf H}_{\chi,\sigma}(W)$.



\end{theorem}

The following is a generalization of parts (a) and (b) of Theorem \ref{th:symmetris D(W)}.

\begin{theorem} 
\label{th:symmetris Dchisigma(W)} In the assumptions of Theorem \ref{th:ZWcross and free D gen} suppose also that 
\begin{equation}
\label{eq:w-invariance of sigma}
\sigma_{wsw^{-1},wsw^{-1}}=\sigma_{s,s}
\end{equation}
for all $w\in W$,  $s\in {\mathcal S}$ of finite order. Then: 

(a) Suppose that
\begin{equation}
\label{eq:acyclicity S}
\sigma_{w,s_1}\sigma_{ws_1,s_2}\cdots \sigma_{ws_1\cdots s_{k-1},s_k}=0
\end{equation}
for any $w\in W$, $k\ge 2$, and any  $s_1,\ldots,s_k\in {\mathcal S}$ such that $s_1\cdots s_k=1$.
Then the algebra ${\bf D}_{\chi,\sigma}(W)$  admits the $W$-action by automorphisms via $w(D_s)=\sigma_{w,s}+\chi_{w,s}D_{wsw^{-1}}$ for $w\in W$, $s\in {\mathcal S}$.

(b) Suppose that for a given $s\in {\mathcal S}$ one has
\begin{equation}
\label{eq:alternating acyclicity S}
\sigma_{s^{-1},s_1}\sigma_{s^{-1}s_1,s_2}\cdots \sigma_{s^{-1}s_1\cdots s_{k-1},s_k}=0
\end{equation}
for any  $k\ge 2$ and any $s_1,\ldots,s_k\in {\mathcal S}$ such that $s_1\cdots s_k=s$.
Then  ${\bf D}_{\chi,\sigma}(W)$  admits  an $s^{-1}$-derivation $\partial_s$ (i.e., $\partial_s(xy)=\partial_s(x)y+s^{-1}(x)\partial_s(y)$) such that $\partial_s(D_{s'})=\delta_{s,s'}\sigma_{s^{-1},s}$, $s,s'\in {\mathcal S}$.

\end{theorem}
We prove Theorem \ref{th:symmetris Dchisigma(W)}  in Section \ref{subsec:proof of Theorem symmetris D(W)}. In fact, the algebra $\hat {\bf D}_{\chi,\sigma}(W)$ has these symmetries if \eqref{eq:w-invariance of sigma} holds (Proposition \ref{pr:symmetris hat D(W) gen}(c)), however, \eqref{eq:acyclicity S} is needed for ${\bf K}_{\chi,\sigma}$ to be invariant the $W$-action and \eqref{eq:alternating acyclicity S} is needed for ${\bf K}_{\chi,\sigma}(W)$ to be in the kernel of each $\partial_s$.

\begin{remark} If $R$ is a field, then the condition \eqref{eq:acyclicity S} implies that the transitive closure of the relation $ws\prec w$ iff $\sigma_{w,s}\ne 0$ is a partial order on $W$, which we can think of as a ``generalized Bruhat order." This is justified by Proposition \ref{pr:chi sigma satisfies cocycle condition}(b) which implies that if $W$ is a Coxeter group and ${\mathcal S}$ is the set of all reflections in $W$, then \eqref{eq:acyclicity S} holds and the partial order coincides with the strong Bruhat order on $W$. 
It is also easy to see that the condition \eqref{eq:alternating acyclicity S} holds for each simple reflection in any Coxeter group. So we can think of all $s$ satisfying \eqref{eq:alternating acyclicity S} as ``generalized simple reflections."
\end{remark}

\begin{conjecture} 
\label{conj:symmetris Dchisigma(W)} In the assumptions of Theorem \ref{th:ZWcross and free D gen} suppose that  $\theta$ is an $R$-linear automorphism of $R W$ such that $\theta(w)\in R^\times \cdot w$ for $w\in W$ and  
$\theta(s)=\chi_{s,s}\cdot s$ for  $s\in {\mathcal S}$. Then $\theta$ uniquely extends to 
an algebra automorphism of  ${\bf H}_{\chi,\sigma}(W)$   such that  $\theta(D_s)= \sigma_{s,s}+\chi_{s,s}D_s$  for $s\in {\mathcal S}$. 
 
\end{conjecture}

If one replaces ${\bf H}_{\chi,\sigma}(W)$ with $\hat {\bf H}_{\chi,\sigma}(W)$, the assertion of the conjecture is true 
(Proposition \ref{pr:symmetris hat D(W) gen}(b)). However, unlike that in Theorem \ref{th:symmetris D(W)}(c), the question whether  $\theta$ preserves ${\bf K}_{\chi,\sigma}(W)$ is still, open, which the conjecture, in fact, asserts. 


The following  is a natural consequence of the above results and constructions. 

In the situation of Theorem \ref{th:ZWcross and free D gen} to a subset ${\mathcal S}_0\subset {\mathcal S}$ and a function    
${\bf q}:{\mathcal S}_0\to R$ ($s\mapsto q_s$) we assign a subalgebra 
$H_{\bf q}(W,{\mathcal S}_0)$ of ${\bf H}_{\chi,\sigma}(W)\otimes \kk$ generated by all 
$s+(1-q_s)D_s$, $s\in {\mathcal S}_0$. By the very construction, $H_{\bf q}(W,{\mathcal S}_0)$ is a left coideal subalgebra in ${\bf H}_{\chi,\sigma}(W)$.

We say that $H_{\bf q}(W,{\mathcal S}_0)$ is a {\it generalized Hecke algebra} if it is a deformation of $R W_0$, 
where $W_0$ is the subgroup of $W$ generated by ${\mathcal S}_0$, or, more precisely, the restriction of the $R$-linear projection 
$\pi: {\bf H}_{\chi,\sigma}(W)\to RW$ given by $\pi(xw)=w$ for $x\in {\bf D}_{\chi,\sigma}$, $w\in W$ to 
$H_{\bf q}(W,{\mathcal S}_0)$, is an isomorphism of $R$-modules $H_{\bf q}(W,{\mathcal S}_0)\widetilde \to RW_0$.

\begin{problem} Classify generalized Hecke algebras. 

\end{problem}

In Section \ref{sec:Taft algebras} we solve the problem for finite cyclic groups $W$ via generalized Taft algebras. 

It would be interesting to compare our constructions with  the Broue-Malle-Rouquier Hecke algebras (\cite{BMR}) attached to complex reflection groups.

\section{Generalization to Hopf algebras}

\label{sect:Generalization to Hopf algebras}
In this section we will extend our constructions from algebras ${\bf H}_{\chi,\sigma}(W)$ to Hopf algebras ${\bf H}$ over a commutative ring $R$ containing a Hopf subalgebra $H$ and a left coideal subalgebra ${\bf D}$.

Recall that, given a coalgebra ${\bf H}$ over a commutative ring $R$, an $R$-submodule ${\bf K}$ is called a left (resp. right) coideal if 
$\Delta({\bf K})\subset {\bf H}\otimes {\bf K}$ (resp. $\Delta({\bf K})\subset {\bf K}\otimes {\bf H}$).
 
The following properties of left (and right) coideals are, apparently, well-known.

\begin{proposition} 
\label{pr:sum intersection of coideals} For any coalgebra ${\bf H}$ over  $R$, one has:

(a) Sum  of left coideals is also a left coideal.

(b) If ${\bf H}$ is a free $R$-module, then the intersection of left coideals is also a left coideal.
\end{proposition}

\begin{proof} Part (a) is immediate. 

To prove (b), we need the following obvious (and, apparently, well-known) fact. 
\begin{lemma} 
\label{le:intersection tensor products}
Let ${\bf A}$ be a free module over a commutative ring $R$ and let ${\bf B}$ be an $R$-module and 
${\bf B}_{\bf i}$, ${\bf i}\in {\bf I}$ be a family of $R$-submodules in ${\bf B}$. Then
$\bigcap\limits_{{\bf i}\in {\bf I}} \left({\bf A}\otimes {\bf B}_{\bf i}\right)= {\bf A}\otimes \left(\bigcap\limits_{{\bf i}\in {\bf I}}{\bf B}_{\bf i}\right)$.
\end{lemma}

Indeed, if ${\bf B}_{\bf i}$, ${\bf i}\in {\bf I}$ is a family of left coideals in ${\bf H}$, then 
$$\Delta(\bigcap\limits_{{\bf i}\in {\bf I}} {\bf B}_{\bf i})\subset \bigcap\limits_{{\bf i}\in {\bf I}} \Delta({\bf B}_{\bf i})\subset \bigcap\limits_{{\bf i}\in {\bf I}} {\bf H}\otimes {\bf B}_{\bf i}={\bf H}\otimes \left(\bigcap\limits_{{\bf i}\in {\bf I}}{\bf B}_{\bf i}\right)\ .$$
by  Lemma \ref{le:intersection tensor products} taken with ${\bf A}={\bf B}={\bf H}$.
This proves (b).

The proposition is proved.
\end{proof}

Let $H$ be a Hopf algebra over $R$ and let ${\bf H}$ be an $H$-module algebra (we denote the action by $h\otimes x\mapsto h(x)$). For any $R$-subalgebra ${\bf D}$ of ${\bf H}$ define
\begin{equation}
\label{eq:general K defined}
{\bf K}(H,{\bf D}):=\{x\in {\bf D}\,|\,H(x)\subset {\bf D}\}\ .
\end{equation}

\begin{lemma} ${\bf K}(H,{\bf D})$ is a subalgebra of ${\bf H}$ invariant under the $H$-action.

\end{lemma}

\begin{proof} Indeed, for $x,y\in {\bf K}(H,{\bf D})$  we have
$h(xy)=h_{(1)}(x)\cdot h_{(2)}(y)\in {\bf D}$
for all $h\in H$. Hence $xy\in {\bf K}(H,{\bf D})$ and the first assertion is proved.

Furthermore, given  $x\in {\bf K}(H,{\bf D})$, $h\in H$ we have
$h'(h(x))=(h'h)(x)\in {\bf D}$ 
for all $h'\in H$, therefore, $h(x)\in {\bf K}(H,{\bf D})$ for all $x\in {\bf K}(H,{\bf D})$, $h\in H$. 
This proves the second assertion.

The lemma is proved.
\end{proof}


The following is immediate.

\begin{lemma} 
\label{le:action conjugation}
Suppose that $H$ is a Hopf algebra over $R$ and also a subalgebra of an $R$-algebra ${\bf H}$.
Then the assignments $h\act x:=h_{(1)}\cdot x\cdot S(h_{(2)})$, $h\in H$, $x\in {\bf H}$, turn ${\bf H}$ into an $H$-module algebra.

\end{lemma}

Replacing, if necessary, an $H$-module algebra ${\bf H}$ with the cross product $\tilde {\bf H}={\bf H}\rtimes H$, we see that Lemma \ref{le:action conjugation} is applicable to $\tilde {\bf H}$. 

%
%
%
%


In the following result, we will use the action from Lemma \ref{le:action conjugation} for constructing new Hopf algebras.

\begin{theorem} 
\label{th:general hopf ideal}
Let ${\bf H}$ be a Hopf algebra over $R$, $H$ be a Hopf subalgebra of ${\bf H}$, and ${\bf D}$ be a left coideal subalgebra of ${\bf H}$.
Suppose that ${\bf H}$ is free as an $R$-module. Then the ideal ${\bf J}(H,{\bf D})$ of ${\bf H}$ generated by  ${\bf K}(H,{\bf D})\cap Ker~\varepsilon$ is a Hopf ideal, hence $\underline {\bf H}:={\bf H}/{\bf J}(H,{\bf D})$ is naturally a Hopf algebra.
\end{theorem}

\begin{proof} We need the following result.

 
\begin{proposition} 
\label{pr:general K} In the assumptions of Theorem \ref{th:general hopf ideal},  
${\bf K}(H,{\bf D})$ is a left coideal subalgebra of ${\bf H}$.
\end{proposition}

\begin{proof} For an $R$-module ${\bf A}$ and an $H$-module  ${\bf H}$ define the action of $H$ on ${\bf A}\otimes {\bf H}$  
by $h\act (x\otimes y)=x\otimes h(y)$ for $x\in {\bf A},y \in {\bf H}$.

We need the following result.

\begin{lemma} Let ${\bf H}$ be a Hopf algebra over $R$ and let $H$ be a Hopf subalgebra of ${\bf H}$. 
Then $h\act \Delta(x)=(S(h_{(1)})\otimes 1)\cdot \Delta(h_{(2)}\act x)\cdot (h_{(3)}\otimes 1)$ (here  $\act$ is the adjoint action  from Lemma \ref{le:action conjugation}) 
for $h\in H$, $x\in {\bf H}$, with the Sweedler notation 
$(\Delta\otimes 1)\circ \Delta(h)=(1\otimes \Delta)\circ \Delta(h)=h_{(1)}\otimes h_{(2)}\otimes h_{(3)}$.
\end{lemma}

\begin{proof}
Indeed, $(S(h_{(1)})\otimes 1)\cdot \Delta(h_{(2)}\act x)\cdot (h_{(3)}\otimes 1)=
(S(h_{(1)})\otimes 1)\cdot \Delta(h_{(2)})\Delta(x)\Delta(S(h_{(3)}))\cdot (h_{(4)}\otimes 1)$
$=
(1\otimes h_{(1)})\cdot \Delta(x)\cdot (1\otimes S(h_{(2)}))=h\act \Delta(x)$
because 
$(S(h_{(1)})\otimes 1)\cdot \Delta(h_{(2)})=
S(h_{(1)})\cdot h_{(2)}\otimes h_{(3)}=1\otimes h$ and 
$\Delta(S(h_{(1)}))\cdot (h_{(2)}\otimes 1)=S(h_{(2)})\cdot h_{(3)}\otimes S(h_{(1)})
=1\otimes S(h)$.

The lemma is proved.
\end{proof}

This proves that, in the assumptions of Theorem \ref{th:general hopf ideal},  we have $H\act \Delta(x)\subset {\bf H}\otimes {\bf D}$ for all $x\in {\bf K}$. To finish the proof of Proposition \ref{pr:general K},  we need the following result.

\begin{lemma} For any free $R$-module ${\bf A}$  one has in the assumptions of \eqref{eq:general K defined}:
\begin{equation}
\label{eq:tensor set K}
\{z\in {\bf A}\otimes {\bf D}\,|\,H\act z\subset {\bf A}\otimes {\bf D}\}= {\bf A}\otimes {\bf K}(H,{\bf D})\ .
\end{equation}
\end{lemma}

\begin{proof} 
Indeed, let ${\bf B}$ be an $R$-basis of ${\bf A}$. Write each $z\in {\bf A}\otimes {\bf H}$ as
$$z=\sum_{b\in {\bf B}} b\otimes x_b$$
where all $x_b\in {\bf D}$ and all but finitely many of them are $0$. Then
$$h\act\limits z=\sum_{b\in {\bf B}} b\otimes h(x_b)\ .$$
In particular, if $h\act\limits z\in {\bf H}\otimes {\bf D}$ for some $h\in H$, then $h(x_b)\in {\bf D}$ for all $b\in {\bf B}$. Therefore, $H\act\limits z\subset {\bf H}\otimes {\bf D}$ implies that $x_b\in {\bf K}(H,{\bf D})$ for $b\in {\bf B}$.

This proves the inclusion of the left hand side of \eqref{eq:tensor set K} into the right hand side. The opposite inclusion is obvious.
 
The lemma is proved.
\end{proof}

Therefore, Proposition \ref{pr:general K} is proved.
\end{proof}

We need the following (probably, well-known) general result.

\begin{proposition}  
\label{pr:from coideal to hopf ideal}
Let ${\bf H}$ be a Hopf algebra over  $R$ and let ${\bf K}\subset {\bf H}$ be a left or right coideal.
Then the ideal ${\bf J}$ generated by ${\bf K}^+:={\bf K}\cap Ker~\varepsilon$ is a Hopf ideal, i.e., $\Delta({\bf J})\subset {\bf H}\otimes {\bf J}+{\bf J}\otimes {\bf H},~S({\bf J})\subset {\bf J}$.

\end{proposition}

\begin{proof} We will prove the assertion when ${\bf K}$ is a left coideal (for the right ones the proof is identical). 
We need the following well-known fact.

\begin{lemma} 
\label{le:short coproduct}
For any coalgebra ${\bf H}$ one has
$\Delta(h)-h\otimes 1\in {\bf H}\otimes Ker~\varepsilon$
for all $h\in {\bf H}$.
\end{lemma}

Indeed, taking into account that $\Delta({\bf K})\subset {\bf H}\otimes {\bf K}^+\oplus {\bf H}\otimes 1$ for any left coideal ${\bf K}\subset {\bf H}$,  where we abbreviated ${\bf K}^+:={\bf K}\cap Ker~\varepsilon$, 
Lemma \ref{le:short coproduct} guarantees that 
\begin{equation}
\label{eq:short coideal}
\Delta(h)-h\otimes 1\in {\bf H}\otimes {\bf K}^+
\end{equation}
for all $h\in {\bf K}^+$. Therefore, 
$\Delta({\bf K}^+)\subset {\bf H}\otimes {\bf K}^++{\bf K}^+\otimes 1$. In turn, this implies that: 
$$\Delta({\bf J})\subset ({\bf H}\otimes {\bf H}) \cdot \Delta({\bf K}^+)\cdot ({\bf H}\otimes {\bf H})\subset ({\bf H}\otimes {\bf H}) \cdot ({\bf H}\otimes {\bf K}^++{\bf K}^+\otimes 1)\cdot ({\bf H}\otimes {\bf H})\subset {\bf H}\otimes {\bf J} + {\bf J}\otimes {\bf H} \ ,$$
i.e., ${\bf J}$ is a bi-ideal.

Furthermore, applying ${\bf m}\circ (S\otimes 1)$ to \eqref{eq:short coideal}  and using the  property of the antipode $m\circ (S\otimes 1)\circ \Delta=\varepsilon$,   
we obtain 
$\varepsilon(h)-S(h)\in S({\bf H})\cdot{\bf K}^+$ for all $h\in {\bf K}^+$, therefore, $S({\bf K}^+)\subset {\bf H}\cdot{\bf K}^+$.
Hence
$$S({\bf J})={\bf H}\cdot S({\bf K}^+)\cdot {\bf H}\subset {\bf H}\cdot ({\bf H}\cdot{\bf K}^+)\cdot {\bf H}={\bf J}\ .$$

The proposition is proved.
\end{proof}

Clearly, the assertion of Theorem \ref{th:general hopf ideal} follows from Propositions \ref{pr:general K} and \ref{pr:from coideal to hopf ideal}.

Theorem \ref{th:general hopf ideal} is proved.
\end{proof}

Let $H$ be a Hopf algebra over $R$, $V$ be an $T(H)$-module (i.e., an $R$-linear map $H\otimes V\to V$),  for an $R$-bilinear map $\gamma:H\times V\to H$ satisfying:
\begin{equation}
\label{eq:unit gamma}
\gamma(1,v)=0
\end{equation}
for all  $v\in V$, let  ${\bf H}_\gamma$ be an algebra generated by $H$ (viewed as an algebra) and $V$ subject to relations
\begin{equation}
\label{eq:deformed cross  product hopf}
h_{(1)}\cdot v\cdot S(h_{(2)})=h(v)+\gamma(h,v)
\end{equation}
for all $h\in H$, $v\in V$.
Using the property of the antipode in $H$, it is easy to see that relations \eqref{eq:deformed cross  product hopf} are equivalent to:
$$hv=h_{(1)}(v)\cdot h_{(2)}+\beta(h\otimes v)$$
for all $h\in H$, $v\in V$ where $\beta:H\otimes V\to H$ is given by 
$\beta(h\otimes v)=\gamma(h_{(1)},v)h_{(2)}$.
This implies that ${\bf H}_\gamma=A_\mu$ in the notation of \eqref{eq:Agamma} and of Corollary \ref{cor:mu nu gamma}, 
where $\beta$ is as above and 
$\nu:H\otimes V\to V\otimes H$ is given by
$\nu(h\otimes v)=h_{(1)}(v)\otimes  h_{(2)}$ for all $h\in H$, $v\in V$.

If $\gamma=0$ and $T(H)$-action on $V$ factors through an $H$-action, then $H_\gamma=T(V)\rtimes H$, the cross product. Using Corollary \ref{cor:mu nu gamma}, we obtain a criterion  for  factorization  of ${\bf H}_\gamma$ into $T(V)$ and $H$.

\begin{proposition} 
\label{pr:mu nu gamma H} Let  $\gamma:H\times V\to H$ be an $R$-bilinear map satisfying \eqref{eq:unit gamma}. Then ${\bf H}_\gamma$ factors as  ${\bf H}_\gamma=T(V)\cdot H$  (i.e., 
the multiplication map defines an isomorphism of $R$-modules
$T(V)\otimes H\widetilde \longrightarrow {\bf H}_\gamma$) as an $R$-module iff $V$ is an $H$-module 
and $\gamma$ satisfies for all $h,h'\in H$, $v\in V$:
\begin{equation}
\label{eq:gamma property H}
\gamma(hh',v)=\gamma(h,h'(v))+h\act \gamma(h',v)
\end{equation}
where $\act$ denotes the adjoint action of the Hopf algebra $H$ on itself (as in Lemma \ref{le:action conjugation}).
\end{proposition} 

\begin{proof} Let us identify both conditions of Corollary \ref{cor:mu nu gamma} with $B=H$ and $\nu$ and $\gamma$ as above. Namely, taking into account that 
$\nu\circ ({\bf m}_H\otimes Id_V)(h\otimes h'\otimes v)=\nu(hh'\otimes v)=(hh')_{(1)}(v)\otimes  (hh')_{(2)}$,
$$(Id_H\otimes {\bf m}_H)\circ (\nu\otimes Id_H)\circ (Id_H\otimes \nu)((h\otimes h'\otimes v))=(hh')_{(1)}(v)\otimes  (hh')_{(2)}\ ,$$
the first condition of Corollary \ref{cor:mu nu gamma} reads
\begin{equation}
\label{eq:nu property H}
(hh')_{(1)}(v)\otimes  (hh')_{(2)}=h_{(1)}(h'_{(1)}(v))\otimes h_{(2)}h'_{(2)}
\end{equation}
for all $h,h'\in T(H)$, $v\in V$.


Furthermore, taking into account that $\beta\circ ({\bf m}_H\otimes Id_V)(h\otimes h'\otimes v)=\beta(hh'\otimes v)$,
$${\bf m}_H\circ (Id_H\otimes \beta)(h\otimes h'\otimes v)=h\beta(h'\otimes h)\ ,$$
$${\bf m}_H\circ (\beta\otimes Id_H)\circ (Id_H\otimes \nu)(h\otimes h'\otimes v)
={\bf m}_H\circ (\beta\otimes Id_H)(h\otimes h'_{(1)}(v)\otimes h'_{(2)})
=\beta(h\otimes h'_{(1)}(v))h'_{(2)}\ ,$$
the second condition of Corollary \ref{cor:mu nu gamma} reads
\begin{equation}
\label{eq:beta property H}
\beta(hh'\otimes v)=h\beta(h'\otimes v)+\beta(h\otimes h'_{(1)}(v))h'_{(2)}
\end{equation}
for all $h,h'\in H$, $v\in V$.

%


Let us show that \eqref{eq:nu property H} is equivalent to 
\begin{equation}
\label{eq:action H}
(hh')(v)=h(h'(v))
\end{equation}
for $h\in H$, $v\in V$.

Indeed, multiplying both sides of \eqref{eq:nu property H} by $S((hh')_{(3)})=S(h'_{(3)})S(h_{(3)})$ on the right we obtain 
\eqref{eq:action H} after cancellations. Conversely, by acting with the first factor of 
$\Delta(hh')=(hh')_{(1)}\otimes (hh')_{(2)}=h_{(1)}h'_{(1)}\otimes h_{(2)}h'_{(2)}$ on $v$ and using \eqref{eq:action H},  we obtain \eqref{eq:nu property H}. Thus, \eqref{eq:action H} and \eqref{eq:unit gamma} assert that $V$ is an $H$-module (and vice versa).

Finally, let us show that \eqref{eq:beta property H} is equivalent to \eqref{eq:gamma property H}.

1. \eqref{eq:beta property H} $=>$ \eqref{eq:gamma property H}. Since $\beta(h\otimes v)=\gamma(h_{(1)},v)h_{(2)}$, 
\eqref{eq:beta property H} becomes:
$$\gamma((hh')_{(1)},v)(hh')_{(2)}=h\gamma(h'_{(1)},v)h'_{(2)}+\gamma(h_{(1)},h'_{(1)}(v))h_{(2)}h'_{(2)}\ .$$
Multiplying both sides  by $S((hh')_{(3)})=S(h'_{(3)})S(h_{(3)})$ on the right, we obtain after cancellations
$$\gamma(hh',v)=h_{(1)}\gamma(h'_{(1)},v)\varepsilon(h_{(2)})h'_{(2)}S(h'_{(3)})S(h_{(3)})+\gamma(h_{(1)},h'_{(1)}(v))h_{(2)}h'_{(2)}S(h'_{(3)})S(h_{(3)})$$
$=h_{(1)}\gamma(h',v)S(h_{(2)})+\gamma(h,h'(v))$, 
which coincides with \eqref{eq:gamma property H}. 

2. \eqref{eq:gamma property H} $=>$ \eqref{eq:beta property H}. Since $\gamma(h,v)=\beta(h_{(1)}\otimes v)S(h_{(2)})$, 
\eqref{eq:gamma property H} becomes:
$$\beta((hh')_{(1)}\otimes v)S((hh')_{(2)})=\beta(h_{(1)}\otimes h'(v))S(h_{(2)})+h\act (\beta(h'_{(1)}\otimes v)S(h'_{(2)}))\ .$$
Multiplying both sides  by $(hh')_{(3)}=h_{(3)}h'_{(3)}$, we obtain after cancellations
$$\beta(hh'\otimes v)=\beta(h_{(1)}\otimes h'(v))S(h_{(2)})h_{(3)}h'_{(3)}+h_{(1)}\cdot \beta(h'_{(1)}\otimes v)S(h'_{(2)})\cdot S(h_{(2)}) h_{(3)}h'_{(3)}$$
$=\beta(h\otimes h'(v))+h\beta(h'_{(1)}\otimes v)S(h'_{(2)})$, 
which coincides with \eqref{eq:beta property H}. 

The proposition is proved.
\end{proof}

It is well-known that if $\gamma=0$, then ${\bf H}_\gamma$ is a Hopf algebra. Now we provide sufficient conditions on $\gamma$ (one can show that they are also necessary) for ${\bf H}_\gamma$ to be a Hopf algebra.

\begin{proposition} 
\label{pr:deformed nichols}
Let $H$ be a Hopf algebra over $R$,  $V$ be a $T(H)$-module, and $\gamma:H\times V\to H$ be an $R$-bilinear map satisfying \eqref{eq:unit gamma}. 
Suppose that: 

$\bullet$  $V$ 
has 
an $H$-coaction $\delta:V\to H\otimes V$ ($\delta(v)=v^{(-1)}\otimes v^{(0)}$ in a Sweedler-like notation) such that for all $v\in V$, $h\in H$ the Yetter-Drinfeld condition (see e.g., \cite[Section 1.2]{AS}) holds:
\begin{equation}
\label{eq:YD compatibility}
\delta(h(v))=h_{(1)}v^{(-1)}S(h_{(3)})\otimes h_{(2)}(v^{(0)})\ .
\end{equation}

$\bullet$ $\Delta(\gamma(h,v))=\gamma(h,v)\otimes 1+h_{(1)}v^{(-1)}S(h_{(3)})\otimes \gamma(h_{(2)},v^{(0)})$
and $\varepsilon(\gamma(h,v))=0$ for $v\in V$, $h\in H$.

\noindent Then ${\bf H}_\gamma$ is a Hopf algebra with the coproduct, counit, and the antipode extending those in $H$ and determined by (for $h\in H$, $v\in V$):
$$\Delta(v)=v\otimes 1+\delta(v)=v\otimes 1 +v^{(-1)}\otimes v^{(0)},~\varepsilon(v)=0,~ S(v)=-S(v^{(-1)})v^{(0)}$$

\end{proposition}  

\begin{proof} We need the following general result.
\begin{lemma}
\label{le:free Hopf} Let $H$ be a Hopf algebra over $R$
and let  $V$ be a left comodule over $H$ (i.e., one has a co-associative and co-unital linear map $\delta:V\to H\otimes V$).
Then the free product  of $R$-algebras ${\bf H}:=H*T(V)$  is  a Hopf algebra over $R$ with the coproduct, 
counit, and the antipode extending those on $H$ and determined by (for $h\in H$, $v\in V$):
$$\Delta(v)=v\otimes 1+\delta(v)=v\otimes 1 +v^{(-1)}\otimes v^{(0)},~\varepsilon(v)=0,~ S(v)=-S(v^{(-1)})v^{(0)}\ .$$
\end{lemma}

\begin{proof} Indeed,  each element $x\in {\bf H}$ can be written as sum of elements of the form:
$$x=h_0v_1h_1\cdots v_kh_k\ ,$$
where $h_0,h_1,\ldots,h_k\in H$, $v_1,\ldots,v_k\in V$, $k\ge 0$ (with the convention $x=h_0$ if $k=0$).
By setting 

$x\mapsto  \Delta(x)=\Delta(h_0)(v_1\otimes 1+\delta(v_1))\Delta(h_1)\cdots (v_k\otimes 1+\delta(v_k))\Delta(h_k),~x\mapsto  \varepsilon(x)=
\begin{cases} 
\varepsilon(h_0) & \text{if $k=0$}\\
0 & \text{if $k>0$}\\
\end{cases},
$
$$x\mapsto  S(x)=S(h_k)(-S(v_k^{(-1)})v_k^{(0)})S(h_1)\cdots (-S(v_1^{(-1)})v_1^{(0)})S(h_k)$$
one has  well-defined $R$-linear maps $ \Delta:{\bf H}\to {\bf H}\otimes {\bf H}$,  $ \varepsilon:{\bf H}\to R$, and $  S:{\bf H}\to {\bf H}$, respectively. 

Clearly, $\Delta$ is a homomorphism of algebras. Therefore, it suffices to verify the remaining compatibility conditions only on generators $v\in V$. Indeed:
$$(m\circ (\varepsilon\otimes 1)\circ \Delta)(v)=(m\circ (\varepsilon\otimes 1))(v\otimes 1+\delta(v))=(m\circ (\varepsilon\otimes 1)\circ \delta)(v)=v\ ,$$
$$(m\circ (1 \otimes \varepsilon)\circ \Delta)(v)=(m\circ (1 \otimes \varepsilon))(v\otimes 1+\delta(v))=m(v\otimes 1)=v \ ,$$
$$(m\circ (S\otimes 1)\circ \Delta)(v)=(m\circ (S\otimes 1)(v\otimes 1+v^{(-1)}\otimes v^{(0)})=S(v)+S(v^{(-1)})v^{(0)}=0=\varepsilon(v)\ ,$$ 
$$(m\circ (1\otimes S)\circ \Delta)(v)=(m\circ (1\otimes S)(v\otimes 1+v^{(-1)}\otimes v^{(0)})=v+v^{(-1)}S(v^{(0)})$$
$$=v-v^{(-2)}S(v^{(-1)})v^{(0)}=v-\varepsilon(v^{(-1)})v^{(0)}=0=\varepsilon(v)\ .$$

This finished the proof of the lemma.
\end{proof} 

Furthermore, let ${\bf K}_\gamma$ be the $R$-submodule of ${\bf H}=H*T(V)$ generated by $1$ and
$$\delta_{h,v}:=h_{(1)}\cdot v\cdot S(h_{(2)})-h(v)-\gamma(h,v)$$ for all $h\in H$, $v\in V$.

\begin{lemma}  
\label{le:Kgamma}
In the assumptions of Proposition \ref{pr:deformed nichols}, $\Delta({\bf K}_\gamma)\subset H\otimes {\bf K}_\gamma$, in particular, ${\bf K}_\gamma$ is a left coideal in ${\bf H}=H*T(V)$.
 
\end{lemma}

\begin{proof} 
First, prove that
\begin{equation}
\label{eq:bicomodule over H}
\Delta(\delta_{h,v})=\delta_{h,v}\otimes 1+h_{(1)}v^{(-1)}S(h_{(3)})\otimes \delta_{h_{(2)},v^{(0)}}
\end{equation}
for all $h\in H$, $v\in V$. Indeed, 
$$\Delta(\delta_{h,v})=\Delta(h_{(1)})\cdot \Delta(v)\cdot \Delta(S(h_{(2)}))-\Delta(h(v))-\Delta(\gamma(h,v))$$
$$=\Delta(h_{(1)})\cdot (v\otimes 1+\delta(v))\cdot\Delta(S(h_{(2)}))-h(v)\otimes 1-\delta(h(v))-\Delta(\gamma(h,v))$$
$$=(\delta_{h,v}+\gamma(h,v))\otimes 1+\Delta(h_{(1)})\cdot \delta(v)\cdot\Delta(S(h_{(2)}))-\delta(h(v))-\Delta(\gamma(h,v))$$
$$=\delta_{h,v}\otimes 1+\Delta(h_{(1)})\cdot \delta(v)\cdot\Delta(S(h_{(2)}))-\delta(h(v))-h_{(1)}v^{(-1)}S(h_{(3)})\otimes \gamma(h_{(2)},v^{(0)})$$
$$=\delta_{h,v}\otimes 1+h_{(1)}v^{(-1)}S(h_{(3)})\otimes \delta_{h_{(2)},v^{(0)}}\ ,$$
where we used sequentially:

(1) The fact that $\Delta(h_{(1)})\cdot (v\otimes 1)\cdot\Delta(S(h_{(2)}))=h_{(1)}\cdot v\cdot S(h_{(4)})\otimes h_{(2)}S(h_{(3)})$
$$=h_{(1)}\cdot v\cdot S(h_{(3)})\otimes \varepsilon(h_{(2)})=h_{(1)}\cdot v\cdot S(h_{(2)})\otimes 1=(\delta_{h,v}+h(v)+\gamma(h,v))\otimes 1\ .$$

(2) The second assumption of Proposition \ref{pr:deformed nichols}.
%

(3) The Yetter-Drinfeld condition \eqref{eq:YD compatibility} in the form $\Delta(h_{(1)})\cdot \delta(v)\cdot\Delta(S(h_{(2)}))-\delta(h(v))$
$$=h_{(1)}v^{(-1)}S(h_{(4)})\otimes h_{(2)}\cdot v^{(-1)}\cdot S(h_{(3)})-h_{(1)}v^{(-1)}S(h_{(3)})\otimes h_{(2)}(v^{(0)})$$
$$=h_{(1)}v^{(-1)}S(h_{(3)})\otimes (\delta_{h_{(2)},v^{(0)}}+\gamma(h_{(2)},v^{(0)}))$$
This proves \eqref{eq:bicomodule over H}. 
%
%

The lemma is proved.
\end{proof}

Note that ${\bf K}_\gamma^+:={\bf K}_\gamma\cap Ker~\varepsilon$ is the $R$-submodule of ${\bf H}=H*T(V)$ generated by 
$\delta_{h,v}$, 
$h\in H$, $v\in V$.
In view of Proposition \ref{pr:from coideal to hopf ideal}, this and Lemma \ref{le:Kgamma} guarantee that the ideal ${\bf J}_\gamma$ generated by $\delta_{h,v}$, $h\in H$, $v\in V$, is a Hopf ideal in ${\bf H}$. 
Therefore, $\underline {\bf H}={\bf H}/{\bf J}_{\gamma}$ is a Hopf algebra.

The proposition is proved.
\end{proof}

We conclude the section with some general facts which we will use frequently. 

\begin{lemma} 
\label{le:general factored quotient H}
Let ${\bf H}$ be an $R$-algebra, and $H$, ${\bf D}$  subalgebras of ${\bf H}$ such that ${\bf H}$ factors as ${\bf H}={\bf D}\cdot H$ over $R$ (i.e., 
the multiplication map defines an isomorphism of  $R$-modules
${\bf D}\otimes H\widetilde \longrightarrow {\bf H}$). Let ${\bf K}\subset  {\bf D}$ be an $R$-submodule such that 
$H\cdot {\bf K}\subset {\bf K}\cdot H$. Then the ideal ${\bf J}_{\bf K}$ of $ {\bf H}$ generated by ${\bf K}$ factors as ${\bf I}_{\bf K}\cdot H$, where ${\bf I}_{\bf K}$ is the ideal of $ {\bf D}$ generated by ${\bf K}$ and the quotient algebra $\underline {\bf H}={\bf H}/{\bf J}_{\bf K}$ factors as $\underline {\bf H}=\underline {\bf D}\cdot H$, where $\underline {\bf D}={\bf D}/{\bf I}_{\bf K}$. 
\end{lemma}

\begin{proof} Indeed, ${\bf J}_{\bf K}={\bf D}\cdot H\cdot {\bf K}\cdot {\bf D}\cdot H\subset {\bf D}\cdot {\bf K}\cdot H\cdot {\bf D}\cdot H={\bf D}\cdot {\bf K}\cdot  {\bf D}\cdot H={\bf I}_{\bf K}\cdot H$
(because ${\bf I}_{\bf K}= {\bf D}\cdot {\bf K}\cdot {\bf D}$). The opposite inclusion is obvious, 
therefore, ${\bf J}_{\bf K}={\bf I}_{\bf K}\cdot H$.

Finally,
$\underline {\bf H}={\bf H}/{\bf J}_{\bf K}=  ({\bf D}\cdot H)/({\bf I}_{\bf K}\cdot H)=
({\bf D}/{\bf I}_{\bf K})\cdot H=\underline {\bf D}\cdot H$
as an $R$-module.

The lemma is proved.
\end{proof}

In some cases, we can describe ${\bf K}(H,{\bf D})$ explicitly.

\begin{lemma} 
 \label{le:factored derivatives}
Let $W$ be a group. Suppose that ${\bf H}$ is an $R$-algebra which factors as 
${\bf H}={\bf D}\cdot R W$ over $R$,  where ${\bf D}$ is a subalgebra of ${\bf H}$. Then, in the notation of Proposition \ref{pr:factored derivatives}, one has (where the $R W$-action on ${\bf H}$ is given by conjugation):
${\bf K}(RW,{\bf D})=\bigcap\limits_{w,w'\in W:w\ne w'} Ker~\partial_{w,w'}$.
Furthermore, $\partial_{w,w}(x)=wxw^{-1}$ for all $x\in {\bf D}$.

\end{lemma}

\begin{proof} Indeed, writing \eqref{eq:factored derivatives} in the form:
$w x w^{-1}=\sum_{w,w'\in W} \partial_{w,w'}(x)w'w^{-1}$
for $w\in W$, $x\in {\bf D}$, 
we see that $w x w^{-1}\in {\bf D}$ iff $\partial_{w,w'}(x)=0$ for all $w'\ne w$, in which case $w x w^{-1}=\partial_{w,w}(x)$.

The lemma is proved.
\end{proof}

\section{Generalized Nichols algebras and symmetries of Hecke-Hopf algebras}
\label{sec:bosonized}

Let $W$ be a monoid and let ${\mathcal R}\subset W\times W$ be a preorder on $W$ such that $(h,1)\in {\mathcal R}$ iff $h=1$.
We say that $W$ is {\it ${\mathcal R}$-finite} if $W_g=\{w\in W\,|\,(w,g)\in {\mathcal R}\}$ is finite. 

Clearly, any finite monoid is ${\mathcal R}$-finite with ${\mathcal R}=W\times (W\setminus \{1\})\cup \{(1,1)\}$. Also any Coxeter group $W$ is ${\mathcal R}$-finite with ${\mathcal R}$ being a Bruhat order on $W$.


Given an ${\mathcal R}$-finite monoid $W$, define the algebra ${\bf B}(W,{\mathcal R})$ over $\ZZ$ to be generated by $d_{g,h}$, $g,h\in W$  subject to relations $d_{g,w}=0$ if $(w,g)\notin {\mathcal R}$,
$d_{1,1}=1$ and:
\begin{equation}
\label{eq:defining relations universal Nichols}
d_{gh,w}=\sum_{w_1,w_2\in W:w_1w_2=w} d_{g,w_1} d_{h,w_2}
\end{equation}
for all $g,h\in W$, $w\in W$.


\begin{proposition} 
\label{pr:bialgebra W}
For any ${\mathcal R}$-finite monoid $W$ one has:

(a) the algebra  ${\bf B}(W,{\mathcal R})$ is a bialgebra with the coproduct $\Delta$ and the counit $\varepsilon$ given respectively by (for all $g,h\in W$):
\begin{equation}
\label{eq:universal Nichols}
\Delta(d_{g,h})=\sum_{w\in W} d_{g,w}\otimes d_{w,h},~\varepsilon(d_{g,h})=\delta_{g,h}
\end{equation}

(b) Suppose that $\varphi$ is any anti-automorphism of $W$ such that $(\varphi\times \varphi)({\mathcal R})={\mathcal R}$. Then the assignments $d_{g,h}\mapsto d_{\varphi(g),\varphi(h)}$, $g,h\in W$ define an anti-automorphism $\varphi^*$ of ${\bf B}(W,{\mathcal R})$ such that $(\varphi^*\otimes \varphi^*)\circ \Delta=\Delta\circ \varphi^*$ and $\varepsilon\circ \varphi^*=\varepsilon$.

\end{proposition}

\begin{proof} 
Prove (a). Let  $U({\mathcal R})$ be the free $\ZZ$-module with the free basis $d_{g,h}$, $(g,h)\in {\mathcal R}$. 
The following is immediate.

\begin{lemma} 
$U({\mathcal R})$  is a coalgebra with the coproduct and the counit given by \eqref{eq:universal Nichols}.

\end{lemma}

This implies that the tensor algebra $T(U({\mathcal R}))$ is naturally a bialgebra. 
Denote by $\hat {\bf B}(W,{\mathcal R})$ the quotient of $T(U({\mathcal R}))$ by the ideal $J$ generated by $d_{1,1}-1$. Since 
$\Delta(d_{1,1}-1)=d_{1,1}\otimes d_{1,1}-1\otimes 1=(d_{11}-1)\otimes d_{11}+1\otimes (d_{1,1}-1)$ and $\varepsilon(d_{1,1}-1)=0$, 
$J$ is a bi-ideal hence $\hat {\bf B}(W,{\mathcal R})$ is a bialgebra.

For each $g,h\in W$ and $w\in W$ define elements $\delta_{g,h;w}\in \hat {\bf B}(W,{\mathcal R})$ by:
$$\delta_{g,h;w}:=d_{gh,w}-\sum_{w_1,w_2\in W:w_1w_2=w} d_{g,w_1} d_{h,w_2}\ .$$
Denote by ${\bf K}={\bf K}(W,{\mathcal R})$ the $\ZZ$-submodule $\sum\limits_{g,h,w\in W} \ZZ\cdot \delta_{g,h;w}$ of $\hat {\bf B}(W,{\mathcal R})$.

\begin{lemma} 
\label{le:K two-sided coideal}
${\bf K}(W,{\mathcal R})$ is a two-sided coideal in $\hat {\bf B}(W,{\mathcal R})$.

\end{lemma} 

\begin{proof} Indeed, $\Delta(\delta_{g,h;w})=\sum\limits_{w'\in W} d_{gh,w'}\otimes d_{w',w}-
\sum\limits_{w_1,w_2,w'_1,w'_2\in W} d_{g,w'_1}d_{h,w'_2}\otimes d_{w'_1,w_1}d_{w'_2,w_2}$ 
 %
%
$$=\sum_{w'\in W}\delta_{g,h;w'}\otimes d_{w',w} +\sum\limits_{w'_1,w'_2\in W} d_{g,w'_1} d_{h,w'_2}\otimes d_{w'_1w'_2,w}-\sum_{w_1,w_2,w'_1,w'_2\in W} d_{g,w'_1}d_{h,w'_2}\otimes d_{w'_1,w_1}d_{w'_2,w_2}$$
$$=\sum_{w'}\delta_{g,h;w'}\otimes d_{w',w} +\sum\limits_{w'_1,w'_2\in W}  d_{g,w'_1}d_{h,w'_2}\otimes \delta_{w'_1,w'_2;w}\ .$$
where we used that $d_{gh,w'}=\delta_{g,h;w'}+\sum\limits_{w'_1,w'_2\in W:w'_1w'_2=w'} d_{g,w'_1} d_{h,w'_2}$.

This proves that $\Delta({\bf K})\subset\hat {\bf B}(W,{\mathcal R})\otimes {\bf K}+{\bf K}\otimes \hat {\bf B}(W,{\mathcal R})$. It remains to show that $\varepsilon({\bf K})=0$. We have
$\varepsilon(\delta_{g,h;w})=\delta_{gh,w}-\sum\limits_{w_1,w_2\in W, w_1w_2=w} \delta_{g,w_1}\delta_{h,w_2}=\delta_{gh,w}-\delta_{gh,w}=0$ 
for all $g,h,w\in W$. 

The lemma is proved.
\end{proof}

Denote by ${\bf J}$ the ideal of $\hat {\bf B}(W,{\mathcal R})$ generated by ${\bf K}={\bf K}(W,{\mathcal R})$. Let us show that ${\bf J}$ is a bi-ideal in $\hat {\bf B}(W,{\mathcal R})$. Lemma \ref{le:K two-sided coideal} implies that 
$\varepsilon({\bf J})=0$ and:  
$$\Delta({\bf J})\subset (\hat {\bf B}(W,{\mathcal R})\otimes \hat {\bf B}(W,{\mathcal R}))\cdot (\hat {\bf B}(W,{\mathcal R})\otimes {\bf K}+{\bf K}\otimes \hat {\bf B}(W,{\mathcal R}))\cdot (\hat {\bf B}(W,{\mathcal R})\otimes \hat {\bf B}(W,{\mathcal R}))$$
$$\subset \hat {\bf B}(W,{\mathcal R})\otimes {\bf K}+{\bf K}\otimes \hat {\bf B}(W,{\mathcal R}) \ .$$

Finally, since ${\bf B}(W,{\mathcal R})=\hat {\bf B}(W,{\mathcal R})/{\bf J}$, this implies that ${\bf B}(W,{\mathcal R})$ is a bialgebra.
This proves (a). 

Prove (b) now. Clearly, the assignments $d_{g,h}\mapsto d_{\varphi(g),\varphi(h)}$, $g,h\in W$ define an anti-automorphism $\varphi^*$ of the coalgebra $U({\mathcal R})$ such that $(\varphi^*\otimes \varphi^*)\circ \Delta=\Delta\circ \varphi^*$ and $\varepsilon\circ \varphi^*=\varepsilon$.  Therefore, passing to the tensor algebra $T(U({\mathcal R}))$ this gives an anti-automorphism $\tilde \varphi^*$ of $T(U({\mathcal R}))$ with the same properties. Furthermore, $\tilde \varphi^*(d_{11}-1)=d_{11}-1$, thus, $\tilde \varphi^*$ preserves the above bi-ideal $J$ generated by $d_{11}-1$, thus, gives a well-defined anti-automorphism of $\hat \varphi^*$ of the quotient bialgebra $\hat {\bf B}(W,{\mathcal R})$. In turn, we have 
$\hat \varphi(\delta_{g,h;w}):=d_{\varphi(gh),\varphi(w)}-\sum_{w_1,w_2\in W:\varphi(w_1w_2)=\varphi(w)}  d_{\varphi(h),\varphi(w_2)}d_{\varphi(g),\varphi(w_1)}$
$$=d_{\varphi(h)\varphi(g),\varphi(w)}-\sum_{w'_1,w'_2\in W:w'_1w'_2=\varphi(w)} d_{\varphi(h),w'_1}d_{\varphi(g),w'_2}=\delta_{\varphi(h),\varphi(g);\varphi(w)}$$
for all $g,h,w\in W$. In particular $\hat \varphi^*({\bf K}(W,{\mathcal R})={\bf K}(W,{\mathcal R})$ hence the bi-ideal ${\bf J}$ generated by ${\bf K}(W,{\mathcal R})$ is $\hat \varphi^*$-invariant hence one has a natural anti-automorphism  $\varphi^*$ on the quotient bialgebra ${\bf B}(W,{\mathcal R})=\hat {\bf B}(W,{\mathcal R})/{\bf J}$.
This proves (b).

The proposition is proved.
\end{proof}

The following is an immediate corollary of Propositions \ref{pr:bialgebra W} and \ref{pr:factored derivatives}. 

\begin{corollary} 
\label{cor:partial action} Let $W$ be a monoid and ${\mathcal R}$ be preorder on $W$ an so that $W$ is ${\mathcal R}$-finite. 
Suppose that ${\bf H}$ is an $R$-algebra which factors as ${\bf H}={\bf D}\cdot R W$ over $R$ where ${\bf D}$ is a subalgebra. Suppose  that $g\cdot {\bf D}\subset {\bf D}\cdot W_g$ for all $g\in W$. Then ${\bf D}$ is a module algebra over ${\bf B}(W,{\mathcal R})\otimes R$ via $d_{g,h}\mapsto \partial_{g,h}$.

\end{corollary}

\begin{remark} The  ``universally acting" bialgebra ${\bf B}(W,{\mathcal R})$ is a particular case of the bialgebras emerging in the forthcoming joint paper of Yury Bazlov with the first author \cite{BBHcross}.

\end{remark}
 
For any ${\mathcal R}$-finite monoid $W$ let $\underline {\bf B}(W,{\mathcal R})$ be the quotient algebra of ${\bf B}(W,{\mathcal R})$ by the ideal generated by all $d_{gh,gh}-d_{g,g}d_{h,h}$. 

\begin{proposition} 
In the assumptions of Proposition \ref{pr:bialgebra W}, suppose that ${\mathcal R}$ is a poset. Then 
$\underline {\bf B}(W,{\mathcal R})$ is naturally a bialgebra.

\end{proposition}

\begin{proof} 
For $g,h\in W$ let $\delta_{g;h}\in {\bf B}(W,{\mathcal R})$ be given by $\delta_{g;h}:=d_{gh,gh}- d_{g,g}d_{h,h}$ and let $\underline {\bf K}(W,{\mathcal R})$ be the $\ZZ$-submodule 
$\sum\limits_{g,h\in W} \ZZ\cdot\delta_{g;h}$ of ${\bf B}(W,{\mathcal R})$.

\begin{lemma} 
\label{le:underline K two-sided coideal}
$\underline {\bf K}(W,{\mathcal R})$ is a two-sided coideal in $\hat {\bf B}(W,{\mathcal R})$.

\end{lemma}

\begin{proof}
Since ${\mathcal R}$ is a partial order, then $\Delta(d_{g,g})=d_{g,g}\otimes d_{g,g}$ for $g\in W$. 
Therefore, 
$$\Delta(\delta_{g;h})=d_{gh,gh}\otimes d_{gh,gh} - d_{g,g}d_{h,h}\otimes d_{g,g}d_{h,h}=d_{gh,gh}\otimes \delta_{g;h}+\delta_{g;h}\otimes d_{g,g}d_{h,h},~\varepsilon(\delta_{g;h})=1-1\cdot 1=0$$
for all $g,h\in W$. Finally, $\varepsilon(\delta_{g;h})=1-1\cdot 1=0$ for all $g,h\in W$.

The lemma is proved.
\end{proof}

Denote $\underline {\bf J}:={\bf B}(W,{\mathcal R})\cdot {\bf K}\cdot {\bf B}(W,{\mathcal R})$. Similarly to the proof of Proposition \ref{pr:bialgebra W}, one shows that this is the bi-ideal of the bialgebra ${\bf B}(W,{\mathcal R})$.

Finally, since $\underline {\bf B}(W,{\mathcal R})={\bf B}(W,{\mathcal R})/\underline  {\bf J}$, this implies that $\underline {\bf B}(W,{\mathcal R})$ is a bialgebra. 

%
%
%
%
%

The proposition is proved.
\end{proof}

\begin{remark} If $W$ is a group, then one can ask whether $\underline {\bf B}(W,{\mathcal R})$ is a Hopf algebra. In that case,  the antipode is given by:
$S(d_{g,h})=\sum (-1)^{k-1} d_{w_1,w_1}^{-1}d_{w_1,w_2}d_{w_2,w_2}^{-1}\cdots d_{w_{k-1},w_{k-1}}^{-1}d_{w_{k-1},w_k}d_{w_k,w_k}^{-1}$,
where the summation is over all $k\ge 1$ and distinct $w_1,\ldots,w_k\in W$ such that $w_1=g$, $w_k=h$.

This computation is based on the following well-known fact: any  lower triangular $n\times n$ matrix $A=(a_{ij})$ over an associative unital ring ${\mathcal A}$ such that all $a_{ii}$ are invertible in ${\mathcal A}$,  is invertible over ${\mathcal A}$ and 
$(A^{-1})_{ij}=\sum\limits_{i=i_1>i_2>\cdots >i_k=j,k\ge 1} (-1)^{k-1}a_{i_1,i_1}^{-1}a_{i_1,i_2}a_{i_2,i_2}^{-1}\cdots a_{i_{k-1},i_{k-1}}^{-1}a_{i_{k-1},i_k}a_{i_k,i_k}^{-1}$ for $1\le j\le i\le n$.
\end{remark}

Furthermore, for $g,h,w\in W$ define elements $v_{g,h}^w\in \underline {\bf B}(W,{\mathcal R})$ by $v_{g,h}^w:=d_{w,w}d_{g,h}d_{h,h}^{-1}d_{w,w}^{-1}$ and let ${\mathcal B}(W,{\mathcal R})$ be the subalgebra of $\underline {\bf B}(W,{\mathcal R})$ generated by all $v_{g,h}^w$.

We refer to ${\mathcal B}(W,{\mathcal R})$ as the {\it generalized Nichols algebra of $(W,{\mathcal R})$} due to the following result.

\begin{theorem} 
\label{th:presentation Nichols}
Let $W$ be a group and ${\mathcal R}$ be a partial order on $W$ so that $W$ is ${\mathcal R}$-finite. Then:

\noindent (a) ${\mathcal B}(W,{\mathcal R})$ is an algebra over  $\ZZ$ generated by $v_{g,h}^w$,  $g,h,w\in W$,  subject to relations $v_{g,h}^w=0$ if $(h,g)\notin {\mathcal R}$ and (for  $w,w',g,h\in W$):
\begin{equation}
\label{eq:presentation Nichols}
v_{w,w}^{w'}=1,~v_{gh,w}^{w'}=\sum\limits_{w_1,w_2\in W:w_1w_2=w} v_{g,w_1}^{w'} v_{h,w_2}^{w'w_1}
\end{equation}

\noindent (b)  ${\mathcal B}(W,{\mathcal R})$ is a module algebra over $\ZZ W$ with respect to the action  given by 
\begin{equation}
\label{eq:W action Nichols}
w(v_{g,h}^{w'})=v_{g,h}^{ww'}
\end{equation} for all  $w,w',g,h\in W$.

\noindent (c) The algebra $\underline {\bf B}(W,{\mathcal R})$ is isomorphic to the cross product ${\mathcal B}(W,{\mathcal R})\rtimes \ZZ W$.

\noindent (d) ${\mathcal B}(W,{\mathcal R})$ is a bialgebra in the (braided monoidal) category $_W^W{\mathcal YD}$ of Yetter-Drinfeld modules over $\ZZ W$ (see e.g., \cite[Section 1.2]{AS}) with:

$\bullet$ $W$-grading given by $|v_{g,h}^w|=wg(wh)^{-1}$ for all $w,g,h\in W$.

$\bullet$ The (braided) coproduct given by $\underline \Delta(v_{g,h}^w)=\sum\limits_{w'\in W} v_{g,w'}^w\otimes v_{w',h}^w$
for all $g,h,w\in W$.

$\bullet$ The (braided) counit given by $\underline \varepsilon(v_{g,h}^w)=\delta_{g,h}$ for all $g,h,w\in W$.

\end{theorem}

\begin{proof} Let ${\mathcal B}'(W,{\mathcal R})$ be the algebra generated by all $v_{g,h}^w$, $g,h,w\in W$, subject to the relations $v_{g,h}^w=0$ if $(h,g)\notin {\mathcal R}$ and \eqref{eq:presentation Nichols}. We need the following immediate fact.

\begin{lemma} \eqref{eq:W action Nichols} defines a $W$-action on ${\mathcal B}'(W,{\mathcal R})$ by algebra automorphisms.

\end{lemma}

Therefore, ${\mathcal B}'(W,{\mathcal R})$ is a $\ZZ W$-module algebra, which, in particular, proves (b). 

Prove (a) and (c) now. 
Denote $\underline {\bf B}'(W,{\mathcal R}):={\mathcal B}'(W,{\mathcal R})\rtimes \ZZ W$.

\begin{proposition} The assignments $d_{g,h}\mapsto v_{g,h}^1\cdot h$ for $g,h\in W$ define an isomorphism of algebras 
${\bf f}_W:\underline {\bf B}(W,{\mathcal R})\widetilde \to \underline {\bf B}'(W,{\mathcal R})$ such that ${\bf f}_W({\mathcal B}(W,{\mathcal R}))={\mathcal B}'(W,{\mathcal R})$.

\end{proposition}

\begin{proof} 
Since $d_{g,h}=d_{w,w}^{-1}v_{g,h}^w d_{w,w}d_{h,h}$ in $\underline {\bf B}(W,{\mathcal R})$ for all $g,h,w\in W$, $(g,h)\in {\mathcal R}$, substituting this to \eqref{eq:defining relations universal Nichols} gives the relations
$$v_{gh,w}^1 d_{w,w}=\sum_{w_1,w_2\in W: w_1w_2=w} v_{g,w_1}^1d_{w_1,w_1}d_{w_1,w_1}^{-1}v_{h,w_2}^{w_1}d_{w_1,w_1}d_{w_2,w_2}=\sum_{w_1,w_2\in W: w_1w_2=w} v_{g,w_1}^1v_{h,w_2}^{w_1}d_{w,w}$$
for all $g,h,w\in W$, which gives the second relation in \eqref{eq:presentation Nichols} with $w'=1$. Finally, using
\begin{equation}
\label{eq:cross product relations W}
d_{w_1,w_1}v_{h,w_2}^{w_3}d_{w_1,w_1}^{-1}=v_{h,w_2}^{w_1w_3}
\end{equation}
in $\underline {\bf B}(W,{\mathcal R})$ for any $h,w_1,w_2,w_3\in W$ we obtain the second relation \eqref{eq:presentation Nichols} for with any $w'\in W$ (the relations $v_{w,w}^{w'}=1$ are obvious).

Let 
$\underline V_W:=\oplus_{g,h\in W:(h,g)\in {\mathcal R}} \ZZ\cdot d_{g,h}$
Clearly, the assignments $d_{g,h}\mapsto d_{g,h}$ define a canonical surjective homomorphism $T(\underline V_W)\twoheadrightarrow {\bf B}(W,{\mathcal R})$, whose kernel is the ideal 
of $T(\underline V_W)$ generated by 
\begin{equation}
\label{eq:ideal of relations in BWR}
d_{1,1}-1,~d_{gh,gh}-d_{g,g}d_{h,h}, ~d_{gh,w}-\sum_{w_1,w_2\in W: w_1w_2=w} d_{g,w_1}d_{h,w_2}
\end{equation}
for $g,h,w\in W$.

Then the homomorphism of algebras $\hat {\bf f}_W:\underline {\bf B}(W,{\mathcal R})\widetilde \to {\bf B}'(W,{\mathcal R})$ by $\hat {\bf f}_W(d_{g,h})= v_{g,h}^1\cdot h$ for $g,h\in W$. Clearly, the image of \eqref{eq:ideal of relations in BWR} under $\hat {\bf f}_W$ is
$$v_{1,1}^1-1=0,~v_{gh,gh}^1 gh-v_{g,g}^1gv_{h,h}^1h=0, ~v_{gh,w}^1w-\sum_{w_1,w_2\in W: w_1w_2=w} v_{g,w_1}^1w_1v_{h,w_2}^1w_2=0$$
because relatons \eqref{eq:presentation Nichols} hold in ${\mathcal B}(W,{\mathcal R})$ and 
\begin{equation}
\label{eq:cross product relations W prim}
w_1v_{h,w_2}^{w_3}w_1^{-1}=v_{h,w_2}^{w_1w_3}
\end{equation}
in $\underline {\bf B}'(W,{\mathcal R})$ for any $h,w_1,w_2,w_3\in W$.

This proves that ${\bf f}_W$ is a well-defined homomorphism of algebras $\underline {\bf B}(W,{\mathcal R}) \to \underline {\bf B}'(W,{\mathcal R})$. It is clearly surjective due to \eqref{eq:cross product relations W}. Injectivity of ${\bf f}_W$ follows from that the defining relations \eqref{eq:presentation Nichols} and \eqref{eq:cross product relations W prim} of ${\bf B}'(W,{\mathcal R})$ (together with $v_{g,h}^w=0$ if $(h,g)\notin {\mathcal R}$) already hold in ${\bf B}(W,{\mathcal R})$ (since ${\bf f}_W(d_{g,g})=g$ the relations \eqref{eq:cross product relations W} match  \eqref{eq:cross product relations W prim}).

The proposition is proved.
\end{proof}

This finishes the proof of (a) and (c).

Prove (d) now. We need the following result.
%
%
%
%
%
%
%
%

\begin{lemma} 
\label{le:YD coalgebra}
In the assumptions of Theorem \ref{th:presentation Nichols}, one has:

(a) The $\ZZ$-module $Y_W:=\oplus_{g,h,w\in W:(h,g)\in {\mathcal R}} \ZZ\cdot v_{g,h}^w$ 
(convention: $v_{g,h}^w=0$ if $(h,g)\notin {\mathcal R}$) 
is a Yetter-Drinfeld module over $W$ with the $W$-action and $W$-grading as in Theorem \ref{th:presentation Nichols}(d).

(b) The maps $\underline \Delta:Y_W\to Y_W\otimes Y_W$ and $\underline \varepsilon: Y_W\to \ZZ$ given by Theorem \ref{th:presentation Nichols}(d)
turn $Y_W$ into a coalgebra in the
(braided monoidal) category $_W^W{\mathcal YD}$ of Yetter-Drinfeld modules over $W$. 
\end{lemma}

\begin{proof} Indeed, 
$|w(v_{g,h}^{w'})|=|v_{g,h}^{ww'}|=ww'gh^{-1}(ww')^{-1}=w|v_{g,h}^{w'}|w^{-1}$ for all $g,h,w,w'\in W$.
This proves (a). 

Prove (b). Clearly, both $\underline \Delta$ and $\underline \varepsilon$ commute with $W$-action. Also using the standard grading on $Y_W\otimes Y_W$ via $|x\otimes y|=|x|\cdot |y|$ for homogeneous $x,y\in Y_W$, we obtain $|v_{g,h}^w|=wg(wh)^{-1}$ and
$$|v_{g,w'}^w\otimes v_{w',h}^w|= |v_{g,w'}^w|\cdot |v_{w',h}^w|=wg(ww')^{-1}ww'(wh)^{-1}=wg(wh)^{-1}=|v_{g,h}^w|$$
for all $g,h,w,w'\in W$, therefore, $|\underline \Delta(v_{g,h}^w)|=|v_{g,h}^w|$ for all $g,h,w\in W$. Similarly, 
$\underline \varepsilon(v_{g,}^w)=|\delta_{g,h}|=\delta_{g,h}\cdot 1$ for $g,h,w\in W$. This proves that both $\underline \Delta$ and 
$\underline \varepsilon$ are morphisms in $_W^W{\mathcal YD}$.
Coassociativity of $\underline \Delta$ and the counit axiom follow.
This proves (b).

The lemma is proved.
\end{proof}
 

Lemma \ref{le:YD coalgebra}(b) implies that $\underline \Delta$ viewed as a morphism from $Y_W$ to the algebra $T(Y_W)\otimes T(Y_W)$ extends to a homomorphism  $\underline \Delta:T(Y_W)\to T(Y_W)\otimes T(Y_W)$ of algebras in the braided monoidal category $_W^W{\mathcal YD}$. Similarly, $\underline \varepsilon$ extends to a homomorphism of algebras $T(Y_W)\to \ZZ$, the latter viewed as the unit object in $_W^W{\mathcal YD}$.
Thus, $T(Y_W)$ is a bialgebra in the braided monoidal category $_W^W{\mathcal YD}$.

For $g,h,w,w'\in W$ define elements $\underline \delta_w^{w'}$, $\underline \delta_{g,h;w'}\in T(Y_W)$ by 
$$\underline \delta_w^{w'}=v_{w,w}^{w'}-1,~\underline \delta_{g,h,w;w'}=v_{gh,w}^{w'}-\sum\limits_{w_1,w_2\in W:w_1w_2=w} v_{g,w_1}^{w'} v_{h,w_2}^{w'w_1}$$
and denote by $\underline {\mathcal K}(W,{\mathcal R})$ the $\ZZ$-submodule of $T(Y_W)$ generated by  all $\underline \delta_w^{w'}$ and $\underline \delta_{g,h;w'}$.

Clearly, these elements are homogeneous, more precisely,  $|\underline \delta_w^{w'}|=1$, $|\underline \delta_{g,h,w;w'}|=w'gh(w'w)^{-1}$ for all $w,w',g,h\in W$. Moreover $w''(\underline \delta_w^{w'})=\underline \delta_w^{w'' w'}$, 
$w''(\underline \delta_{g,h,w;w'})=\underline \delta_{g,h,w;w''w'}$ for all $w''\in W$, in particular,  $\underline {\mathcal K}(W,{\mathcal R})$ is a Yetter-Drinfeld submodule of $T(Y_W)$. 

The following is a braided version of Lemma \ref{le:underline K two-sided coideal}.

\begin{lemma} 
\label{le:underline K two-sided coideal braided} 
$\underline {\mathcal K}(W,{\mathcal R})$ is a two-sided coideal in $T(Y_W)$ in  $_W^W{\mathcal YD}$ and $\underline \varepsilon({\mathcal K}(W,{\mathcal R}))=\{0\}$.

\end{lemma}

\begin{proof} Indeed, 
$\underline \Delta(\underline \delta_w^{w'})=v_{w,w}^{w'}\otimes v_{w,w}^{w'}-1\otimes 1=\underline \delta_w^{w'}\otimes v_{w,w}^{w'}+1\otimes \underline \delta_w^{w'}$.

Furthermore, 
$\underline \Delta(\underline \delta_{g,h,w;w'})=\underline \Delta(v_{gh,w}^{w'})-\sum\limits_{w_1,w_2\in W:w_1w_2=w} \underline \Delta(v_{g,w_1}^{w'}) \underline \Delta(v_{h,w_2}^{w'w_1})$
$$=\sum_{w''\in W} v_{gh,w''}^{w'}\otimes v_{w'',w}^{w'}-\sum\limits_{w_1,w_2,w_1'',w_2''\in W:w_1w_2=w} (v_{g,w_1''}^{w'}\otimes v_{w_1'',w_1}^{w'}) (v_{h,w''_2}^{w'w_1}\otimes v_{w_2'',w_2}^{w'w_1})$$
$$=\sum_{w''\in W} v_{gh,w''}^{w'}\otimes v_{w'',w}^{w'}-\sum\limits_{w_1,w_2,w_1'',w_2''\in W:w_1w_2=w} v_{g,w_1''}^{w'}v_{h,w_2''}^{w'w_1''}\otimes v_{w_1'',w_1}^{w'}v_{w_2'',w_2}^{w'w_1}\ ,$$
where we used the fact that 
$(v_{g,w_1''}^{w'}\otimes v_{w_1'',w_1}^{w'}) (v_{h,w_2''}^{w'w_1}\otimes v_{w_2'',w_2}^{w'w_1})=
v_{g,w_1''}^{w'}v_{h,w_2''}^{w'w_1''}\otimes v_{w_1'',w_1}^{w'}v_{w_2'',w_2}^{w'w_1}$
because $(x\otimes y)(z\otimes t)=x\cdot (|y|(z))\otimes yt$ for any $x,y,z,t\in T(Y_W)$, where $y$ is homogeneous of degree $|y|$, and 
$|v_{w_1'',w_1}^{w'}|(v_{h,w_2''}^{w'w_1})=(w'w_1''(w'w_1)^{-1})(v_{h,w_2''}^{w'w_1})=v_{h,w_2''}^{w'w_1''}$.
Finally, taking into account that 
$$\sum_{w''\in W} v_{gh,w''}^{w'}\otimes v_{w'',w}^{w'}=\sum_{w''} \underline\delta_{g,h,w''}^{w'}\otimes v_{w'',w}^{w'}
+\sum\limits_{w_1'',w_2''\in W}v_{g,w_1''}^{w'}v_{h,w_2''}^{w'w_1''}\otimes v_{w_1''w_2'',w}^{w'}\ ,$$
we obtain
$\underline \Delta(\underline \delta_{g,h,w;w'})=\sum\limits_{w''\in W} \underline\delta_{g,h,w''}^{w'}\otimes v_{w'',w}^{w'}+\sum\limits_{w_1'',w_2''\in W} v_{g,w_1''}^{w'}v_{h,w_2''}^{w'w_1''}\otimes \underline \delta_{w_1''w_2'',w;w'}$.

This proves that $\underline \Delta(\underline {\mathcal K}(W,{\mathcal R}))\subset 
\underline {\mathcal K}(W,{\mathcal R})\otimes T(Y_W)+T(Y_W)\otimes \underline {\mathcal K}(W,{\mathcal R})$. It remains to show that $\underline \varepsilon({\mathcal K}(W,{\mathcal R}))=\{0\}$. 

Indeed, $\underline \varepsilon(\underline \delta_w^{w'})=1-1=0$, 
$\underline \varepsilon(\underline \delta_{g,h,w;w'})=\delta_{gh,w}-\sum\limits_{w_1,w_2\in W:w_1w_2=w} \delta_{g,w_1} \delta_{h,w_2}=\delta_{gh,w}-\delta_{gh,w}=0$.

The lemma is proved.
\end{proof}

Similarly to the conclusion of the proof of Proposition \ref{pr:bialgebra W}, denote $\underline {\bf J}':=T(Y_W)\cdot \underline {\mathcal K}(W,{\mathcal R})\cdot T(Y_W)$. This is the ideal of $T(Y_W)$ generated by $\underline {\mathcal K}(W,{\mathcal R})$. Let us show that $\underline {\bf J}'$ is a bi-ideal in $\hat {\bf B}(W,{\mathcal R})$. Clearly, $\varepsilon(\underline {\bf J}')=0$ by Lemma \ref{le:underline K two-sided coideal braided}. Furthermore, Lemma \ref{le:underline K two-sided coideal braided} implies that  
$$\Delta(\underline {\bf J}')\subset (T(Y_W)\otimes T(Y_W))\cdot (T(Y_W)\otimes {\mathcal K}(W,{\mathcal R})+{\mathcal K}(W,{\mathcal R})\otimes 
T(Y_W))\cdot (T(Y_W)\otimes T(Y_W))$$
$$\subset T(Y_W)\otimes {\mathcal K}(W,{\mathcal R})+{\mathcal K}(W,{\mathcal R})\otimes T(Y_W)$$
because 
$(T(Y_W)\otimes T(Y_W))\cdot (T(Y_W)\otimes Y+Y\otimes T(Y_W))\cdot (T(Y_W)\otimes T(Y_W))$
$$\subset T(Y_W)YT(Y_W)\otimes T(Y_W)+T(Y_W)\otimes T(Y_W)YT(Y_W)$$
for any Yetter-Drinfeld submodule $Y$ of $T(Y_W)$.

Thus, $\underline {\bf J}'$ is a bi-ideal and ${\mathcal B}(W,{\mathcal R})=T(Y_W)/\underline {\bf J}'$ is a bialgebra in $_W^W{\mathcal YD}$.
This proves (d). 

The theorem is proved.
\end{proof}

It turns out that for Coxeter groups these Hopf algebras are closely related to the graded versions of Hecke-Hopf algebra.

\begin{definition}
\label{def:H_0 Coxeter}
For any Coxeter group $W$ let $\hat {\bf H}_0(W)$ be the algebra generated by $s_i,d_i$, $i\in I$ subject to relations:
\label{def:Hecke-Hopf algebra W prime}

(i) Rank $1$ relations: $s_i^2=1$,  $d_i^2=0$, $s_id_i+d_is_i=0$ for $i\in I$.


(ii) Coxeter relations: $(s_is_j)^{m_{ij}}=1$ and linear braid relations: $\underbrace{d_is_js_i\cdots s_{j'}}_{m_{ij}} =\underbrace{s_j\cdots s_{i'}s_{j'}d_{i'}}_{m_{ij}}$
for all distinct $i,j\in I$ with $m_{ij}\ne 0$, where $i'=\begin{cases} 
i & \text{if $m_{ij}$ is even}\\
j & \text{if $m_{ij}$ is odd}\\
\end{cases}$ and  $\{i',j'\}=\{i,j\}$.
%

\end{definition} 

That is, $\hat {\bf H}_0(W)$ is given by ``homogenizing"  Definition \ref{def:Hecke-Hopf algebra W}.

Similarly to Section \ref{sect:intro}, for any $s\in {\mathcal S}$ there is a unique element $d_s\in \hat {\bf H}_0(W)$ such that $d_{s_i}=d_i$ for $i\in I$ and $d_{s_iss_i}=s_iD_ss_i$ 
for any $i\in I$, $s\in {\mathcal S}\setminus \{s_i\}$. It is easy to see that $wd_sw^{-1}=\chi_{w,s}d_{wsw^{-1}}$, where $\chi_{w,s}$ is defined in \eqref{eq:chi sigma defined} (cf. \cite[Section 5]{MS}). 

Denote by $\hat {\bf D}_0(W)$ the subalgebra of $\hat {\bf H}_0(W)$ generated by $d_s$, $s\in {\mathcal S}$. 

The following is an immediate homogeneous analogue of Theorem \ref{th:ZWcross and free D}.

\begin{lemma} 
For any Coxeter group $W$, one has:

(a) the algebra $\hat {\bf D}_0(W)$  is generated by all $d_s$, $s\in {\mathcal S}$ subject to  relations $d_s^2=0$, $s\in {\mathcal S}$. 

(b) $\hat {\bf H}_0(W)$ is naturally isomorphic to the cross product $\hat {\bf D}_0(W)\rtimes \ZZ W$ with respect to the action of $W$ on $\hat {\bf D}_0(W)$ given by $w(d_s)=\chi_{w,s}d_s$ for $w\in W$, $s\in {\mathcal S}$, $\chi_{w,s}$ is defined in \eqref{eq:chi sigma defined}.

(c) $\hat {\bf D}_0(W)$ is graded by $W$ via $|d_s|=s$ for $s\in {\mathcal S}$ and is a Hopf algebra in the category $_W^W{\mathcal YD}$ with the braided coproduct, counit, and the antipode given respectively by (for $s\in {\mathcal S}$):
$$\underline \Delta(d_s)=d_s\otimes 1+1\otimes d_s,~\underline \varepsilon(d_s)=0,~\underline S(d_s)=-d_s\ .$$

\end{lemma} 

\begin{remark} 
\label{rem:pre-Nichols}
In fact, $\hat {\bf D}_0(W)$ is a {\it pre-Nichols algebra} of the braided vector space $\oplus_{s\in {\mathcal S}} \ZZ \cdot d_s$ in terminology of \cite{Ma}.

\end{remark}

We can ``approximate" the braided bialgebra ${\mathcal B}(W,{\mathcal R}_W)$, where ${\mathcal R}_W$ is the strong Bruhat order on $W$ (see e.g., \cite[Section 2]{BjBr}), by the pre-Nichols algebra $\hat {\bf D}_0(W)$.
  
\begin{theorem} 
\label{th:nichols graded homomorphism}
Let $W$ be a Coxeter group and ${\mathcal R}_W$ be the strong Bruhat order on $W$. Then 

(a) The assignments $s_i\mapsto d_{s_i,s_i}$, $d_i\mapsto -d_{s_i,1}$, $i\in I$ define a surjective homomorphism of bialgebras 
\begin{equation}
\label{eq:nichols graded homomorphism}
\hat \varphi_W:\hat {\bf H}_0(W)\twoheadrightarrow \underline {\bf B}(W,{\mathcal R}_W)\ .
\end{equation}
whose restriction  to $\ZZ W$ is injective.

(b) In the notation of Theorem \ref{th:presentation Nichols}(d), the restriction of $\hat \varphi_W$ to $\hat {\bf D}_0(W)$ is a surjective homomorphism of  bialgebras in $_W^W{\mathcal YD}$
\begin{equation}
\label{eq:nichols homomorphism}
\underline {\hat \varphi}_W:\hat {\bf D}_0(W)\twoheadrightarrow {\mathcal B}(W,{\mathcal R}_W)\ .
\end{equation}
whose restriction  to $\underline Y_W:=\oplus_{s\in {\mathcal S}} \ZZ \cdot d_s$ is injective.


\end{theorem}

\begin{proof}
We need the following result.

\begin{proposition} 
\label{pr:refined factored derivatives} Let $W$ be a Coxeter group. Then  
$gx=g(x)\cdot g+\sum\limits_{h\in W\setminus \{g\}:(h,g)\in {\mathcal R}_W} \partial_{g,h}(x)h$ for all $g\in W$, $x\in \hat {\bf D}(W)$, in the notation of Proposition \ref{pr:factored derivatives},  where $(g,x)\mapsto g(x)$ is the $W$-action on $\hat {\bf D}(W)$ given by Theorem \ref{th:symmetris D(W)}(b)(i). In particular, $\hat {\bf D}(W)$ is a module algebra over $\underline {\bf B}(W,{\mathcal R}_W)$.


\end{proposition}

\begin{proof} 
Let us prove the implication
\begin{equation}
\label{eq:implication}
\partial_{g,h}\ne 0 ~=> ~(h,g)\in {\mathcal R}_W
\end{equation}
for all $g,h\in W$.

We need the following result.

\begin{lemma} 
\label{le:characterization of Bruhat order}
Let $W$ be a Coxeter group. Suppose that $w,w'\in W$ such that $(w',w)\in {\mathcal R}_W$ and let $i\in I$ be such that $\ell(s_iw)=\ell(w)+1$ and $\ell(s_iw')=\ell(w')+1$. Then $(s_iw',s_iw)\in {\mathcal R}_W$.

\end{lemma}

\begin{proof} Indeed, it is well-known (see e.g., \cite[Theorem 2.2.2]{BjBr}) that  $(w',w)\in {\mathcal R}_W$ iff 
\begin{equation}
\label{eq:the characterization of strong Bruhat default}
w=w_1s_{i_1}w_2 s_{i_2}\cdots w_ks_{i_k}w_{k+1},~w'=w_1\cdots w_{k+1}
\end{equation}
for some $i_1,\ldots,i_k\in I$, and $k\ge 0$ such that $\ell(w)=k+\sum\limits_{r=1}^{k+1} \ell(w_r)$ and $\ell(w')=\sum\limits_{r=1}^{k+1}\ell(w_r)=\ell(w)-k$.

Then, by the assumption of the lemma, the pair $(s_iw', s_iw)$ satisfies \eqref{eq:the characterization of strong Bruhat default} because $\ell(s_iw_1)=\ell(w_1)+1$, 
 hence $(s_iw',s_iw)\in {\mathcal R}_W$.

The lemma is proved.
\end{proof}

Furthermore, we prove \eqref{eq:implication} by induction in $\ell(g)$. If $\ell(g)=0$, i.e., $g=1$, then $\partial_{1,h}=\delta_{1,h}$ and we have nothing to prove. 
Suppose that $\ell(g)\ge 1$, 
i.e.,  $\ell(s_ig)=\ell(g)-1$ for some $i\in I$. 

We need the following result. 

\begin{lemma} 
\label{le:conjugation vs action} 
For each Coxeter group $W$ one has the following symmetries of $\hat {\bf D}(W)$: 

\noindent (a) 
$\hat{\bf D}(W)$ is a $\ZZ W$-module algebra via $w(D_s)=\begin{cases} 
D_{wsw^{-1}} & \text{if $\ell(ws)>\ell(w)$}\\
1-D_{wsw^{-1}} & \text{if $\ell(ws)<\ell(w)$}\\
\end{cases}$
for $w\in W$, $s\in {\mathcal S}$.

 
\noindent (b) The $\ZZ$-linear transformation $d_i$ given by $d_i(x):=s_i(x)\cdot s_i-s_ix$
for  $x\in \hat {\bf D}(W)$ and $i\in I$, is an $s_i$-derivation $d_i$ of $\hat {\bf D}(W)$ determined by $d_i(D_s)=\delta_{s,s_i}$.

\end{lemma}

\begin{proof} Prove (a). Theorem \ref{th:ZWcross and free D} and the fact that $(w(D_s))^2=w(D_s)$ for $w\in W$, $s\in {\mathcal S}$  imply that the assignment $x\mapsto w(x)$ for $x\in \hat {\bf D}(W)$ is an algebra automorphism for any $w\in W$. It suffices to show that $w_1(w_2(x))=(w_1w_2)(x)$ for all $x\in \hat {\bf D}(W)$, $w,w'\in W$. Since the involved maps are automorphisms, it suffices to do so only on generators $x=D_s$, $s\in {\mathcal S}$. Indeed, $w(D_s)= \sigma_{w,s}+\chi_{w,s}D_{wsw^{-1}}$  
for $w\in W$, $s\in {\mathcal S}$ and $\chi,\sigma$ given by \eqref{eq:chi sigma defined}.

Then  $w_1(w_2(D_s))= w_1(\sigma_{w_2,s}+\chi_{w_2,s}D_{w_2sw_2^{-1}})=\sigma_{w_2,s}+\chi_{w_2,s}w_1(D_{w_2sw_2^{-1}})$
$$=\sigma_{w_2,s}+\chi_{w_2,s}\chi_{w_1,w_2sw_2^{-1}}D_{w_1w_2sw_2^{-1}w_1^{-1}}+\chi_{w_2,s}\sigma_{w_1,w_2sw_2^{-1}}=\sigma_{w_1w_2,s}+\chi_{w_1w_2,s}D_{w_1w_2sw_2^{-1}w_1^{-1}}$$
for all $w\in W$, $s\in {\mathcal S}$, by \eqref{eq:2 cocycle}. 
This proves (a).

Prove (b). Indeed, 
$$d_i(xy)=s_i(xy)s_i-s_ixy=(s_i(x)s_i-s_ix)y+s_i(x)(s_i(y)s_i-s_iy)=d_i(x)y+s_i(x)d_i(y)$$ for all $x,y\in \hat {\bf D}(W)$, $i\in I$. 
Also, $d_i(D_s)=s_i(D_s)s_i-s_iD_s=\delta_{s,s_i}$ for all $s\in {\mathcal S}$, $i\in I$ because 
$$s_iD_s=\begin{cases}
D_{s_iss_i}s_i & \text{if $s\ne s_i$}\\
s_i-1-D_is_i & \text{if $s=s_i$}
\end{cases}, s_i(D_s)= \begin{cases}
D_{s_iss_i} & \text{if $s\ne s_i$}\\
1-D_i & \text{if $s=s_i$}
\end{cases}.$$
This proves (b).
%
%
%

The lemma is proved.
\end{proof}

Furthermore, if $g=s_i$ for $i\in I$, then Lemma 
\ref{le:conjugation vs action} guarantees that $\partial_{s_i,h}=0$ iff $h\notin \{1,s_i\}$ and $\partial_{s_i,1}=d_i$, $\partial_{s_i,s_i}$ is the action of $s_i$. Together with  Proposition \ref{pr:factored derivatives} these imply that 
\begin{equation}
\label{eq:partial recursion}
\partial_{g,h}(x)=\sum\limits_{h_1,h_2\in W:h_1h_2=h}
\partial_{s_i,h_1}(\partial_{s_ig,h_2}(x))=s_i(\partial_{s_ig,s_ih}(x))+d_i(\partial_{s_ig,h}(x))
\end{equation}
for all $x\in \hat {\bf D}(W)$, $g,h\in W$, $i\in I$ such that $(s_ig,g)\in {\mathcal R}_W$.

In particular, for a given $i\in I$, $g,h\in W$ such that $\ell(s_ig)=\ell(g)-1$, i.e., $(s_ig,g)\in {\mathcal R}_W$, the equation  \eqref{eq:partial recursion} guarantees that $\partial_{g,h}\ne 0$ implies that either $\partial_{s_ig,h}\ne 0$ or $\partial_{s_ig,s_ih}\ne 0$. 
Using the inductive hypothesis, we obtain the implication:
\begin{equation}
\label{eq:implication1}
\partial_{g,h}\ne 0 ~=> \text{either}~ (s_ih,s_ig)\in {\mathcal R}_W~\text{or}~(h,s_ig)\in {\mathcal R}_W
\end{equation}

Clearly, if $(h,s_ig)\in {\mathcal R}_W$ in \eqref{eq:implication1}, then \eqref{eq:implication} holds by transitivity. If $(s_ih,s_ig)\in {\mathcal R}_W$ in \eqref{eq:implication1} and $\ell(s_ih)=\ell(h)-1$, then $(h,g)\in {\mathcal R}_W$ by Lemma \ref{le:characterization of Bruhat order}.

It remains to consider the case $(s_ih,s_ig)\in {\mathcal R}_W$, $\ell(s_ih)=\ell(h)+1$. Indeed, $(h,s_ih), (s_ig,g)\in {\mathcal R}_W$ hence $(h,g)\in {\mathcal R}_W$ by transitivity 
The implication \eqref{eq:implication} is proved.

Finally, let us prove the claim that $\partial_{g,g}(x)=g(x)$ for all $g\in W$, $x\in \hat {\bf D}(W)$. Once again, we proceed by induction in $\ell(g)$.
If $\ell(g)=0$, i.e., $g=1$, then  we have nothing to prove. 
Suppose that $\ell(g)\ge 1$, 
i.e.,  $\ell(s_ig)=\ell(g)-1$ for some $i\in I$. Taking into account that $(g,s_ig)\notin {\mathcal R}_W$ hence $\partial_{s_ig,g}=0$ by 
\eqref{eq:implication}, \eqref{eq:partial recursion} implies that 
$\partial_{g,g}(x)=s_i(\partial_{s_ig,s_ig}(x))$
for all $x\in \hat {\bf D}(W)$. Using the inductive hypothesis in the form $\partial_{s_ig,s_ig}(x)=s_ig(x)$ for $x\in \hat {\bf D}(W)$, we obtain:
$\partial_{g,g}(x)=s_i(\partial_{s_ig,s_ig}(x))=s_i(s_ig(x))=g(x)$, 
which proves the claim.

This finishes the proof of Proposition \ref{pr:refined factored derivatives}.
\end{proof} 

Prove (a) now. Indeed,  $(w,s_i)\in {\mathcal R}_W$ iff $W\in \{1,s_i\}$. Therefore, the defining relations \eqref{eq:defining relations universal Nichols} read
\begin{equation}
\label{eq:defining relations universal Nichols s_i}
d_{s_ig,h}=d_{s_i,1}d_{g,h}+d_{s_i,s_i}d_{g,s_ih},~d_{gs_i,h}=d_{g,h}d_{s_i,1}+d_{g,hs_i}d_{s_i,s_i}
\end{equation}
for all $i\in I$, $g,h\in W$. 

In particular, if $g=s_i,h=s_i$, we obtain:
\begin{equation}
\label{eq:defining relations universal Nichols s_i,1}
d_{s_i,1}d_{s_i,s_i}+d_{s_i,s_i}d_{s_i,1}=0
\end{equation}
because $d_{1,s_i}=0$.

Taking  $g=h$, such that $s_ig=gs_{i'}$ and $\ell(s_ig)>\ell(g)$ for some $i,i'\in I$, \eqref{eq:defining relations universal Nichols s_i} implies that  $d_{s_i,1}d_{g,g}=d_{g,g}d_{s_{i'},1}$. In particular, taking $g=\underbrace{s_js_i\cdots s_{j'}}_{m_{ij}-1}$ whenever $m_{ij}\ge 2$ in the notation of Definition \ref{def:H_0 Coxeter}, we have $s_ig=gs_{i'}$  and we obtain:
\begin{equation}
\label{eq:defining relations universal Nichols s_i s_j} 
\underbrace{d_{s_i,1}d_{s_j,s_j}d_{s_i,s_i}\cdots d_{s_{j'}s_{j'}}}_{m_{ij}} =\underbrace{d_{s_j,s_j}\cdots d_{s_{i'}s_{i'}}d_{s_{j'}s_{j'}}d_{s_{i'},1}}_{m_{ij}}
\end{equation}

The relations \eqref{eq:defining relations universal Nichols s_i,1} and \eqref{eq:defining relations universal Nichols s_i s_j} guarantee that \eqref{eq:nichols graded homomorphism} defines a  homomorphism $\hat \varphi_W$ of algebras. The relations \eqref{eq:defining relations universal Nichols s_i} guarantee that $\underline {\bf B}(W,{\mathcal R}_W)$ is generated by $d_{s_i,1}$ and $d_{s_i,s_i}$, $i\in I$ hence the homomorphism 
\eqref{eq:nichols graded homomorphism} is surjective. 

Finally, taking into account that $\Delta(d_{s_i,s_i})=d_{s_i,s_i}\otimes d_{s_i,s_i}$ and $\Delta_{d_{s_i,1}}=d_{s_i,1}\otimes d_{1,1}+d_{s_i,s_i}\otimes d_{s_i,1}$ for $i\in I$ and $d_{1,1}=1$, we see that $\hat \varphi_W$ is a homomorphism of Hopf algebras.

Let us prove the second assertion of (a). First, show that $d_{w,w}\ne 1$ in ${\mathcal B}(W,{\mathcal R}_W)$ 
for each  $w\in W\setminus \{1\}$. By the construction, $\hat \varphi_W(w)=d_{w,w}$ for all $w\in W$. 

We need the following fact.
\begin{lemma} 
\label{le:d's}
For each $s\in {\mathcal S}$ there is a unique nonzero element $d'_s\in {\mathcal B}(W,{\mathcal R}_W)$ such that  
$d'_s=\hat \varphi_W(d_s)$ for $s\in {\mathcal S}$ and $d'_{ws_iw^{-1}}=\chi_{w,s_i}d_{w,w}d'_{s_i}d_{w,w}^{-1}$ for any $w\in W$, $i\in I$.

\end{lemma}

\begin{proof} The uniqueness follows from the fact that $d_s$ is determined uniquely by same property and  $\hat \varphi_W$ is a homomorphism of algebras. The fact that $d_s'\ne 0$ follows from that $d'_{s_i}=-d_{s_1,1}\ne 0$ for all $i\in I$, 
which, in turn, follows from Corollary \ref{cor:partial action} and Proposition  \ref{pr:refined factored derivatives} since $\partial_{s_i,1}(D_i)=-1$ for $i\in I$.

The lemma is proved. \end{proof}
 
Suppose that $d_{w,w}=1$ for some $w$. Lemma \ref{le:d's} implies that $d'_{ws_iw^{-1}}=-d_{w,w}d'_{s_i}d_{w,w}^{-1}$ in $\underline {\bf B}(W,{\mathcal R}_W)$ for  $i\in I$ such that 
$\ell(ws_i)=\ell(w)-1$ hence 
$ws_iw^{-1}=s_i$  and $w=1$.

This proves that $d_{w,w}\ne 1$ for each $w\in W$, $w\ne 1$ hence $d_{w,w}\ne d_{w',w'}$ if $w\ne w'$.

Finally, since the $\ZZ$-linear span of all $d_{w,w}$ is a sub-bialgebra of  $\underline {\bf B}(W,{\mathcal R})$ and each $d_{w,w}$ is grouplike, then the set $\{d_{w,w}\,|\,w\in W\}$ is $\ZZ$-linearly independent.
This proves the second assertion and finishes the proof of (a). 

Prove Theorem \ref{th:nichols graded homomorphism}(b). The presentation \eqref{eq:presentation Nichols} implies (by induction in length) that ${\mathcal B}(W,{\mathcal R}_W)$ is generated by all $v_{s_i,1}^g$, i.e., by all $d'_s$, i.e., by $\hat \varphi(Y_W)$. This implies that $\underline {\hat \varphi}$ is surjective. Also $\underline {\hat \varphi}$ commutes with $W$-action and preserves $W$-grading, therefore, it is a homomorphism of algebras in $_W^W{\mathcal YD}$. In turn, this implies that $\underline {\hat \varphi}\otimes \underline {\hat \varphi}$ is a well-defined surjective homomorphism of algebras  $\hat {\bf H}_0(W)\otimes \hat {\bf H}_0(W)\twoheadrightarrow {\mathcal B}(W,{\mathcal R}_W)\otimes {\mathcal B}(W,{\mathcal R}_W)$. 
Note also that $\underline \Delta(v_{s_i,1}^g)=v_{s_i,1}^g\otimes 1+1\otimes v_{s_i,1}^g$ for all $g\in W$, $i\in I$ by \eqref{th:presentation Nichols}, that is, $\underline \Delta(d'_s)=d'_s\otimes 1+1\otimes d'_s$ for any $s\in {\mathcal S}$. This and the above imply that  
$$\underline \Delta\circ \underline {\hat \varphi}=(\underline {\hat \varphi}\otimes \underline {\hat \varphi})\circ \underline \Delta\ .$$
It is also immediate that $\underline \varepsilon(d_{s'})=0$ for all $s\in {\mathcal S}$ hence $\underline \varepsilon\circ \underline {\hat \varphi}=\underline \varepsilon$. 

Finally, note that since  $d'_s\ne 0$ by Lemma \ref{le:d's} and $|d'_s|=s$ for all $s\in {\mathcal S}$, the set $\{d'_s\,|\,s\in {\mathcal S}\}$ is $\ZZ$-linearly independent.
This finishes the proof of (b).

 
Theorem \ref{th:nichols graded homomorphism} is proved.
\end{proof}

\begin{remark} Theorem \ref{th:nichols graded homomorphism}(b) asserts that ${\mathcal B}(W,{\mathcal R}_W)$ is essentially a pre-Nichols algebra, however, we are not yet aware of existence of the braided antipode in ${\mathcal B}(W,{\mathcal R}_W)$.

\end{remark}

\begin{definition}
\label{def:tilde H}
Let $W$ be a simply-laced Coxeter group. Denote by ${\bf H}_0(W)$ the $\ZZ$-algebra  generated by $s_i,d_i$, $i=1,\ldots,n-1$ subject to relations: 

$\bullet$   $s_id_i+d_is_i=0$, $d_i^2=0$ for $i\in I$.

$\bullet$   $s_is_j=s_js_i$ $d_js_i=s_id_j$, $d_jd_i=d_id_j$ for all $i,j\in I$ with $m_{ij}=2$.

$\bullet$   $s_js_is_j=s_is_js_i$, $s_jd_is_j=s_id_js_i$, $d_js_id_j=s_id_jd_i+d_id_js_i$ for all $i,j\in I$ with $m_{ij}=3$.

\end{definition} 

That is, the simply-laced ${\bf H}_0(W)$ is obtained by ``homogenizing" Theorem \ref{th:hopf hat intro W} and  is naturally a Hopf algebra.
In particular, the canonical surjective algebra homomorphism $\hat {\bf H}_0(W)\twoheadrightarrow {\bf H}_0(W)$ is that of Hopf algebras.

The following is an immediate graded version of Proposition   \ref{pr:presentation simply-laced D}. 

\begin{lemma} 
\label{le:Nichols S_n} For any simply-laced Coxeter group $W$ one has:

\noindent (a) the algebra ${\bf H}_0(W)$ is isomorphic to the cross product ${\bf D}_0(W)\rtimes \ZZ W$, where  ${\bf D}_0(W)$ is the $\ZZ$-algebra generated by $d_s$, $s\in {\mathcal S}$, 
subject to relations (in the notation of Proposition \ref{pr:presentation simply-laced D}):

$\bullet$ $d_s^2=0$ for all $s\in {\mathcal S}$.

$\bullet$ $d_sd_{s'}=d_{s'}d_s$ for all compatible pairs $(s,s')\in {\mathcal S}\times {\mathcal S}$ with $m_{s,s'}=2$.

$\bullet$ $d_sd_{s'}=d_{ss's}d_s+d_{s'}d_{ss's}$ for all compatible pairs $(s,s')\in {\mathcal S}\times {\mathcal S}$ with $m_{s,s'}=3$.

\noindent (b) ${\bf D}_0(W)$ is a (braided) Hopf algebra in the category $_W^W{\mathcal YD}$ so that the canonical surjective homomorphism $\hat {\bf D}_0(W)\twoheadrightarrow {\bf D}_0(W)$ is that of braided Hopf algebras.
\end{lemma} 

\begin{remark} In view of Remark \ref{rem:pre-Nichols}, for any simply-laced Coxeter group $W$ the Hopf algebra ${\bf D}_0(W)$ is a pre-Nichols algebra  of the  Yetter-Drinfeld module $\underline Y_W=\oplus_{s\in {\mathcal S}} \ZZ\cdot d_s$ over $W$   so that the canonical surjective homomorphism $\hat {\bf D}_0(W)\twoheadrightarrow {\bf D}_0(W)$ is that of pre-Nichols algebras. 

\end{remark}

\begin{remark} 
\label{rem:FK}
The algebra ${\bf D}_0(S_n)$ coincides with the Fomin-Kirillov algebra ${\mathcal E}_n$ defined in \cite{FK}.

\end{remark}

\begin{theorem} For any simply-laced Coxeter group $W$  the homomorphism \eqref{eq:nichols homomorphism} factors through the following surjective homomorphism of bialgebras in $_W^W{\mathcal YD}$.
\begin{equation}
\label{eq:D0 to BW}
{\bf D}_0(W)\twoheadrightarrow \mathcal{\bf B}(W,{\mathcal R}_W) \ .
\end{equation}
\end{theorem}

\begin{proof} First, prove that \eqref{eq:nichols graded homomorphism} factors through the homomorphism of bialgebras
\begin{equation}
\label{eq:H0 to BW}
{\bf H}_0(W)\twoheadrightarrow \underline {\bf B}(W,{\mathcal R}_W) \ .
\end{equation}

Clearly, for distinct $i,j\in I$ we have by 
\eqref{eq:defining relations universal Nichols s_i}:
\begin{equation}
\label{eq:dsisj}
d_{s_is_j,1}
=d_{s_i,1}d_{s_j,1}, d_{s_is_j,s_i}
=d_{s_i,s_i}d_{s_j,1},~ d_{s_is_j,s_j}=d_{s_i,1}d_{s_j,s_j}
\end{equation}
because $d_{s_j,s_i}=d_{s_j,s_i}=d_{s_j,s_is_j}=0$.

Let $i,j\in I$ be such that $m_{ij}=2$, i.e., $s_is_j=s_js_i$. Then using  \eqref{eq:dsisj} (also with  $i$ and $j$ interchanged where necessary), we obtain 
\begin{equation}
\label{relation d 2}
d_{s_i,1}d_{s_j,1}=d_{s_j,1}d_{s_i,1},~d_{s_i,s_i}d_{s_j,1}=d_{s_j,1}d_{s_i,s_i}
\end{equation}

Now let $m_{ij}=3$ and let $s_{ij}:=s_is_js_i=s_js_is_j$.

Indeed, let us compute $d_{s_{ij},1}$  
  in two ways using \eqref{eq:defining relations universal Nichols s_i} and \eqref{eq:dsisj} (also interchanging $i$ and $j$ where necessary). We obtain:
$$d_{s_{ij},s_i}=d_{s_is_js_i,s_i}=d_{s_i,1}d_{s_js_i,s_i}+d_{s_i,s_i}d_{s_js_i,1}=d_{s_i,1}d_{s_j,1}d_{s_i,s_i}+d_{s_i,s_i}d_{s_j,1}d_{s_i,1}$$
$$d_{s_{ij},s_i}=d_{s_js_is_j,s_i}=d_{s_j,1}d_{s_is_j,s_i}+d_{s_j,s_j}d_{s_is_j,s_js_i}=d_{s_j,1}d_{s_is_j,s_i}=d_{s_j,1}d_{s_j,s_j}d_{s_i,1}$$

Therefore, 
$d_{s_j,1}d_{s_j,s_j}d_{s_i,1}=d_{s_i,1}d_{s_j,1}d_{s_i,s_i}+d_{s_i,s_i}d_{s_j,1}d_{s_i,1}$.
Clearly, the above relation and \eqref{relation d 2}  ensure that \eqref{eq:H0 to BW} is a well-defined homomorphism of algebras. Clearly, it commutes with the coproduct, the counit and the antipode, so is a homomorphism of Hopf algebras. 
%
%
%

Then, copying the argument of the proof of Theorem \ref{th:nichols graded homomorphism}(b), we conclude that \eqref{eq:D0 to BW} is surjective, commutes with the $W$-action,  preserves $W$-grading, braided coproduct and the braided counut.

The theorem is proved.
\end{proof}

\begin{remark}
In  \cite[Section 6]{MS} the authors conjectured that ${\bf D}_0(S_n)$ is, in fact,  a Nichols algebra. In turn, this would imply that \eqref{eq:D0 to BW} is an isomorphism for $W=S_n$, $n\ge 2$.
So is natural to ask whether \eqref{eq:D0 to BW} is  an isomorphism for each simply-laced Coxeter group $W$.

\end{remark}
We conclude the section with a (conjectural) generalization of \eqref{eq:D0 to BW} to all Coxeter groups as follows. Define a filtration on $\hat {\bf H}(W)$ by assigning the filtered degree $1$ to each $D_i$ and $0$ to each $s_i$. 

The following is an immediate consequence of Theorems \ref{th:ZWcross and free D} and \ref{th:hopf hat intro W}.

\begin{lemma}  For any Coxeter group $W$ one has:

(a) The assignments $d_s\mapsto D_s$, $s\in {\mathcal S}$, define a natural isomorphism  of graded algebras
$$\hat gr_W:\hat {\bf H}_0(W) \widetilde \to gr~\hat {\bf H}(W)\ ,$$
where $gr~\hat {\bf H}(W)$ is the associated graded of $\hat {\bf H}(W)$.

(b) If $W$ is simply laced, then $\hat gr_W$ factors through a surjective homomorphism ${\bf D}_0(W)\twoheadrightarrow gr~ {\bf D}(W)$ of $W$-graded algebras. 

\end{lemma}

\begin{remark} For any Coxeter group $W$  the composition of $\hat gr_W^{-1}$ with \eqref{eq:nichols graded homomorphism} is a surjective homomorphism $
gr~\hat {\bf D}(W) \twoheadrightarrow {\mathcal B}(W,{\mathcal R}_W)
$ of bialgebras in $_W^W{\mathcal YD}$.
We expect this homomorphism to factor through the surjective homomorphism of bialgebras in $_W^W{\mathcal YD}$:
$gr~ {\bf D}(W) \twoheadrightarrow {\mathcal B}(W,{\mathcal R}_W)$.

\end{remark}

\section{Hecke-Hopf algebras of cyclic groups and generalized Taft algebras}


\label{sec:Taft algebras}
 
In this section we study a variant of the generalized Hecke-Hopf algebra for cyclic groups. In fact, these Hopf algebras are  bialgebras  universally coacting (in the sense of \cite{BBHcross}) on finite dimensional principal ideal domains. 
It turns out that the actual (generalized) Hecke-Hopf algebra of a cyclic group is the quotient of such a universal Hopf algebra and is isomorphic to the Taft algebras.

Let $R$ be a commutative unital ring and let  $f\in R[x]\setminus \{0\}$. Denote by ${\bf H}_f$ the $R$-algebra generated by $s,D$ subject to relations $s^{\deg f}=1$ and the relations given by the functional equation 
\begin{equation}
\label{eq:fstD=ft}
f(ts+D)=f(t)
\end{equation}
over $\kk[t]$ (with the convention that if $\deg f=0$, then $s$ is of infinite order). 

In other words, if we write $f=a_0+a_1x+\cdots +a_n x^n$, $a_0,\ldots,a_n\in R$, $a_n\ne 0$, then ${\bf H}_f$ is subject to relations
$\sum\limits_{r=k}^n a_r\{s,D\}_{k,r-k}=a_k$
for $k=0,\ldots,n$, where $\{{\bf a},{\bf b}\}_{k,r-k}=\{{\bf b},{\bf a}\}_{r-k,k}$ denotes the coefficient of $t^k$ in the expansion of the noncommutative binomial $({\bf a}t+{\bf b})^r$ (that is,
$\{{\bf a},{\bf b}\}_{k,r-k}=\sum\limits_{\substack{\varepsilon_1,\ldots,\varepsilon_r\in \{0,1\}:\\
\varepsilon_1+\cdots +\varepsilon_r=k}} {\bf a}^{\varepsilon_1}{\bf b}^{1-\varepsilon_1}\cdots {\bf a}^{\varepsilon_r}{\bf b}^{1-\varepsilon_r}$).
Clearly, ${\bf H}_{cf+d}={\bf H}_f$ for any $d\in R$ and  $c\in R^\times$. 

\begin{example}

$~$

$\bullet$ $f(x)=x+a_0$. Then ${\bf H}_f=R$.

$\bullet$ $f(x)=x^2+a_1x+a_0$. Then ${\bf H}_f$ is generated by $s$ and $D$ subject to relations $s^2=1$, $D^2=-a_1D$, $sD + Ds=a_1(1-s)$. In particular, if $a_1=-1$, then ${\bf H}_f={\bf H}(S_2)$ by Definition \ref{def:hat H}.

$\bullet$ $f(t)=x^3+a_2x^2+a_1x+a_0$. Then ${\bf H}_f$ is generated by $s$ and $D$ subject to relations $s^3=1$ and
$$D^3=-a_2D^2-a_1D,~s^2D+sDs+Ds^2=a_2(1-s^2),~D^2s+DsD+sD^2+a_2(sD+Ds)=a_1(1-s)\ .$$

\end{example}

\begin{proposition} 
\label{pr:cyclic Hecke-Hopf}
For each  $f(x)\in R[x]$,  ${\bf H}_f$ is a Hopf algebra over $R$  with the coproduct, the counit, and the antipode given respectively by 
\begin{equation}
\label{eq:pre-taft hopf}
\Delta(D)=D\otimes 1+s\otimes D,~\Delta(s)=s\otimes s,~\varepsilon(D)=0,\varepsilon(s)=1, S(s)=s^{-1},S(D)=-s^{-1}D\ .
\end{equation}

\end{proposition}

\begin{proof} Denote by ${\bf H}'$ the free product (over $R$) of the cyclic group algebra $R[s]/(s^n-1)$, where $n:=\deg f$, 
and the polynomial algebra $R[D]$. By Lemma \ref{le:free Hopf} taken with $H=R[s]/(s^n-1)$, $V=R D$ and $\delta(D)=s\otimes D$, ${\bf H}'$ is  a Hopf algebra with  coproduct, counit, and  antipode as in \eqref{eq:pre-taft hopf}. 



Let $y_k\in {\bf H}'$, $k=0,\ldots,n$ be the coefficients in the expansion 
$f(t s +D)=\sum\limits_{k=0}^n  y_k t^k$.
In fact, 
\begin{equation}
\label{eq:yk of I_f}
y_k=\sum\limits_{i=k}^n a_i\{s,D\}_{i,k-i}
\end{equation}
for $k=0,\ldots,n=\deg f$, where $\sum\limits_{k=0}^n a_kt^k=f(t)$.

Denote by ${\bf K}_f$ the $R$-submodule of ${\bf H}'$ generated by $1$ and $y_0,\ldots,y_{n-1}$. Let ${\bf H}'[t]={\bf H}'\otimes_R R[t]$, which, clearly, is a Hopf algebra over $R[t]$.  

%

\begin{lemma} 
\label{le:right coideal Hf}
The $R$-module ${\bf K}_f$ is a right coideal in ${\bf H}'$. 

\end{lemma}

\begin{proof}
We have in ${\bf H}'[t]$:
$$\Delta(f(st+D))=f(ts\otimes s+D\otimes 1+s\otimes D)=f(s\otimes (st+D)+D\otimes 1)=f(s't'+D')$$
where $s'=s\otimes 1$, $t'=1\otimes (st+D)$, $D'=D\otimes 1$. 
Taking into account that the assignment $s\mapsto s'$, $D\mapsto D'$ is an algebra homomorphism ${\bf H}'\to {\bf H}'\otimes 1$, we obtain
$$\Delta(f(st+D))=\sum_{k=0}^n y'_k t'^k=\sum_{k=0}^n y_k\otimes  (st+D)^k\subset {\bf K}_f\otimes {\bf H}'[t]$$
where $y'_k:=y_k\otimes 1$. 
%
This implies that $\delta(y_k)\in {\bf K}_f\otimes {\bf H}'$ for $k=0,\ldots,n=\deg f$. 

The lemma is proved.
\end{proof}

%
%
%
%
%
%

Finally, note that ${\bf K}_f^+={\bf K}_f\cap Ker~\varepsilon=\sum\limits_{k=0}^n R\cdot(y_k-a_k)$, i.e., ${\bf K}_f^+$ is an $R$-submodule of ${\bf H}'$ generated by all coefficients of $f(ts+D)-f(t)$.  
In view of Proposition \ref{pr:from coideal to hopf ideal}, this implies that the ideal ${\bf J}_f$ generated by $y_k-a_k$, $k=0,\ldots,n$, is a Hopf ideal in ${\bf H}'$, i.e., ${\bf H}_f={\bf H}'/{\bf J}_f$ is a Hopf algebra. 

Proposition \ref{pr:cyclic Hecke-Hopf} is proved. \end{proof}

The following is an analogue of Theorems \ref{th:Hecke in Hecke-Hopf} and \ref{th:upper Hecke in Hecke-Hopf}.
\begin{proposition} 
\label{pr:cyclic Hecke-Hopf coideal} 
For an $R$-algebra $\kk$, $c\in \kk$ and any $f\in \kk[x]$  the assignment $x\mapsto cs+D$ defines a homomorphism of algebras 
\begin{equation}
\label{eq:covariant H_f coaction}
\varphi_c:\kk[x]/(f-f(c))\to {\bf H}_f\otimes_R \kk
\end{equation}
whose image is a left coideal subalgebra in ${\bf H}_f$.

\end{proposition}

\begin{proof} Indeed, defining functional relations \eqref{eq:fstD=ft} imply that $f(cs+D)-f(c)=0$. This proves that $\varphi_c$ is a homomorphism of algebras. 

Since $\underline x:=\varphi_c(x)=cs+D$ and $\Delta(\underline x)=c s\otimes s+D\otimes 1+s\otimes D=D\otimes 1+s\otimes \underline x$. Thus, $R\cdot \underline x$ is a left coideal in ${\bf H}_f$ hence the subalgebra of ${\bf H}_f$ generated by $\underline x$ is a left coideal subalgebra in ${\bf H}_f$.

%

Proposition \ref{pr:cyclic Hecke-Hopf coideal} is proved.
\end{proof}

\begin{remark} We expect that \eqref{eq:covariant H_f coaction} is always injective.

\end{remark}

For $a,b\in R$ denote by ${\bf H}_f(a,b)$ the $R$-algebra generated by $D,s$ subject to relations $s^{\deg~f}=1$, the functional relations \eqref{eq:fstD=ft}, and $sDs^{-1}=aD+b(1-s)$.

\begin{proposition} 
\label{pr:modified Taft Hopf} 
For any nonzero $f\in R[x]$ and $a,b\in R$, the algebra ${\bf H}_f(a,b)$ is a Hopf algebra with the coproduct $\Delta$, counit $\varepsilon$, and the antipode $S$ given respectively by:
$$\Delta(s)=s\otimes s,~\Delta(D)=D\otimes 1+s\otimes D,~\varepsilon(s)=1,~\varepsilon(D)=0,~S(s)=s^{-1},~S(D)=-s^{-1}D\ .$$

\end{proposition}

\begin{proof} 
 Let ${\bf K}_{a,b}$ be the $R$-submodule of ${\bf H}_f$ generated by $1$ and $sDs^{-1}-aD+bs$.
%

We need the following result.
\begin{lemma} 
\label{le:coideal deltaab} ${\bf K}_{a,b}$ is a left coideal in ${\bf H}_f$. 


\end{lemma}

\begin{proof} Indeed, let $\delta:=sDs^{-1}-aD+bs$. Then
$\Delta(\delta)=(s\otimes s)\Delta(D)(s^{-1}\otimes s^{-1})-a\Delta(D)+bs\otimes s$
$$=sDs^{-1}\otimes 1+s\otimes sDs^{-1}-aD\otimes 1-s\otimes aD+bs\otimes s=(\delta-bs)\otimes 1+s\otimes \delta\in {\bf H}_f\otimes {\bf K}_{a,b} \ .$$
The lemma is proved. \end{proof}

Finally, note that ${\bf K}_{a,b}^+={\bf K}_{a,b}\cap Ker~\varepsilon=R\cdot \delta_{a,b}$, where $\delta_{a,b}=sDs^{-1}-aD-b(1-s)$. 
In view of Proposition \ref{pr:from coideal to hopf ideal}, this guarantees that the ideal ${\bf J}_f$ generated by $\delta_{a,b}$ is a Hopf ideal in ${\bf H}_f$. 
Hence the quotient ${\bf H}_f(a,b)={\bf H}_f/{\bf J}_{a,b}$ is a Hopf algebra. 
%

The proposition is proved. 
\end{proof}

%
%
 %
%
%
%

We abbreviate ${\bf H}_n(a,b):={\bf H}_{f_n^{a,b}}(a,b)$ for $a,b\in R$, where  
\begin{equation}
\label{eq:fnab}
f_n^{a,b}:=x(x-b)(x-b(1+a))\cdots (x-b(1+a+\cdots+a^{n-2}))
\end{equation}

Note that if $a\in R\setminus\{1\}$ is a root of unity, i.e., $1+a+\cdots+a^{n-1}=1$, then
\begin{equation}
\label{eq:fab symmetry}
f_n^{a,b}(ax+b)=f_n^{a,b}(x)
\end{equation}
because the set of roots of $f_n^{a,b}$ is invariant under the linear change $x\mapsto ax+b$.

We call ${\bf H}_n(a,b)$  a {\it generalized Taft algebra}. 
This terminology is justified by 
the following result.

\begin{proposition} 
\label{pr:presentation taft}
Given a commutative unital ring $R$ and $a,b\in R$,  the Hopf algebra ${\bf H}_n(a,b)$ has a presentation: 
$s^n=1$, $sDs^{-1}=a D+b(1-s)$, and 
\begin{equation}
\label{eq:taft relations generalized}
{n \brack k}_a D(aD+b)(a^2D+b(1+a))\cdots (a^{k-1}D+b(1+a+\cdots +a^{k-2}))=0
\end{equation}
for $k=1,\ldots,n$, where $\displaystyle{{n \brack k}_q=\prod\limits_{i=1}^k \frac{q^{n+1-i}-1}{q^i-1}}\in \ZZ_{\ge 0}[q]$ 
is the $q$-binomial coefficient.

In particular, if $a$ is a primitive $n$-th root of unity in $R^\times$, then ${\bf H}_n(a,b)$ has a presentation:
$$s^n=1,~sDs^{-1}=a D+b(1-s), D(aD+b)(a^2D+b(1+a))\cdots (a^{n-1}D+b(1+a+\cdots +a^{n-2}))=0$$
\end{proposition}

\begin{proof}  We need the following result.

\begin{lemma} The algebra ${\bf H}_n(a,b)$ has a presentation
$s^n=1$, $sDs^{-1}=a D+b(1-s)$, and the functional relations
\begin{equation}
\label{eq:modified Taft}
(t-D)(t-aD-b)(t-a^2D-b(1+a))\cdots (t-a^{n-1}D-b(1+a+\cdots +a^{n-2}))=f_n^{a,b}(t) \ .
\end{equation}

\end{lemma}

\begin{proof} Applying the antipode to the defining functional relations \eqref{eq:fstD=ft} we see that ${\bf H}_{f_n^{a,b}}$ 
has a presentation: $s^n=1$ and
$$(s^{-1}t-s^{-1}D-b(1+a+\cdots+a^{n-2}))\cdots (s^{-1}t-s^{-1}D-b)(s^{-1}t-s^{-1}D)=f_n^{a,b}(t)\ .$$
Equivalently, factoring out $s^{-1}$ to the left from each factor, we obtain $s^n=1$ and:
\begin{equation}
\label{eq:modified Taft intermediate}
(t-s^{n-1}Ds^{1-n}-b(1+a+\cdots+a^{n-2})s)\cdots (t-sDs^{-1}-bs)(t-D)=f_n^{a,b}(t)
\end{equation}
Passing to ${\bf H}_n(a,b)$, we obtain one more defining relation $sDs^{-1}=aD+b(1-s)$, which immediately implies 
$s^kDs^{-k}=a^kD+b(1+a+\cdots+a^{k-1})(1-s)$ for $k\in \ZZ_{\ge 0}$. 
Taking this into account, we see that the left hand side of \eqref{eq:modified Taft intermediate} becomes the left hand side of \eqref{eq:modified Taft}. The lemma is proved.
\end{proof}

We need the following combinatorial fact. For $n\ge 0$ let $f_n(t,x;p,q)\in \ZZ[t,x,p,q]$ be given by
$$f_n(t,x,p,q)=\prod_{i=0}^{n-1}\left (t+q^ix+p\frac{q^i-1}{q-1}\right)$$
with the convention that $f_0(t,x;p,q)=1$.

The following is a generalization of the $q$-binomial formula.

\begin{lemma} 
\label{le:generalized binomial}
$\displaystyle{f_n(t,x;p,q)=
\sum\limits_{k=0}^n 
{n \brack k}_q
f_k(0,x;p,q)}f_{n-k}(t,0;p,q) $ for $n\ge 0$.

\end{lemma}

Applying Lemma \ref{le:generalized binomial} with $q=a$, $p=-b$, $x=-D$, we see that the left hand side of \eqref{eq:modified Taft} equals $f_n(t,-D;-b,a)$ and $f_n(t,0;-b,a)=f^{a,b}_n(t)$, so \eqref{eq:modified Taft} becomes:
$\sum\limits_{k=1}^n {n \brack k}_a f_k(0,-D;-b,a)f_{n-k}^{a,b}(t) =0$.

Finally, using ${\bf H}_n(a,b)$-linear independence of $f_k^{a,b}(t)$, $k\ge 0$ in ${\bf H}_n(a,b)[t]$, we obtain \eqref{eq:taft relations generalized}.

The proposition is proved.
\end{proof}

By Proposition \ref{pr:presentation taft}, for  $a$ being a primitive $n$-th root of unity in $R$,  ${\bf H}_n(a,0)$ is the  Tuft algebra with the presentation: $\underline s^n=1,~\partial^n=0,\underline s \partial=a \partial \underline s$.
  
It turns out that ${\bf H}_n(a,b)$ is always a module algebra over the Taft algebra, and the multiplication in the former can be expressed in terms of the action.

\begin{corollary} In the notation of Proposition \ref{pr:presentation taft}, suppose that $a$ is a primitive $n$-th root of unity in $R$. Then

(a) ${\bf H}_n(a,b)$ is an ${\bf H}_n(a^{-1},0)$-module algebra via:
$\underline s\act s=a s,~\partial\act s=0,~\underline s\act D=b+aD$.

(b) $\displaystyle{s^\ell p(D)s^{-\ell}=\sum_{k=0}^\ell (-1)^k{k+\ell \brack \ell-1}_a b^k ((\underline s^\ell \partial^k)\act p(D))\cdot s^k}$
for any polynomial $p\in R[x]$ and $\ell\ge 0$. 
\end{corollary}

\section{Proofs of main results}

\subsection{Almost free Hopf algebras and proof of Theorems \ref{th:hopf hat W}  and \ref{th:hopf hat W gen}} 
\label{subsec:proof of  Theorem hopf hat}
 Given a group $W$,  a conjugation-invariant subset ${\mathcal S}\subset W\setminus\{1\}$, 
and any maps $\chi,\sigma:W\times {\mathcal S}\to R$,  let  $\hat {\bf H}'_{\chi,\sigma}(W)$ 
be an $R$-algebra generated by $W$ and $D_s$, $s\in {\mathcal S}$ subject to relations \eqref{eq:relations hat H}
for all $s\in {\mathcal S}$, $w\in W$.

\begin{proposition} 
\label{pr:hopf hat W gen pf} For any maps $\chi,\sigma:W\times {\mathcal S}\to R$ one has 

(a) $\hat {\bf H}'_{\chi,\sigma}(W)$  is a Hopf algebra with the coproduct $\Delta$, counit $\varepsilon$, and the antipode $S$ given by \eqref{eq:hopf H'}.

(b) $\hat {\bf H}'_{\chi,\sigma}(W)$ factors as $\hat {\bf H}'_{\chi,\sigma}(W)=T(V)\cdot RW$ over $R$, where  $V=\oplus_{s\in {\mathcal S}} R\cdot D_s$,
iff $\chi$ and $\sigma$ satisfy 
\eqref{eq:2 cocycle}.

\end{proposition}

\begin{proof} Prove (a). Clearly, $\hat {\bf H}'_{\chi,\sigma}(W)={\bf H}_\gamma$ in the notation of Proposition \ref{pr:deformed nichols}, where:

$\bullet$ $H=RW$, $V=\oplus_{s\in {\mathcal S}} R\cdot D_s$ is a $T(H)$-module via 
\begin{equation}
\label{eq:quasi-action of W on V}
w(D_s)=\chi_{w,s} D_s
\end{equation} for $w\in W$, $s\in {\mathcal S}$ and an $H$-comodule via $\delta(D_s)=s\otimes D_s$.

$\bullet$ $\gamma:RW\times V\to R$ is given by 
\begin{equation}
\label{eq:gamma special W}
\gamma(w,D_s)=\sigma_{w,s}(1-wsw^{-1})
\end{equation}
 for $w\in W$, $s\in {\mathcal S}$.

Then, clearly the Yetter-Drinfeld condition \eqref{eq:YD compatibility} holds because
$$\delta(w(D_s))=\chi_{w,s}\delta(D_{wsw^{-1}})=wsw^{-1}\otimes \chi_{w,s}D_{wsw^{-1}}=wsw^{-1}\otimes w(D_s)$$
for all $w\in W$, $s\in {\mathcal S}$.

The second condition of Proposition \ref{pr:deformed nichols} also holds automatically because
$$\Delta(\gamma(w,D_s))=\sigma_{w,s}(1\otimes 1-wsw^{-1}\otimes wsw^{-1})=
\gamma(w,D_s)\otimes 1+wsw^{-1}\otimes \gamma(w,D_s)$$
and $\varepsilon(\gamma(w,D_s))=0$ for for all $w\in W$, $s\in {\mathcal S}$.

Thus, $\hat {\bf H}'_{\chi,\sigma}(W)={\bf H}_\gamma$ is a Hopf algebra by Proposition \ref{pr:deformed nichols}.
This proves (a).

Prove (b). It suffices to translate the conditions of Proposition \ref{pr:mu nu gamma H}. Indeed, taking into account that the first condition of \eqref{eq:2 cocycle} implies $\chi_{1,s}=1$ for all $s\in {\mathcal S}$, we see that the first condition of \eqref{eq:2 cocycle} is equivalent to that \eqref{eq:quasi-action of W on V} is a $RW$-action on $V=\oplus_{s\in {\mathcal S}} R\cdot D_s$.

Finally, the condition \eqref{eq:gamma property H} reads for this action and $\gamma$ given by \eqref{eq:gamma special W}:
$$\sigma_{w_1w_2,s}(1-w_1w_2sw_2^{-1}w_1^{-1})=\gamma(w_1w_2,D_s)=\gamma(w_1,w_2(D_s))+w_1\gamma(w_2,D_s)w_1^{-1}$$
$$=\chi_{w_2,s}\sigma_{w_1,w_2sw_2^{-1}}(1-w_1w_2sw_2^{-1}w_1^{-1})+\sigma_{w_2,s}(1-w_1w_2sw_2^{-1}w_1^{-1})$$
in $R W$ for all $w_1,w_2\in W$, $s\in {\mathcal S}$, which is, clearly, equivalent to the second condition of \eqref{eq:2 cocycle}.
This proves (b).

Proposition \ref{pr:hopf hat W gen pf} is proved.
\end{proof}

Furthermore,  we say that a family  ${\bf f}=(f_s)\in (R[x]\setminus \{0\})^{\mathcal S}$ of polynomials $f_s\in R[x]\setminus \{0\}$ is {\it adapted} to ${\mathcal S}$ if $\deg f_s=|s|$ for all $s\in {\mathcal S}$ (with the convention $|s|=0$ if $s$ is of infinite order, hence, $f_s$ is a nonzero constant in that case).

For any maps $\chi,\sigma:W\times {\mathcal S}\to R$  and any family  ${\bf f}=(f_s)\in (R[x]\setminus\{0\})^{\mathcal S}$ adapted to  ${\mathcal S}$ let $\hat {\bf H}_{\chi,\sigma,{\bf f}}(W)$ be an $R$-algebra generated by $W$ and $D_s$, $s\in {\mathcal S}$ subject to relations \eqref{eq:relations hat H}
for all $s\in {\mathcal S}$, $w\in W$ and the functional relations
\begin{equation}
\label{eq:fstDs}
f_s(ts+D_s)=f_s(t)
\end{equation} (if $s$ is of infinite order, i.e., $|s|=0$, then the condition \eqref{eq:fstDs} is vacuous). 

By definition, one has a surjective homomorphism of $R$-algebras 
$\pi_{\bf f}:\hat {\bf H}'_{\chi,\sigma}(W)\twoheadrightarrow \hat {\bf H}_{\chi,\sigma,{\bf f}}(W)$. 
%

\begin{proposition} For any family ${\bf f}$ adapted to ${\mathcal S}$,  
$\hat {\bf H}_{\chi,\sigma,{\bf f}}(W)$ is naturally a Hopf algebra (i.e., $\pi_{\bf f}$ is a homomorphism of Hopf algebras).

\end{proposition}
 
\begin{proof} Using notation from Section \ref{sec:Taft algebras} 
and copying \eqref{eq:yk of I_f}, define for each finite order element $s\in {\mathcal S}$  the elements $y_0^s,\ldots,y_{|s|-1}^s\in \hat {\bf H}'_{\chi,\sigma}(W)$ by  
$y_k^s=\sum\limits_{i=k}^{|s|} a_i^s\{s,D_s\}_{i,k-i}$, 
where $\sum\limits_{k=0}^{|s|} a_{k_s}t^s=f(t)$.


Denote by ${\bf K}_{f_s}$ the $R$-submodule of ${\bf H}'_{\chi,\sigma}(W)$ generated by 1 and $y_k^s$, $k=0,\ldots,|s|-1$.
(with the convention that ${\bf K}_{f_s}=R$ if $|s|=0$).

\begin{lemma} 
\label{le:Kfs}
${\bf K}_{f_s}$ is a right coideal in $\hat {\bf H}'_{\chi,\sigma}(W)$ for each $s\in {\mathcal S}$. 

\end{lemma}

\begin{proof} The proof is identical to that of Lemma \ref{le:right coideal Hf}. 
\end{proof}

Therefore, ${\bf K}_{\bf f}:=\sum\limits_{s\in {\mathcal S}} {\bf K}_{f_s}$ is a right coideal in $\hat {\bf H}'_{\chi,\sigma}(W)$ by Proposition \ref{pr:sum intersection of coideals} (for right coideals) and  
${\bf K}_{\bf f}^+:={\bf K}_{\bf f}\cap Ker~\varepsilon$ is the $R$-submodule of ${\bf H}'_{\chi,\sigma}(W)$ generated by $y_k^s-a_k^s$, $k=0,\ldots,|s|$, $s\in {\mathcal S}$. 

By definition,  the kernel of $\pi_{\bf f}$ is the ideal of ${\bf H}'_{\chi,\sigma}(W)$ generated by ${\bf K}_{\bf f}^+$.
In view of Proposition \ref{pr:from coideal to hopf ideal}, this guarantees that the kernel of $\pi_{\bf f}$ is a Hopf ideal in 
${\bf H}'_{\chi,\sigma}(W)$. 
Therefore, ${\bf H}_{\chi,\sigma,{\bf f}}(W)=\hat {\bf H}'_{\chi,\sigma}(W)/(Ker~\pi_{\bf f})$ is a Hopf algebra and $\pi_{\bf f}$ is a homomorphism of Hopf algebras.

The proposition is proved.

\end{proof}

\noindent {\bf Proof of Theorem \ref{th:hopf hat W gen}}. 
Let us show that $\hat {\bf H}_{\sigma,\chi}(W)={\bf H}_{\chi,\sigma,{\bf f}}(W)$, where 
\begin{equation}
\label{eq:actual fs}
f_s=f_{|s|}^{a_s,b_s}=x(x-b_s)(x-b_s(1+a_s))\cdots (x-b_s(1+a_s+\cdots+a_s^{|s|-2})) 
\end{equation}
in the notation \eqref{eq:fnab}, where we abbreviated $a_s:=\chi_{s,s}$ and $b_s:=\sigma_{s,s}$ (with the convention $f_s=1$ if $|s|=0$). 
Indeed, in view of Proposition \ref{pr:presentation taft}, since each relevant $\chi_{s,s}$ is the primitive $|s|$-th root of unity,  the defining functional relation \eqref{eq:fstDs} for $\hat {\bf H}_{\chi,\sigma,{\bf f}}(W)$ coincides with the defining relation \eqref{eq:relations hat H taft} for $\hat {\bf H}_{\sigma,\chi}(W)$.
Thus, $\hat {\bf H}_{\sigma,\chi}(W)={\bf H}_{\chi,\sigma,{\bf f}}(W)$ is a Hopf algebra. 

Theorem \ref{th:hopf hat W gen} is proved.
\endproof

\noindent {\bf Proof of Theorem \ref{th:hopf hat W}}. 
Similarly to Definition \ref{def:Hecke-Hopf algebra W}, for any Coxeter group $W=\langle s_i|i\in I\rangle$  let  $\hat {\bf H}'(W)$ the $\ZZ$-algebra generated by 
$s_i,D_i$, $i\in I$ subject to relations:

(i) Rank $1$ relations: $s_i^2=1$, $s_iD_i+D_is_i=s_i-1$ for $i\in I$.


(ii) Coxeter relations: $(s_is_j)^{m_{ij}}=1$ and linear braid relations: $\underbrace{D_is_js_i\cdots s_{j'}}_{m_{ij}} =\underbrace{s_j\cdots s_{i'}s_{j'}D_{i'}}_{m_{ij}}$
for all distinct $i,j\in I$, where $i'=\begin{cases} 
i & \text{if $m_{ij}$ is even}\\
j & \text{if $m_{ij}$ is odd}\\
\end{cases}$ and  $\{i',j'\}=\{i,j\}$.

\begin{proposition} 
\label{pr:relations hat H coxeter}
$\hat {\bf H}'(W)=\hat {\bf H}'_{\chi,\sigma}(W)$ for any Coxeter group $W$  the notation of Proposition \ref{pr:hopf hat W gen pf} with $R=\ZZ$, where $\chi:W\times {\mathcal S}\to \{-1,1\}\subset \ZZ$ and $\sigma:W\times {\mathcal S}\to \{0,1\}\subset \ZZ$ are given  by: 
\begin{equation}
\label{eq:chi sigma defined}
\chi_{w,s}=(-1)^{\ell(w)+\frac{1}{2}(\ell(wsw^{-1})-\ell(s))},~\sigma_{w,s}=\frac{1-\chi_{w,s}}{2}
\end{equation}
for all $w\in W$, $s\in {\mathcal S}$.
In particular, $\hat {\bf H}'(W)$ is a Hopf algebra  with the coproduct $\Delta$, counit $\varepsilon$, and the antipode $S$ given by \eqref{eq:hopf H'}.

\end{proposition}

\begin{proof} Clearly, $\hat {\bf H}'(W)$ is generated over $\ZZ$ by $V=\oplus_{s\in {\mathcal S}} \ZZ D_s$ and the group $W$.

We need the following  result.
\begin{lemma} 
\label{le:cocycle coxeter}
For any Coxeter group $W$ the map  \eqref{eq:chi sigma defined} satisfies 
$\chi_{s_i,s}=\begin{cases}
-1 & \text{if $s= s_i$}\\
1 & \text{if $s\ne s_i$}\\
\end{cases}$ for $s\in {\mathcal S}$, $i\in I$. In particular, 
\begin{equation}
\label{eq:si conj Ds}
s_iD_ss_i=
\begin{cases}
D_{s_i s s_i} & \text{if $s\ne s_i$}\\
-D_{s_i}+1-s_i & \text{if $s= s_i$}\\
\end{cases}=\chi_{s_i,s}D_{s_i s s_i}+\frac{1-\chi_{s_i,s}}{2}(1-s_iss_i)
\end{equation}
 in $\hat {\bf H}'(W)$
for all $s\in {\mathcal S}$, $i\in I$.

\end{lemma}

\begin{proof} 
Clearly, $\chi$ defined by \eqref{eq:chi sigma defined} satisfies the first assertion of the lemma because $\ell(s_i)=1$ and
$\ell(s_iss_i)-\ell(s)\in \{-2,2\}$ for all \ $i\in I$, 
$s\in {\mathcal S}\setminus \{s_i\}$, i.e., $\ell(s_i)+\frac{1}{2}(\ell(s_iss_i)-\ell(s))\in \{0,2\}$ (of course, $\chi_{s_i,s_i}=-1$).
Then \eqref{eq:si conj Ds} follows.

The lemma is proved.
\end{proof}

Prove that \eqref{eq:relations hat H} hold in $\hat {\bf H}'(W)$ by induction on $\ell(w)$. If $w=1$, we have nothing to prove. If $\ell(w)=1$, i.e., $w=s_i$ for some $i\in I$, then the \eqref{eq:si conj Ds} which verifies \eqref{eq:relations hat H}. 

Suppose that $\ell(w)\ge 2$, i.e., $w=w_1w_2$ for some $w_1,w_2\in W\setminus \{1\}$ with $\ell(w_1)+\ell(w_2)=\ell(w)$. Then using the inductive hypothesis in the form:
$w_2D_sw_2^{-1}=\chi_{w_2,s}D_{s'}+\frac{1-\chi_{w_2,s}}{2}(1-w_2sw_2^{-1})$, 
where we abbreviated $s'=w_2sw_2^{-1}$, we obtain, by conjugating both sides with $w_1$:
$$wD_sw^{-1}=w_1(w_2D_sw_2^{-1})w_1^{-1}=\chi_{w_2,s}w_1D_{s'}w_1^{-1}+\frac{1-\chi_{w_2,s}}{2}(1-wsw^{-1})$$
$$=\chi_{w_2,s}(\chi_{w_1,s'}D_{w_1s'w_1^{-1}}+\frac{1-\chi_{w_1,s'}}{2}(1-w_1s'w_1^{-1}))+\frac{1-\chi_{w_2,s}}{2}(1-wsw^{-1})$$
$$=\chi_{w,s}D_{wsw^{-1}}\frac{1-\chi_{w,s}}{2}(1-wsw^{-1})$$
by the inductive hypothesis with $w'$ and by the first condition of \eqref{eq:2 cocycle}.
%

This finishes the inductive proof of \eqref{eq:relations hat H}. Thus, $\hat {\bf H}'(W)=\hat {\bf H}'_{\chi,\sigma}(W)$ in the notation \eqref{eq:relations hat H} with $\chi$ and $\sigma$ are as in \eqref{eq:chi sigma defined}.
Therefore, $\hat {\bf H}'(W)$ is a Hopf algebra by Proposition \ref{pr:hopf hat W gen pf}(a).

The proposition is proved.
\end{proof}

Finally, note that for $\chi$ and $\sigma$ given by \eqref{eq:chi sigma defined} the relations \eqref{eq:relations hat H taft} become $D_s^2=D_s$, $s\in {\mathcal S}$. 
Moreover, it follows from \eqref{eq:si conj Ds} that these relations considered in $\hat {\bf H}(W)$ follow from the relations $D_i^2=D_i$, $i\in I$. This proves the following 
 %
%
\begin{lemma} 
\label{le:hats are equal}
$\hat {\bf H}(W)=\hat {\bf H}_{\chi,\sigma}(W)$ for any Coxeter group $W$ and $\chi,\sigma$ given by \eqref{eq:chi sigma defined}.
\end{lemma}
Thus, $\hat {\bf H}(W)$ is a Hopf algebra by Theorem \ref{th:hopf hat W gen}.

Theorem \ref{th:hopf hat W} is proved. \endproof

\subsection{Factorization of Hecke-Hopf algebras and proof of Theorems  \ref{th:ZWcross and free D}, \ref{th:ZWcross and free D gen}}
\label{subsec:proof of Theorem ZWcross and free D}

Prove  Theorem \ref{th:ZWcross and free D gen} first. Proposition \ref{pr:hopf hat W gen pf}(b) together with \eqref{eq:2 cocycle} guarantee that $\hat {\bf H}'_{\chi,\sigma}(W)$ factors as $\hat {\bf H}'_{\chi,\sigma}(W)=T(V)\cdot RW$ over $R$, where  $V=\oplus_{s\in {\mathcal S}} R\cdot D_s$. To establish the factorization of $\hat {\bf H}_{\chi,\sigma}(W)$ we need the following result (which is a pre-condition in Lemma \ref{le:general factored quotient H}).

\begin{proposition} 
\label{pr:wKsw-1}
In the notation of Lemma \ref{le:Kfs}, $w\cdot {\bf K}_s\cdot w^{-1}={\bf K}_{wsw^{-1}}$ for all $w\in W$, $s\in {\mathcal S}$.

\end{proposition}

\begin{proof}
The following is an immediate consequence of \eqref{eq:2 cocycle}. 
\begin{lemma} 
\label{le:basic properties of chi sigma}
For any $\sigma,\chi:W\times {\mathcal S}\to R$ satisfying \eqref{eq:2 cocycle} one has for all $w\in W$, $s\in {\mathcal S}$:

(a) $\chi_{1,s}=1$, $\sigma_{1,s}=0$, $\chi_{w^{-1},wsw^{-1}}= \frac{1}{\chi_{w,s}}$, $\sigma_{w^{-1},wsw^{-1}}=-\frac{\sigma_{w,s}}{\chi_{w,s}}$.

(b) $\chi_{w sw^{-1},wsw^{-1}}= \chi_{s,s}$, $(\chi_{s,s}-1)\sigma_{w,s}=\chi_{w,s}\sigma_{w sw^{-1},wsw^{-1}}-\sigma_{s,s}$.
\end{lemma}

Furthermore, in the notation of Section \ref{subsec:proof of  Theorem hopf hat}, for $s\in {\mathcal S}$ of finite order $|s|$ we abbreviate: $a_s=\chi_{s,s}$, $b_s=\sigma_{s,s}$, $f_s:=f_{|s|}^{a_s,b_s}\in R[x]$ and denote 
$\delta_s(t):=f_s(ts+D_s)\in \hat {\bf H}'_{\chi,\sigma}(W)[t]$.

 We need the following result.

\begin{lemma} 
\label{le:fwsw-1}
In the assumptions of Theorem \ref{th:ZWcross and free D gen} one has (in the notation \eqref{eq:actual fs}):
\begin{equation}
\label{eq:fwsw-1}
f_{wsw^{-1}}(x)=f_s(\chi_{w,s}\cdot x+\sigma_{w,s})
\end{equation}
for all $w\in W$, $s\in {\mathcal S}$ of finite order. In particular, $w\cdot \delta_s(t)\cdot w^{-1}=\delta_{wsw^{-1}}\left(\frac{t-\sigma_{w,s}}{\chi_{w_s}}\right)$.

\end{lemma}

\begin{proof} 
For a given $w\in W$ we abbreviate  $s':=wsw^{-1}$ and $n:=|s|=|s'|$. Then $a_{s'}^n=1$ and \eqref{eq:sigmaws etc} reads:
$\sigma_{w,s}=b_s\frac{1-a_s^k}{1-a_s}$, 
where $k=\kappa_{w,s}$. Also $a_{s'}=a_s$ and $b_{s'}=\frac{b_s}{\chi_{w,s}}a_s^k$ 
by Lemma \ref{le:basic properties of chi sigma}. Combining, we obtain
$\sigma_{w,s}=b_{s'}{\chi_{w,s}}\frac{a_{s'}^{-k}-1}{1-a_{s'}}$. Then:
$$f_s(\chi_{w,s}\cdot x+\sigma_{w,s})=\prod_{i=1}^n \left(\chi_{w,s}\cdot x+\sigma_{w,s}-b_s\frac{1-a_s^i}{1-a_s}\right)=
\prod_{i=1}^n \left(x+\frac{\sigma_{w,s}}{\chi_{w,s}}-\frac{b_s}{\chi_{w,s}}\cdot \frac{1-a_{s'}^i}{1-a_{s'}}\right)$$
$$=\prod_{i=1}^n \left(x+b_{s'} \frac{a_{s'}^{-k}-1}{1-a_{s'}}-b_{s'}\frac{a_{s'}^{-k}-a_{s'}^{i-k}}{1-a_{s'}}\right)
=\prod_{i=1}^n \left(x-b_{s'}\frac{1-a_{s'}^{i-k}}{1-a_{s'}}\right)=\prod_{i=1}^n (x-b_{s'}\frac{1-a_{s'}^i}{1-a_{s'}})=f_{s'}(x) \ .$$
This proves the first assertion of the lemma. Prove the second assertion now. Indeed, using \eqref{eq:fwsw-1} the form $f_s(t)=f_{s'}(p)$, where $p=\frac{t-\sigma_{w,s}}{\chi_{w_s}}$, we obtain:
$$w\cdot \delta_s(t)\cdot w^{-1}=f_s(w\cdot(ts+D_s)\cdot w^{-1})=f_s(ts'+\chi_{w,s}D_{s'}+\sigma_{w,s}(1-s'))$$
$$=f_s((t-\sigma_{w,s})s'+\chi_{w,s}D_s+\sigma_{w,s})=f_s(\chi_{w,s}(ps'+D_s)+\sigma_{w,s})=f_{s'}(ps'+D_{s'})+\sigma_{w,s})
=\delta_{wsw^{-1}}(p)\ .$$

The lemma is proved.
\end{proof}

Since ${\bf K}_s$ is generated by the coefficients of $\delta_s(t)$, the second assertion of Lemma \ref{le:fwsw-1}, finishes the proof of Proposition \ref{pr:wKsw-1}. \end{proof}

In particular, ${\bf K}=\sum\limits_{s\in {\mathcal S}} {\bf K}_s$ satisfies $w\cdot {\bf K}={\bf K}\cdot w$ for all $w\in W$.
This and the factorization of ${\bf H}=\hat {\bf H}'_{\chi,\sigma}(W)=T(V)\cdot RW$ guarantee that Lemma \ref{le:general factored quotient H} is applicable here, therefore, $\underline {\bf H}=\hat {\bf H}_{\chi,\sigma}(W)$ factors as $\underline {\bf H}=\underline {\bf D}\cdot RW$ over $R$, where $\underline {\bf D}=\hat {\bf D}_{\chi,\sigma}(W)=T(V)/\langle {\bf K}\rangle$.

Theorem \ref{th:ZWcross and free D gen} is proved.
\endproof 

\noindent {\bf Proof of Theorem \ref{th:ZWcross and free D}}. 
We need the following result.

\begin{proposition} 
\label{pr:chi sigma satisfies cocycle condition}
For any Coxeter group $W$, one has 

(a) the maps $\chi$ and $\sigma$ defined by \eqref{eq:chi sigma defined}
satisfy \eqref{eq:2 cocycle}.

(b) 
the map  $\chi$ given by  \eqref{eq:chi sigma defined} satisfies 
$\chi_{w,s}=\begin{cases}
1 & \text{if $\ell(ws)>\ell(w)$}\\
-1 & \text{if $\ell(ws)<\ell(w)$}\\
\end{cases}$
for all $s\in {\mathcal S}$,  $w\in W$.

\end{proposition}
 
\begin{proof} Prove (a). The following immediate fact  gives a ``default" $\chi$ satisfying  \eqref{eq:2 cocycle}.

\begin{lemma} 
\label{le:default chi} 
Let $W$ be a group and ${\mathcal S}$ be a conjugation-invariant subset of $W$, then for any ring $R$, a group homomorphism $\rho:W\to R^\times$, and map $\varphi:{\mathcal S}\to R^\times$, the map $\chi=\chi^{\varphi,\rho}:W\times {\mathcal S}\to R^\times$ given by 
$\chi_{w,s}=\rho(w)\cdot\frac{\varphi(wsw^{-1})}{\varphi(s)}$
for $w\in W$, $s\in {\mathcal S}$ satisfies the first condition of \eqref{eq:2 cocycle}.

\end{lemma}

We use Lemma \ref{le:default chi} with $R=\ZZ$, a homomorphism 
$\rho:W\to \{-1,1\}=\ZZ^\times$, and a map $\varphi:{\mathcal S}\to \{-1,1\}$ given respectively by:
$\rho(w)= (-1)^\ell(w),~\varphi(s)=(-1)^{\frac{1}{2}(\ell(s)-1)}$
for $w\in W$, $s\in {\mathcal S}$. Then, clearly, $\chi$ defined by \eqref{eq:chi sigma defined} equals $\chi^{\varphi,\rho}$ in the notation of Lemma \ref{le:default chi} and, thus, satisfies the first condition \eqref{eq:2 cocycle}.

Finally,
$2\sigma_{w_2,s}+2\chi_{w_2,s}\sigma_{w_1,w_2sw_2^{-1}} =1-\chi_{w_2,s}+\chi_{w_2,s}\cdot (1-\chi_{w_1,w_2sw_2^{-1}})
=1-\chi_{w_1w_2,s}=2\sigma_{w_1w_2,s}$
by the first condition \eqref{eq:2 cocycle}. This proves (a).

Prove (b) now. Note that $\chi_{s_i,s'}=1$ iff $s'\ne s_i$ by Lemma \ref{le:cocycle coxeter}. This, Proposition \ref{pr:chi sigma satisfies cocycle condition}(a) and the first equation \eqref{eq:2 cocycle} taken with $w_1=s_i$, $w_2=s_iw$ imply that 
\begin{equation}
\label{eq:chi recursion}
\chi_{w,s}=(-1)^{\delta_{s_i,wsw^{-1}}}\chi_{s_iw,s}
\end{equation}
for all $w\in W$, $s\in {\mathcal S}$. 
Furthermore, we  proceed by induction in $\ell(w)$ in the form
\begin{equation}
\label{eq:chi=-1 inductive hypothesis} 
\text{If $\chi_{w,s}=-1$, then $\ell(w)>\ell(ws)$}
\end{equation}
Indeed, if $w=s_i$ for some $i\in I$, then $\chi_{s_i,s}=-1$ iff $s=s_i$ and we have nothing to prove. 

Suppose that $\chi_{w,s}=-1$ for some $w\in W$ with $\ell(w)\ge 2$ and some $s\in {\mathcal S}$.  Now choose $i\in I$ such that  $\ell(s_iw)=\ell(w)-1$. If $s_i=wsw^{-1}$, i.e., $s_iw=ws$, then clearly, $\ell(w)>\ell(ws)$ and we have nothing to prove. If $s_i\ne wsw^{-1}$, then 
\eqref{eq:chi recursion} guarantees that 
$\chi_{s_iw,s}=-1$  and the inductive hypothesis for $(s_iw,s)$ asserts that 
$\ell(s_iw)>\ell(s_iws)$. Taking into account that $\ell(s_iws)\ge \ell(ws)-1$, we obtain $\ell(w)>\ell(ws)$, which finishes the proof of \eqref{eq:chi=-1 inductive hypothesis}. Finally, using \eqref{eq:chi=-1 inductive hypothesis}
let us prove 
\begin{equation}
\label{eq:chi=1 hypothesis} 
\text{If $\chi_{w,s}=1$, then $\ell(ws)>\ell(w)$}
\end{equation}
Indeed, taking into account that $\chi_{s,s}=-1$ for all $s\in {\mathcal S}$ by \eqref{eq:chi sigma defined} and using Proposition \ref{pr:chi sigma satisfies cocycle condition}(a) again,  the first equation \eqref{eq:2 cocycle} taken with $w_1=ws$, $w_2=s$ implies
$\chi_{w,s}=-\chi_{ws,s}$
for all $w\in W$, $s\in {\mathcal S}$. Therefore,  $\chi_{w,s}=1$ implies that $\chi_{ws,s}=-1$ hence $\ell(ws)>\ell(wss)=\ell(w)$ by 
\eqref{eq:chi=-1 inductive hypothesis}. This proves \eqref{eq:chi=1 hypothesis}. Part (b) is proved.

The proposition is proved.
\end{proof} 

Let us show that these  $\chi$ and $\sigma$ also satisfy \eqref{eq:sigmaws etc}. Indeed, $|s|=2$  $\chi_{s,s}=-1$, $\sigma_{s,s}=1$, for all $s\in {\mathcal S}$ and $\chi_{w,s}^2=1$ for all $w\in W$, therefore \eqref{eq:sigmaws etc} holds automatically with $\kappa_{w,s}$ taken to be the exponent in \eqref{eq:chi sigma defined}, so that ${\chi_{w,s}}=(-1)^{\kappa_{w,s}}$.
Therefore, $\hat {\bf H}(W)=\hat {\bf H}_{\chi,\sigma}(W)$ (by Lemma \ref{le:hats are equal}) and it
factors over $\ZZ$ as $\hat {\bf H}(W)=\hat {\bf D}(W)\cdot \ZZ W$ by Theorem \ref{th:ZWcross and free D gen}.

Theorem \ref{th:ZWcross and free D} is proved. \endproof

\subsection{Left coideals and proof of Theorems \ref{th:Coideal Hopf}, \ref{th:Coideal Hopf ij}, \ref{th:HW factorization}, \ref{th:Coideal Hopf gen}, and \ref{th:ZWcross and free D gen H}}
\label{subsec:proof of Theorem Coideal Hopf}
Prove Theorem \ref{th:Coideal Hopf gen} first. 
%
%
%
Indeed, by definition \eqref{eq:KchisigmaW}, 
\begin{equation}
\label{eq:Kchisigma}
{\bf K}_{\chi,\sigma}(W)=\{x\in \tilde {\bf D}_{\chi,\sigma}(w)\,|\,wxw^{-1}\in \tilde {\bf D}_{\chi,\sigma}(W)~\forall~w\in W \}={\bf K}(R W,\tilde {\bf D}_{\chi,\sigma}(W))
\end{equation}
in the notation 
\eqref{eq:general K defined}.
Also, by definition, $\tilde {\bf D}_{\chi,\sigma}(W)$ is a left coideal subalgebra in  $\hat {\bf H}(W)$, i.e., $\Delta(\tilde {\bf D}_{\chi,\sigma}(W))\subset \hat {\bf H}(W) \otimes \tilde {\bf D}_{\chi,\sigma}(W)$. 
These and $R$-freeness of $\hat {\bf H}_{\chi,\sigma}(W)$ guarantee that all conditions of Theorem \ref{th:general hopf ideal} are satisfied, therefore $\underline {\bf H}={\bf H}_{\chi,\sigma}(W)$ is naturally a Hopf algebra.

Theorem \ref{th:Coideal Hopf gen} is proved.
\endproof

\noindent {\bf Proof of Theorem \ref{th:ZWcross and free D gen H}}. 
We need the following result. 

\begin{lemma} \label{le:free product}
In the notation of Theorem \ref{th:hopf hat W gen}, one has:

(a) the algebra $\hat {\bf D}_{\chi,\sigma}(W)$ is the free product (over $R$) of algebras ${\bf D}_s$, $s\in {\mathcal S}$, where ${\bf D}_s$ is the $R$-algebra generated by $D_s$ subject to relations \eqref{eq:relations hat H taft} (e.g.,  ${\bf D}_s=R[D_s]$ if $s$ is of infinite order).

(b) If the condition \eqref{eq:s conj} holds for all $s\in {\mathcal S}$, then $\hat {\bf D}_{\chi,\sigma}(W)$ is a free $R$-module.
\end{lemma}

\begin{proof} Part (a) is immediate from the presentation \eqref{eq:relations hat H taft} of $\hat {\bf D}_{\chi,\sigma}(W)$.

Prove (b). If $s$ is of finite order $|s|$ and $\chi_{s,s}$ is a primitive $|s|$-th root of unity in $R^\times$, then, according to Remark \ref{rem:single monic relation D_s}, ${\bf D}_s$ is generated by $D_s$ subject to the only (monic) polynomial relation, therefore, is a free $R$-module. Since free product of free $R$-modules is also a free $R$-module,  this finishes the proof of (b).

The lemma is proved.
\end{proof}

\begin{lemma} 
\label{le:free hat H}
In the assumptions of Theorem \ref{th:ZWcross and free D gen H}, one has:

(a)  
 $\hat {\bf H}_{\chi,\sigma}(W)$ is a free $R$-module.

(b) In the notation \eqref{eq:Kchisigma}, one has
\begin{equation}
\label{eq:Kchisigma hat}
{\bf K}_{\chi,\sigma}(W)=\{x\in \hat {\bf D}_{\chi,\sigma}(w)\,|\,wxw^{-1}\in \hat {\bf D}_{\chi,\sigma}(W)~\forall~w\in W \}={\bf K}(R W,\hat {\bf D}_{\chi,\sigma}(W))
\end{equation}
and ${\bf K}_{\chi,\sigma}(W)$ is a left coideal in $\hat {\bf H}_{\chi,\sigma}(W)$.
\end{lemma}

\begin{proof} In the assumptions of Theorem \ref{th:ZWcross and free D gen H}, 
one has a factorization $\hat {\bf H}_{\chi,\sigma}(W)=\hat {\bf D}_{\chi,\sigma}(W)\cdot RW$ by Theorem \ref{th:ZWcross and free D gen}, in particular, $\tilde {\bf D}_{\chi,\sigma}(W)=\hat {\bf D}_{\chi,\sigma}(W)$. Taking into account that $RW$ is also a free $R$-module and tensor product of free modules is free, we finish the proof of part (a).

Prove (b). \eqref{eq:Kchisigma hat} is immediate. The second assertion follows from (a) and Proposition \ref{pr:general K}. 

The lemma is proved. \end{proof}

Thus, all conditions of Theorem \ref{th:general hopf ideal} are satisfied for the Hopf algebra $\hat {\bf H}_{\chi,\sigma}(W)$, therefore $\underline {\bf H}={\bf H}_{\chi,\sigma}(W)$ is a Hopf algebra by Theorem \ref{th:general hopf ideal}. Finally, the factorization ${\bf H}_{\chi,\sigma}(W)={\bf D}_{\chi,\sigma}(W)\cdot RW$ follows from Proposition \ref{pr:from coideal to hopf ideal} and Theorem \ref{th:ZWcross and free D gen}.

Theorem \ref{th:ZWcross and free D gen H} is proved. \endproof

\noindent {\bf Proof of Theorem \ref{th:Coideal Hopf}}. Taking into account that $\hat {\bf H}(W)=\hat {\bf H}_{\chi,\sigma}(W)$ by Lemma \ref{le:hats are equal}   for $\chi$ and $\sigma$ given by 
\eqref{eq:chi sigma defined}, we see that ${\bf K}(W)={\bf K}_{\chi,\sigma}(W)$, therefore, $\underline {\bf H}(W)={\bf H}_{\chi,\sigma}(W)={\bf H}(W)$, which is a Hopf algebra by Theorem \ref{th:ZWcross and free D gen H}. 
%
%

Theorem \ref{th:Coideal Hopf} is proved. \endproof

\noindent {\bf Proof of Theorem \ref{th:Coideal Hopf ij}}. Clearly, the algebra $\hat {\bf H}(W)$ is filtered by $\ZZ_{\ge 0}$ via 
$\deg D_s=1$, $s\in {\mathcal S}$, $\deg w=0$. For each $r\in \ZZ_{\ge 0}$ denote by $\hat {\bf H}(W)_{\le r}$ the filtered component of degree $r$. In particular, $\hat {\bf H}(W)_{\le 0}=R W$.  For each subset $X\subset \hat {\bf H}(W)$ we abbreviate 
$X_{\le r}:=X\cap \hat {\bf H}(W)_{\le r}$.

\begin{proposition} 
\label{pr:graded coideal J} 
For each $r\ge 0$ and $J\subset I$ the $\ZZ$-module ${\bf K}(W_J)_{\le r}$ is a left coideal in $\hat {\bf H}(W)$.

\end{proposition}

\begin{proof}We need the following result.

\begin{lemma} 
\label{le:graded coideal}
For each $r\ge 0$ the $\ZZ$-module $\hat {\bf D}(W)_{\le r}$ is a left coideal in $\hat {\bf H}(W)$.

\end{lemma}

\begin{proof} We proceed by induction in $r$. Indeed, since $\hat {\bf D}(W)_{\le 0}=\ZZ$, 
$\hat {\bf D}(W)_{\le 1}=\ZZ+\sum\limits_{s\in {\mathcal S}} \ZZ\cdot D_s$, the assertion is immediate for $r=0,1$. Suppose that $r>1$. Clearly, ${\bf D}(W)_{\le r}={\bf D}(W)_{\le r-1}\cdot \hat {\bf D}(W)_{\le 1}$. Therefore, 
$\Delta({\bf D}(W)_{\le r})
\subset ({\bf H}(W)\otimes {\bf D}(W)_{\le r-1})\cdot({\bf H}(W)\otimes {\bf D}(W)_{\le 1})={\bf H}(W)\otimes {\bf D}(W)_{\le r}$
by the inductive hypothesis.

The lemma is proved.
\end{proof}

We need the following result.

\begin{lemma}
\label{le:levi}
For any Coxeter group $W$ and  any subset $J\subset I$ one has:

(a)  the subalgebra of $\hat {\bf H}(W)$ generated by $s_j,D_j$, $j\in J$ is naturally 
isomorphic to $\hat {\bf H}(W_J)$. 

(b) Under the identification from (a), $\hat {\bf D}(W_J)$ is a subalgebra 
$\hat {\bf D}(W)$ generated by $D_s$, $s\in {\mathcal S}\cap W_J$.

(c) ${\bf K}(W_J) \subset \hat {\bf D}(W_J)\subset \hat {\bf H}(W)$  is a left coideal in $\hat {\bf H}(W)$.


\end{lemma} 


\begin{proof} Indeed, we have a natural homomorphism of algebras $\varphi_J:\hat {\bf H}(W_J)\to \hat {\bf H}(W)$ determined by 
$\varphi_J(s_j)=s_j$, $\varphi_J(D_j)=D_j$, $j\in J$.  Clearly, the restriction of $\varphi_J$ to $\ZZ W_J$ is an injective homomorphism $\ZZ W_J\hookrightarrow \ZZ W$. Also, $\varphi_J(D_s)=D_s$ for $s\in {\mathcal S}_J={\mathcal S}\cap W_J$, which follows from \eqref{eq:si conj Ds} and the fact that ${\mathcal S}_J$ is the set of all reflections in $W_J$. In view of Lemma \ref{le:free product}(a) applied to $\hat {\bf D}(W)=\hat {\bf D}_{\chi,\sigma}(W)$ with $\chi,\sigma$ given by \eqref{eq:chi sigma defined}, the restriction of $\varphi_J$ to  $\hat {\bf D}(W_J)$ is an injective homomorphism $\hat {\bf D}(W_J)\hookrightarrow \hat {\bf D}(W)$. Therefore, by 
Theorem \ref{th:ZWcross and free D}, which asserts factorizations $\hat {\bf H}(W)=\hat {\bf D}(W)\cdot \ZZ W$ and $\hat {\bf H}(W_J)=\hat {\bf D}(W_J)\cdot \ZZ W_J$, the map
$\varphi:\hat {\bf D}(W_J)\cdot \ZZ W_J\to \hat {\bf D}(W)\cdot \ZZ W$ 
is also injective as the tensor product of injective $\ZZ$-linear maps.

This proves (a) and (b).

Prove (c). By Lemma \ref{le:hats are equal}, $\hat {\bf H}(W)=\hat {\bf H}_{\chi,\sigma}(W)$ (for $\chi,\sigma$ defined by  \eqref{eq:chi sigma defined}). Therefore, taking into account that  $\hat {\bf D}(W)={\hat D}_{\chi,\sigma}(W)$, ${\bf K}(W)={\bf K}_{\chi,\sigma}(W)$   is a left coideal in ${\bf H}(W)$ by Lemma \ref{le:free hat H}(c). Replacing $W$ with $W_J$ and  using (a), we finish proof of (c).
\end{proof}

Lemmas \ref{le:graded coideal}, \ref{le:levi}(c),  and Proposition \ref{pr:sum intersection of coideals} guarantee that ${\bf K}(W_J)_{\le r}=\hat {\bf D}(W)_{\le r}\cap {\bf K}(W_J)$ is a left coideal in $\hat {\bf H}(W)$, which is  a free $\ZZ$-module by Lemmas \ref{le:hats are equal} and \ref{le:free hat H}.
 %

The proposition is proved.
\end{proof}

Furthermore, by definition, $\ZZ+{\bf K}_{ij}(W)={\bf K}(W_{\{i,j\}})_{\le m_{ij}}$. This and Propositions \ref{pr:sum intersection of coideals}(a),
\ref{pr:graded coideal J} 
imply that $\underline {\bf K}=\ZZ+\sum\limits_{i,j\in I}{\bf K}_{ij}(W)$ is a left coideal in $\hat {\bf H}(W)$. Proposition \ref{pr:from coideal to hopf ideal} guarantees that the ideal $\underline {\bf J}(W)$ of $\hat {\bf H}(W)$   
generated by $\underline {\bf K}$, is a  Hopf ideal, hence  ${\bf H}(W)=\hat {\bf H}(W)/\underline {\bf J}(W)$ is a Hopf algebra.

Theorem \ref{th:Coideal Hopf ij} is proved.
\endproof

\noindent {\bf Proof of Theorem \ref{th:HW factorization}}.  
$\underline {\bf H}(W)={\bf H}_{\chi,\sigma}(W)$ for $\chi,\sigma$ given by \eqref{eq:chi sigma defined} by the argument from the proof of Theorem \ref{th:Coideal Hopf}. Therefore,  the first assertion of Theorem \ref{th:HW factorization} coincides with the second assertion of Theorem \ref{th:ZWcross and free D gen H}.
 
	%
Prove the second assertion of Theorem \ref{th:HW factorization}. 
We need the following result.

\begin{proposition}
\label{pr:levi K} For any subset $J\subset I$, 
under the natural inclusion $\hat {\bf H}(W_J)\subset \hat {\bf H}(W)$ from Lemma \ref{le:levi}(a), one has ${\bf K}(W_J)\subset {\bf K}(W)$. 

\end{proposition}

\begin{proof} 
We need  the following  immediate consequence of  \eqref{eq:relations hat H} and Proposition \ref{pr:chi sigma satisfies cocycle condition}(b).

\begin{lemma} 
\label{le:conjugation D} 
The following relations hold in  $\hat {\bf H}'(W)$
\begin{equation}
\label{eq:si conj Dsi w}
wD_sw^{-1}=
\begin{cases}
D_{ws w^{-1}} & \text{if $\ell(ws)>\ell(w)$}\\
1-D_{wsw^{-1}}-wsw^{-1} & \text{if $\ell(ws)<\ell(w)$}\\
\end{cases}
\end{equation}
for all $w\in W$, $s\in {\mathcal S}$.

\end{lemma}

Given $J\subset I$, denote $W^J:=\{w\in W\,|\,\ell(ws_j)=\ell(w)+1~\forall~j\in J\}$. 
It is well-known (see e.g., \cite{BW}) that $W$ has a unique factorization  $W=W^J\cdot W_J$, which we write element-wise as 
$w=[w]^J\cdot [w]_J$ 
for any $w\in W$, where $[w]^J\in W^J$ and $[w]_J\in W_J$.

%

\begin{lemma}  
\label{le:W^J conjugation}
For any  Coxeter group $W$, and any subset $J\subset I$ one has

(a) $w\hat {\bf D}(W_J)w^{-1}\subset \hat {\bf D}(W)$ for $w\in W^J$.

(b) $w{\bf K}(W_J)w^{-1}=[w]^J{\bf K}(W_J)([w]^J)^{-1}\subset {\bf K}(W)$ for all $w\in W$. 

\end{lemma}

\begin{proof} It is easy to see that  
$\ell(w_1w_2)=\ell(w_1)+\ell(w_2)$
for any $w_1\in W^J$, $w_2\in W_J$. 
This and \eqref{eq:si conj Dsi w} imply that  $wD_sw^{-1}=D_{wsw^{-1}}$ in $\hat {\bf D}(W)$ for all $w\in W^J$, $s\in {\mathcal S}_J={\mathcal S}\cap W_J$.
Hence $w\hat {\bf D}(W_J)w^{-1}\subset \hat {\bf D}(W)$ for all $w\in W^J$. This proves (a).

Prove (b) now. We have, based on the proof of (a): 
$$w{\bf K}(W_J)w^{-1}=[w]^J[w]_J\hat {\bf K}(W_J)([w]_J)^{-1}([w]^J)^{-1}=[w]^J{\bf K}(W_J)([w]^J)^{-1}\subset \hat {\bf D}(W)$$
for all $w\in W$ because $w_1 {\bf K}(W_J)w_1^{-1}={\bf K}(W_J)$, $w_2 {\bf K}(W_J)w_2^{-1}\subset \hat {\bf D}(W_J)$ for all $w_1\in W_J$, $w_2\in W^J$. 

In particular, ${\bf K}(W_J)\subset {\bf K}(W)$. Conjugating with $w$ and using the fact that $w{\bf K}(W)w^{-1}={\bf K}(W)$ for all $w\in W$, we finish the proof of (b). 

The lemma is proved.
\end{proof} 

Therefore, the proposition is proved.
\end{proof}

Let $\underline {\bf K}(W):=\sum\limits_{w\in W, i,j\in I:i\ne j} w{\bf K}_{ij}(W)w^{-1}$. By definition, $w\cdot \underline {\bf K}(W)=\underline {\bf K}(W)\cdot w$ for all $w\in W$ and
the ideal $\underline {\bf J}(W)$ from the proof of Theorem \ref{th:Coideal Hopf ij} is generated by $\underline {\bf K}(W)$.
Also, $\underline {\bf K}(W)\subset \hat {\bf D}(W)$ by Proposition \ref{pr:levi K}.  
Therefore ${\bf H}=\hat {\bf H}(W)$, ${\bf D}=\hat {\bf D}(W)$, and ${\bf K}=\underline {\bf K}(W)\cap Ker~\varepsilon$ satisfy the assumptions of Lemma \ref{le:general factored quotient H}, thus $\underline {\bf H}={\bf H}(W)$ factors as ${\bf H}(W)=\underline {\bf D}\cdot \ZZ W$ over $\ZZ$, where $\underline {\bf D}={\bf D}(W)=\hat {\bf D}(W)/\langle \underline {\bf K}(W)\rangle$.

Theorem \ref{th:HW factorization} is proved. \endproof



\subsection{Relations in ${\bf D}(W)$ and proof of Theorem \ref{th:relations rank 2}}
\label{subsec:proof of th:relations rank 2}
For all distinct $i,j\in I$, $w\in W^{\{i,j\}}$, and $s\in {\mathcal S}\cap W_{\{i,j\}}$ we have $wD_sw^{-1}=D_{wsw^{-1}}$ by \eqref{eq:si conj Dsi w}. Therefore,  it suffices to prove the assertion only when $w=1$ and $W=W_{\{i,j\}}$. 
Define $Q_{ij}^{(n,r,p)}$ and $R_{ij}^{(n,r,t)}\in {\hat D}(W_{\{i,j\}})$  for all divisors $n$ of $m=m_{ij}$, $r\in [1,n]$, and  $1\le p<\frac{m}{2n}$, $0\le t\le \frac{m}{n}$ by:
$$Q_{ij}^{(n,r,p)}=\sum_{0\le a<b< \frac{m}{n}:b-a=\frac{m}{n}-p} D_{r+bn} D_{r+an} - \sum_{0\le a'<b'< \frac{m}{n}:b'-a'=p} D_{r+a'n} D_{r+b'n}+\sum_{p\le c<\frac{m}{n}-p} D_{r+cn}\ ,$$
$$R_{ij}^{(n,r)}= D_r D_{r+n} \cdots D_{r+m-n}-D_{r+m-n} \cdots D_{r+n} D_r \ .$$
$$R_{ij}^{(n,r,t)}:=
\overset{\longrightarrow}{\prod\limits_{t\le a\le \frac{m}{n}-1}}(1-D_{r+an})
\overset{\longrightarrow}{\prod\limits_{0\le b\le t-1}} D_{r+bn}
-\overset{\longleftarrow}{\prod\limits_{0\le b\le t-1}} D_{r+bn}
\overset{\longleftarrow}{\prod\limits_{t\le a\le \frac{m}{n}-1}}(1-D_{r+an})
\ .$$

We need the following fact.

\begin{proposition}
\label{pr:Dij2} For any Coxeter group $W$ and $i,j\in I$ with $m:=m_{ij}\ge 2$ one has for all divisors $n$ of $m=m_{ij}$ and $r\in [1,n]$: 

(a) 
$s_iQ_{ij}^{(n,r,p)}s_i=
\begin{cases} Q_{ij}^{(n,r-1,p)} & \text{if $r>1$}\\
Q_{ji}^{(n,n,p)} & \text{if $r=1$}\\
\end{cases}$, $s_jQ_{ij}^{(n,r,p)}s_j=
\begin{cases} Q_{ij}^{(n,r+1,p)} & \text{if $r<n$}\\
Q_{ji}^{(n,1,p)} & \text{if $r=n$}\\
\end{cases}$ for $1\le p<\frac{m}{2n}$.

(b) $s_iR_{ij}^{(n,r,t)}s_i=
\begin{cases} R_{ij}^{(n,r-1,t)} & \text{if $r>1$}\\
R_{ji}^{(n,n,t-1)} & \text{if $r=1$, $t\ge 1$}\\
-R_{ij}^{(n,1,1)} & \text{if $r=1$, $t=0$}\\
\end{cases}$,  $s_iR_{ij}^{(n,r,t)}s_i=
\begin{cases} R_{ij}^{(n,r+1,t)} & \text{if $r<n$}\\
R_{ji}^{(n,1,t+1)} & \text{if $r=n$, $t< \frac{m}{n}$}\\
-R_{ij}^{(n,1,0)} & \text{if $r=n$, $t=\frac{m}{n}$}\\
\end{cases}$ for
$1\le t\le \frac{m}{n}$.
\end{proposition}

\begin{proof} We need the following immediate consequence of \eqref{eq:si conj Ds}.

\begin{lemma} 
\label{le:Dkij}
If $m:=m_{ij}\ge 2$, then one has in $\hat {\bf D}(W_{\{i,j\}})$:
\begin{equation}
\label{eq:conjugation si Dk}
s_iD_ks_i=
\begin{cases} 1-D_1-s_i & \text{if $k=1$}\\
D_{m+2-k} & \text{if $2\le k\le m$}\\
\end{cases}
\end{equation}
for $k=1,\ldots,m$, where  $D_k:=D_k^{i,j}=D_{\underbrace{s_is_j\cdots s_i}_{2k-1}}$ for $k=1,\ldots,m$.

\end{lemma}

Taking into account that $D^{ji}_k=D_{m+1-k}^{i,j}$, we will repeatedly use \eqref{eq:conjugation si Dk} in the form:
\begin{equation}
\label{eq:conjugation si Dk ji}
s_iD_ks_i=\begin{cases} 1-D_m^{ji}-s_i & \text{if $k=1$}\\
D_{k-1}^{ji} & \text{if $2\le k\le m$}\\
\end{cases}
,~s_iD_\ell^{ji}s_i=D_{m-\ell}^{ji}
\end{equation}
for $k=1,\ldots,m$, $\ell=1,\ldots,m-1$.

Prove (a). 
First, suppose that $r>1$. Then, using \eqref{eq:conjugation si Dk ji}, we have
$$s_iQ_{ij}^{(n,r,p)}s_i=\sum_{0\le a<b< \frac{m}{n}:b-a=\frac{m}{n}-p} s_iD_{r+bn} D_{r+an}s_i $$
$$- \sum_{0\le a'<b'< \frac{m}{n}:b'-a'=p} s_iD_{r+a'n} D_{r+b'n}s_i+\sum_{p\le c<\frac{m}{n}-p} s_iD_{r+cn}s_i$$
$$=\sum_{0\le a<b< \frac{m}{n}:b-a=\frac{m}{n}-p} D_{r-1+bn}^{ji} D^{ji}_{r-1+an}$$
$$- \sum_{0\le a'<b'< \frac{m}{n}:b'-a'=p} D_{r-1+a'n}^{ji} D_{r-1+b'n}^{ji}+\sum_{p\le c<\frac{m}{n}-p} D_{r-1+cn}^{ji}=Q_{ji}^{n,r-1,p}$$
Interchanging $i$ and $j$, we also obtain $s_jQ_{ij}^{(n,r,p)}s_j=Q_{ji}^{(n,r+1,p)}$ whenever $r<n$.

Finally, suppose that $r=1$. Then, using \eqref{eq:conjugation si Dk ji} again, we have
$$s_iQ_{ij}^{(n,1,p)}s_i=s_iD_{1+m-pn} D_1s_i+\sum_{0<a<b< \frac{m}{n}:b-a=\frac{m}{n}-p} s_iD_{1+bn} D_{1+an}s_i $$
$$-s_iD_1 D_{1+pn}s_i- \sum_{0< a'<b'< \frac{m}{n}:b'-a'=p} s_iD_{1+a'n} D_{1+b'n}s_i+\sum_{p\le c<\frac{m}{n}-p} s_iD_{1+cn}s_i$$
$$=D^{ji}_{m-pn}(1-D^{ji}_m-s_i)+\sum_{0< a<b< \frac{m}{n}:b-a=\frac{m}{n}-p} D^{ji}_{bn} D^{ji}_{an} $$
$$-(1-D^{ji}_m-s_i)D^{ji}_{pn}- \sum_{0< a'<b'< \frac{m}{n}:b'-a'=p} D^{ji}_{a'n} D^{ji}_{b'n}+\sum_{p\le c<\frac{m}{n}-p} D^{ji}_{cn}$$
$$=D^{ji}_m D^{ji}_{pn}-D^{ji}_{pn}+ D^{ji}_{m-pn} +\sum_{0<a<b< \frac{m}{n}:b-a=\frac{m}{n}-p} D^{ji}_{bn} D^{ji}_{an} $$
$$-D^{ji}_{m-pn}D^{ji}_m-\sum_{0<a'<b'< \frac{m}{n}:b'-a'=p} D^{ji}_{a'n} D^{ji}_{b'n}+\sum_{p\le c<\frac{m}{n}-p} D^{ji}_{cn}$$
$$=\sum_{0\le a-1<b-1< \frac{m}{n}:b-a=\frac{m}{n}-p} D^{ji}_{bn} D^{ji}_{an}-\sum_{0\le a'-1<b'-1< \frac{m}{n}:b'-a'=p} D^{ji}_{a'n} D^{ji}_{b'n}+\sum_{p\le c-1<\frac{m}{n}-p} D^{ji}_{cn}=Q_{ji}^{(n,n,p)}$$

Interchanging $i$ and $j$, we obtain $s_jQ_{ij}^{(n,n,p)}s_j=Q_{ji}^{(n,1,p)}$. 
This proves (a).


Prove (b) now.
First, suppose that $r>1$. Then, using \eqref{eq:conjugation si Dk ji}, we have
$$s_iR_{ij}^{(n,r,t)}s_i=
\overset{\longrightarrow}{\prod\limits_{t\le a\le \frac{m}{n}-1}}(1-s_iD_{r+an}s_i)
\overset{\longrightarrow}{\prod\limits_{0\le b\le t-1}} s_iD_{r+bn}s_i
-\overset{\longleftarrow}{\prod\limits_{0\le b\le t-1}} s_iD_{r+bn}s_i
\overset{\longleftarrow}{\prod\limits_{t\le a\le \frac{m}{n}-1}}(1-s_iD_{r+an}s_i)$$
$$=\overset{\longrightarrow}{\prod\limits_{0\le b\le t-1}} D^{ji}_{r-1+bn}
\overset{\longrightarrow}{\prod\limits_{t\le a\le \frac{m}{n}-1}}(1-D^{ji}_{r-1+an})-\overset{\longleftarrow}{\prod\limits_{0\le b\le t-1}} D^{ji}_{r-1+bn}
\overset{\longleftarrow}{\prod\limits_{t\le a\le \frac{m}{n}-1}}(1-D^{ji}_{r-1+an})=R_{ji}^{(n,r-1,t)}\ .$$
%
Interchanging $i$ and $j$, we also obtain $s_jR_{ij}^{(n,r,t)}s_j=R_{ji}^{(n,r+1,t)}$ whenever $r<n$.

Now suppose that $r=1$, $t\ge 1$. Then, using \eqref{eq:conjugation si Dk ji} again, we have
$$s_iR_{ij}^{(n,1,t)}s_i=\overset{\longrightarrow}{\prod\limits_{t\le a\le \frac{m}{n}-1}}(1-s_iD_{1+an}s_i)
\overset{\longrightarrow}{\prod\limits_{0\le b\le t-1}} s_iD_{1+bn}s_i
-\overset{\longleftarrow}{\prod\limits_{0\le b\le t-1}} s_iD_{1+bn}s_i
\overset{\longleftarrow}{\prod\limits_{t\le a\le \frac{m}{n}-1}}(1-s_iD_{1+an}s_i)$$
$$=\overset{\longrightarrow}{\prod\limits_{t\le a\le \frac{m}{n}-1}}(1-D^{ji}_{an})\cdot
(1-D_m^{ji}-s_i)\overset{\longrightarrow}{\prod\limits_{1\le b\le t-1}} D^{ji}_{bn}
-\overset{\longleftarrow}{\prod\limits_{1\le b\le t-1}} D^{ji}_{bn}\cdot (1-D_m^{ji}-s_i)
\overset{\longleftarrow}{\prod\limits_{t\le a\le \frac{m}{n}-1}}(1-D^{ji}_{an})$$
%
$$=\overset{\longrightarrow}{\prod\limits_{t\le a\le \frac{m}{n}-1}}(1-D^{ji}_{an})\cdot
(1-D_m^{ji})\overset{\longrightarrow}{\prod\limits_{1\le b\le t-1}} D^{ji}_{bn}
-\overset{\longleftarrow}{\prod\limits_{1\le b\le t-1}} D^{ji}_{bn}\cdot (1-D_m^{ji})
\overset{\longleftarrow}{\prod\limits_{t\le a\le \frac{m}{n}-1}}(1-D^{ji}_{an})=R_{ji}^{(n,n,t-1)}$$
%
%
because $-\overset{\longrightarrow}{\prod\limits_{t\le a\le \frac{m}{n}-1}}(1-D^{ji}_{an})\cdot s_i
\overset{\longrightarrow}{\prod\limits_{1\le b\le t-1}} D^{ji}_{bn}
+\overset{\longleftarrow}{\prod\limits_{1\le b\le t-1}} D^{ji}_{bn}\cdot s_i
\overset{\longleftarrow}{\prod\limits_{t\le a\le \frac{m}{n}-1}}(1-D^{ji}_{an})$
$$=-s_i\left(\overset{\longrightarrow}{\prod\limits_{t\le a\le \frac{m}{n}-1}}(1-D^{ji}_{m-an})
\overset{\longrightarrow}{\prod\limits_{1\le b\le t-1}} D^{ji}_{bn}
+\overset{\longleftarrow}{\prod\limits_{1\le b\le t-1}} D^{ji}_{m-bn}\overset{\longleftarrow}{\prod\limits_{t\le a\le \frac{m}{n}-1}}(1-D^{ji}_{an})\right)=0$$
%
which is immediate if $t=1$ or $t=\frac{m}{n}$ and follows from the relations $D_s^2=D_s$ if $1< t< \frac{m}{n}$ 
(which we use here in the form $(1-D_n^{ji})D_n^{ji}=D_n^{ji}(1-D_n^{ji})=0$). 

Interchanging $i$ and $j$, we obtain $s_jR_{ij}^{(n,n,t)}s_j=R_{ji}^{(n,1,t+1)}$ whenever  $1\le t<\frac{m}{n}$.

Finally, suppose that $r=1$, $t=0$. Then, using \eqref{eq:conjugation si Dk}, we have
$$s_iR_{ij}^{(n,1,0)}s_i=\overset{\longrightarrow}{\prod\limits_{0\le a\le \frac{m}{n}-1}}(1-s_iD_{1+an}s_i)-
\overset{\longleftarrow}{\prod\limits_{t\le a\le \frac{m}{n}-1}}(1-s_iD_{1+an}s_i)$$
$$=(D_1+s_i)\cdot \overset{\longrightarrow}{\prod\limits_{1\le a\le \frac{m}{n}-1}}(1-D_{m+1-an})-
\overset{\longleftarrow}{\prod\limits_{1\le a\le \frac{m}{n}-1}}(1-D_{m+1-an})\cdot (D_1+s_i)$$
$$=D_1\cdot \overset{\longrightarrow}{\prod\limits_{1\le a\le \frac{m}{n}-1}}(1-D_{m+1-an})-
\overset{\longleftarrow}{\prod\limits_{1\le a\le \frac{m}{n}-1}}(1-D_{m+1-an})\cdot D_1$$
$$=D_1\cdot \overset{\longleftarrow}{\prod\limits_{1\le a'\le \frac{m}{n}-1}}(1-D_{1+a'n})-
\overset{\longleftarrow}{\prod\limits_{1\le a'\le \frac{m}{n}-1}}(1-D_{1+a'n})\cdot D_1=-R_{ij}^{(n,1,1)}$$
%
because $s_i\underbrace{(1-D_{m+1-n})\cdots (1-D_{n+1})}_{\frac{m}{n}-1}-\underbrace{ (1-D_{n+1})\cdots (1-D_{m+1-n})}_{\frac{m}{n}-1}s_i=0$.

Interchanging $i$ and $j$, we obtain $s_jR_{ij}^{(n,1,1)}s_j=-R_{ij}^{(n,1,0)}$ whenever  $1\le t<\frac{m}{n}$.
This proves (b). 

The proposition is proved.
\end{proof}

Finally, Proposition \ref{pr:Dij2} implies that all $Q_{ij}^{(n,r,p)}$, $\overline{Q_{ij}^{(n,r,p)}}$ and $R_{ij}^{(n,r,t)}$ belong to ${\bf K}_{ij}(W)$, where $\overline {\cdot }$ is the anti-involution of $\hat {\bf D}(W)$ given by $\overline D_s=D_s$ for $s\in \in{\mathcal S}$ (see also Theorem \ref{th:symmetris D(W)}(a) and its proof).  This proves Theorem \ref{th:relations rank 2}.
\endproof

\subsection{Braid relations and proof of Theorems  \ref{th:Hecke in Hecke-Hopf} and \ref{th:upper Hecke in Hecke-Hopf}}
\label{subsec:proof of Theorem Hecke in Hecke-Hopf}

For commutative ring  $\kk$,  $i,j\in I$ with $m_{ij}\ge 2$,  $c_i,c_j\in \kk$ such that $c_i=c_j$ if $m_{ij}$ is odd, define the element in $\hat {\bf H}(W)\otimes \kk$  by:
$$\Delta_{ij}^{c_i,c_j}=w_{ij}(\underbrace{\cdots (s_i-c_is_iD_{s_i})(s_j-c_js_jD_{s_j})}_m
-\underbrace{\cdots (s_j-c_js_jD_{s_j})(s_i-c_is_iD_{s_i})}_m)$$
where $m:=m_{ij}$ and $w_{ij}=\underbrace{s_is_j\cdots }_{m_{ij}}=\underbrace{s_js_i\cdots}_{m_{ij}}$ is the longest element in $W_{\{i,j\}}$.
It is easy to see that 
\begin{equation}
\label{eq:antipode braid yang baxter}
S^{-1}(w_{ij}\Delta_{ij}^{c_i,c_j})=\underbrace{\cdots (s_i+c_iD_{s_i})(s_j+c_jD_{s_j})}_m-\underbrace{\cdots (s_j+s_jD_{s_j})(s_i+c_iD_{s_i})}_m
\end{equation}
and
\begin{equation}
\label{eq:Delta factored into roots}
\Delta_{ij}^{c_i,c_j}=(1-c_{i_1}D_1)\cdots (1-c_{i_m}D_m)-(1-c_{i_m}D_m)\cdots (1-c_{i_1}D_1)
\end{equation}
where $D_k=D_k^{i,j}$ are as in Lemma \ref{le:Dkij}, $m=m_{ij}$, and $(i_1,\ldots,i_m)=\underbrace{(\ldots,i,j)}_m$.

In particular, $\Delta_{ij}^{c_i,c_j}\in \hat {\bf D}(W)\otimes \kk$ for all $i,j$ with $m_{ij}\ge 2$ and $c_i,c_j\in \kk$.

\begin{proposition} 
\label{pr:Delta in K}
In the assumptions as above, each $\Delta_{ij}^{c_i,c_j}$ belongs to ${\bf K}_{ij}(W)\otimes \kk$.

\end{proposition}

\begin{proof} We need the following result.

\begin{lemma} 
\label{le:siDeltasi}
For all $i,j\in I$ with $m_{ij}\ge 2$, $c_i,c_j\in \kk^\times$ such that $c_i=c_j$ if $m_{ij}$ is odd, one has:
\begin{equation}
\label{eq:siDeltasi}
s_i\Delta_{ij}^{c_i,c_j}s_i=\frac{1}{c_i-1}D_{i,c_i}\Delta_{ij}^{c_i,c_j}D_{i,c_i},~s_j\Delta_{ij}^{c_i,c_j}s_j=\frac{1}{c_j-1}D_{j,c_j}\Delta_{ij}^{c_i,c_j}D_{j,c_j} \ ,
\end{equation}
where $D_{i,c_i}=(1-c_i)(1-D_{s_i})=(1-c_i)(1-c_iD_{s_i})^{-1}$. 
\end{lemma}

\begin{proof} Prove the first equation \eqref{eq:siDeltasi}. Using  relations \eqref{eq:conjugation si Dk} we obtain
$$s_i\Delta_{ij}^{c_i,c_j}s_i=s_i(1-c_{i_1}D_1)\cdots (1-c_{i_m}D_m)s_i-s_i(1-c_{i_m}D_m)\cdots (1-c_{i_1}D_1)s_i$$
$$=(1-c_i(1-D_{s_i}-s_i))(1-c_{i_m}D_m)\cdots (1-c_{i_2}D_2)-(1-c_{i_2}D_2)\cdots (1-c_{i_m}D_m)(1-c_i(1-D_{s_i}-s_i))$$
$$=D_{i,c_i}(1-c_{i_m}D_m)\cdots (1-c_{i_2}D_2)-(1-c_{i_2}D_2)\cdots (1-c_{i_m}D_m)D_{i,c_i}$$
because
$s_i(1-c_{i_m}D_m)\cdots (1-c_{i_2}D_2)=(1-c_{i_2}D_2)\cdots (1-c_{i_m}D_m)s_i$.

Taking into account that $D_{i,c_i}=(1-c_i)\cdot (1-c_iD_{s_i})^{-1}$, we obtain the first equation \eqref{eq:siDeltasi}. The second one also follows because $\Delta_{ji}^{c_j,c_i}=-\Delta_{ij}^{c_i,c_j}$.

The lemma is proved. \end{proof}

Conjugating $\Delta_{ij}^{c_i,c_j}$ with $w=\underbrace{\cdots s_js_i}_\ell\in W_{\{i,j\}}$, $\ell\le m$, and using \eqref{eq:siDeltasi} repeatedly, we obtain:
$$w\Delta_{ij}^{c_i,c_j}w^{-1}=\left(\prod\limits_{k=1}^\ell\frac{1}{c_{i_\ell}-1} \right)\cdot  \tilde D_\ell\cdots \tilde D_1\Delta_{ij}^{c_i,c_j}\tilde D_1\cdots \tilde D_\ell$$
where we abbreviate $\tilde D_k=s_{i_1}\cdots s_{i_{k-1}}D_{i_k,c_{i_k}}s_{i_{k-1}}\cdots s_{i_1}=1-c_{i_k}(1-D_k)$ for $k=1,\ldots,\ell$ in the notation of \eqref{eq:Delta factored into roots}, where $i_k=\begin{cases}  i & \text{if $k$ is odd}\\
j & \text{if $k$ is even}
\end{cases}$.
This implies that $w\Delta_{ij}^{c_i,c_j}w^{-1}\in {\bf D}(W_{\{i,j\}})\otimes \kk$. Similarly, taking  $w=\underbrace{\cdots s_is_j}_\ell\in W_{\{i,j\}}$, $\ell\le m$, one shows that $w\Delta_{ij}^{c_i,c_j}w^{-1}\in {\bf D}(W_{\{ij\}})_{\le m}\otimes \kk$. 
Thus, $w\Delta_{ij}^{c_i,c_j}w^{-1}\in {\bf D}(W_{\{i,j\}})\otimes \kk$ for any $w\in W_{\{i,j\}}$ for any $\kk$ and any $c_i,c_j\in \kk$ such that $c_i=c_j$ if $m_{ij}$ is odd. Suppose that $\kk$ is a free $\ZZ$-module. This implies that 
$\Delta_{ij}^{c_i,c_j}\in \bigcap\limits_{w\in W} w\cdot \left({\bf D}(W_{\{i,j\}})_{\le m}\otimes \kk\right)\cdot w^{-1}$ 
where the intersection is in $\hat {\bf H}(W_{\{i,j\}})\otimes \kk$.

Suppose that $\kk$ is a free $\ZZ$-module. Then, taking into account that  
$$\bigcap_{w\in W} w\cdot ({\bf D}(W_{\{i,j\}})_{\le m}\otimes \kk)\cdot w^{-1}= \left(\bigcap_{w\in W} w\cdot {\bf D}(W_{\{ij\}})_{\le m}\cdot w^{-1}\right)\otimes \kk$$ by Lemma \ref{le:intersection tensor products} and 
$\bigcap\limits_{w\in W} w\cdot {\bf D}(W_{\{ij\}})_{\le m}\cdot w^{-1}=\left(\bigcap\limits_{w\in W} w\cdot {\bf D}(W_{\{ij\}})\cdot w^{-1}\right)_{\le m}=K_{ij}(W),$
 we obtain 
$\Delta_{ij}^{c_i,c_j}\in {\bf K}_{ij}(W)\otimes \kk$. 

Finally, we can remove freeness condition for $\kk$ over $\ZZ$ by first replacing $\kk$ with a commutative ring $\hat \kk$ free over $\ZZ$  and then noting that any commutative ring $\kk$ is a homomorphic image of some $\hat \kk$. Then extending the structural homomorphism $f:\hat \kk\twoheadrightarrow \kk$ to $f:\hat {\bf H}(W_{\{i,j\}})\otimes \hat \kk\twoheadrightarrow \hat {\bf H}(W_{\{i,j\}})\otimes \kk$ we see that an inclusion 
$\Delta_{ij}^{\hat c_i,\hat c_j}\in {\bf K}_{ij}(W)\otimes \hat \kk$ implies an inclusion 
$\Delta_{ij}^{c_i,c_j}\in {\bf K}_{ij}(W)\otimes \kk$, where $c_i=f(\hat c_i)$, $c_j=f(\hat c_j)$.
 
The proposition is proved.
\end{proof}

\noindent {\bf Proof of Theorems \ref{th:Hecke in Hecke-Hopf} and \ref{th:upper Hecke in Hecke-Hopf}}. Indeed, the braid relations between $T'_i=s_i+c_iD_{s_i}$ and $T'_j=s_j+c_jD_{s_j}$ (where $c_i=1-q_i$, $c_j=1-q_j$) in ${\bf H}(W)\otimes \kk$ follow from  \eqref{eq:antipode braid yang baxter} and Proposition \ref{pr:Delta in K} because $S^{-1}(w_{ij}\Delta_{ij}^{c_i,c_j})\in {\bf K}_{ij}(W)\otimes \kk$.

It remains to verify the quadratic relations for $T'_i$. Indeed, one has:
$${T'_i}^2=(s_i+(1-q_i)D_{s_i})^2=s_i^2+(1-q_i)(s_iD_{s_i}+D_{s_i}s_i)+(1-q_i)^2D_{s_i}^2$$
$$=1+(1-q_i)(s_i-1)+(1-q_i)^2D_{s_i}=(1-q_i)T'_i+q_i\ .$$

This proves that there is a unique  homomorphism of algebras $\varphi_W:H_{\bf q}(W)\to {\bf H}(W)\otimes \kk$ such that $\varphi_W(T_i)=T'_i$ for $i\in I$, as in Theorem \ref{th:upper Hecke in Hecke-Hopf}. 

Furthermore, Proposition \ref{pr:levi K} with $J=\{i,j\}$ guarantees the inclusions ${\bf K}_{ij}(W)\subset {\bf K}(W_{\{i,j\}})\subset {\bf K}(W)$ hence we have a surjective homomorphism of Hopf algebras $\pi_W:{\bf H}(W)\twoheadrightarrow \underline {\bf H}(W)$ as in Theorem 
\ref{th:Hecke in Hecke-Hopf}. Denote $\underline \varphi_W:=(\pi_W\otimes 1)\circ \varphi_W$, which is a homomorphism $H_{\bf q}(W)\to \underline {\bf H}(W)\otimes \kk$, as in Theorem \ref{th:Hecke in Hecke-Hopf}

Let us prove it injectivity of $\underline \varphi_W$ (the injectivity of $\varphi_W$ will follow verbatim).

Recall that for each $w\in W$ there is a unique element $T_w\in H_{\bf q}(W)$ such that $T_w=T_{i_1}\cdots T_{i_m}$ for any reduced decomposition $w=s_{i_1}\cdots s_{i_m}$ in $W$. Clearly, the elements $T_w$ generate $H_{\bf q}(W)$ as a $\kk$-module (in fact, they form a $\kk$-basis - see Corollary \ref{cor:Tw form a basis} below).

Thus, to prove injectivity of $\underline \varphi_W$, it suffices to show that the images $\underline \varphi_W(T_w)$ are $\kk$-linearly independent in ${\bf H}(W)\otimes \kk$. 

\begin{proposition} 
\label{pr:strong Bruhat phi(Tw)}
For each $w\in W$ one has:
$\underline \varphi_W(T_w)\in w+\sum_{w':w'\prec w} \kk\cdot {\bf D}(W)\cdot w'$, 
where $\prec$ denotes the strong Bruhat order on $W$. 
\end{proposition}

\begin{proof} For each $w\in W$ denote $W_{\prec w}:=\{w'\in W:w'\prec w\}$ and $W_{\preceq w}:=\{w\}\sqcup W_{\prec w}$.

We need the following fact.

\begin{lemma} $W_{\prec w}\cdot D_i\in {\bf D}(W)\cdot W_{\prec w}$ for any $w\in W$.

\end{lemma}

\begin{proof} Since $W_{\preceq \tilde w}\subset W_{\prec w}$ for any $\tilde w\prec w$, it suffices to show that 
\begin{equation}
\label{eq:tilde w D_i}
\tilde w\cdot D_i\in {\bf D}\cdot W_{\preceq \tilde w}
\end{equation}
in ${\bf H}(W)$ for all $\tilde w\in W$, $i\in I$.

Indeed, by definition of generators $D_s$ of ${\bf D}(W)$, which are images of their counterparts in 
$\hat {\bf D}(W)$,  if $\ell(\tilde ws_i)=\ell(\tilde w)+1$, then \eqref{eq:si conj Dsi w} implies that 
$\tilde w\cdot D_i=D_{\tilde ws_i\tilde w^{-1}}\cdot \tilde w\in {\bf D}(W)\cdot W_{\preceq \tilde w}$.

Otherwise, i.e., if $\tilde w\in W$ is such that $\ell(\tilde ws_i)=\ell(\tilde w)-1$, then using \eqref{eq:si conj Dsi w} again, we obtain
$$\tilde w\cdot D_i=\tilde ws_i\cdot (-D_is_i+s_i-1)=-D_{\tilde ws_i\tilde w^{-1}}\cdot \tilde w+\tilde w-
\tilde ws_i\in {\bf D}(W)\cdot W_{\preceq \tilde w}$$ 

since $\tilde ws_i\cdot D_i=D_{\tilde ws_i\tilde w^{-1}}\cdot \tilde ws_i$.

The lemma is proved.
\end{proof}

The following finishes the proof of the proposition.
\begin{lemma} For all $w\in W$ one has
$\underline \varphi_W(T_w)\in w+\kk\cdot {\bf D}(W)\cdot W_{\prec w}$.
\end{lemma}

\begin{proof}
We will prove the assertion by induction in length $\ell(w)$. 
Indeed, if $w=1$, we have nothing to prove. 
Suppose $w\ne 1$, then choose $i\in I$ such that $\ell(ws_i)=\ell(w)-1$ 
(or, equivalently, $ws_i\prec w$). Using the inductive hypothesis for $ws_i$ and that  
$\underline \varphi_W(T_w)=\underline \varphi_W(T_{ws_i})\underline \varphi_W(T_i)$, we obtain:
$$\underline \varphi_W(T_w)=\underline \varphi_W(T_{ws_i})(s_i+(1-q_i)D_i)\in  (ws_i+\kk\cdot {\bf D}(W)\cdot W_{\prec ws_i})(s_i+(1-q_i)D_i)\subset $$
$$\subset w+(1-q_i)ws_iD_i+\kk\cdot {\bf D}(W)\cdot W_{\prec ws_i}\cdot s_i+\kk\cdot {\bf D}(W)\cdot W_{\prec ws_i}\cdot D_i\subset$$
$$\subset w+(1-q_i){\bf D}(w)ws_i+\kk\cdot {\bf D}(W)\cdot W_{\prec w}+\kk\cdot {\bf D}(W)\cdot W_{\prec ws_i}
=\kk\cdot {\bf D}(W)\cdot W_{\prec w}$$
because $W_{\preceq ws_i}\cup (W_{\prec ws_i}\cdot s_i)=W_{\prec w}$ 
for any $w\in W$ and $i\in I$ such that $\ell(ws_i)=\ell(w)-1$.

The lemma is proved.
\end{proof}

Therefore, Proposition \ref{pr:strong Bruhat phi(Tw)} is proved.
\end{proof}

Finally, Theorem \ref{th:HW factorization} implies that elements $w\in W$ are $\kk$-linearly independent in ${\bf H}(W)\otimes \kk$. This and Proposition \ref{pr:strong Bruhat phi(Tw)} imply that the elements $\underline \varphi_W(T_w)$, $w\in W$ are also $\kk$-linearly independent in ${\bf H}(W)\otimes \kk$.

This proves that $\underline \varphi_W$ is an injective homomorphism of algebras $H_{\bf q}(W)\hookrightarrow {\bf H}(W)$. Injectivity of 
$\varphi_W$ is then immediate. 

Theorems \ref{th:Hecke in Hecke-Hopf} and \ref{th:upper Hecke in Hecke-Hopf} are proved.
\endproof

\subsection{Symmetries of ${\bf H}(W)$ and proof of Theorems \ref{th:symmetris D(W)}, \ref{th:symmetris hat D(W) gen}, and \ref{th:symmetris Dchisigma(W)}}
\label{subsec:proof of Theorem symmetris D(W)} 
We need the following

\begin{proposition}
\label{pr:symmetris hat D(W) gen} In the notation of Theorem \ref{th:hopf hat W gen} we have:

\noindent (a)  Suppose that $\overline{\cdot}$ is an involution on $R$ such that $\overline \chi_{w,s}=\chi_{w,s^{-1}}$, $\overline \sigma_{w,s}=\sigma_{w,s^{-1}}$ for all $w\in W$, $s\in {\mathcal S}$. Then the assignments $\overline w=w^{-1}$, $\overline D_s=D_{s^{-1}}$ for $w\in W$, $s\in {\mathcal S}$ extends to a unique $R$-linear anti-involution of $\hat {\bf H}_{\chi,\sigma}(W)$.

\noindent (b) Suppose that $R W$ admits an $R$-linear automorphism $\theta$ such that $\theta(w)\in R^\times \cdot w$ for $w\in W$ and  
$\theta(s)=\chi_{s,s}\cdot s$ for  $s\in {\mathcal S}$. Then  $\theta$ uniquely  extends to an $R$-linear automorphism 
of $\hat {\bf H}_{\chi,\sigma}(W)$ such that $\theta(D_s)= \chi_{s,s}D_s+\sigma_{s,s}$ for $s\in {\mathcal S}$. Moreover, $\theta({\bf K}_{\chi,\sigma}(W))={\bf K}_{\chi,\sigma}(W)$.

\noindent (c) In the assumptions of Theorem \ref{th:ZWcross and free D gen}, suppose that $\sigma_{wsw^{-1},wsw^{-1}}=\sigma_{s,s}$ for all $w\in W$,  $s\in {\mathcal S}$ of finite order. Then   $\hat {\bf D}_{\chi,\sigma}(W)$ admits:

(i) A $W$-action by automorphisms via $w(D_s)= 
\sigma_{w,s}+\chi_{w,s}D_{wsw^{-1}}$  
for $w\in W$, $s\in {\mathcal S}$. 

(ii) An $s$-derivation $d_s$ (i.e., $d_s(xy)=d_s(x)y+s(x)d_s(y)$) 
such that $d_s(D_{s'})=\delta_{s,s'}$, $s,s'\in {\mathcal S}$.

\noindent These actions satisfy for all $w\in W$, $s\in {\mathcal S}$:
\begin{equation}
\label{eq:w-conjugate ds}
d_{wsw^{-1}}=\chi_{w,s} \cdot w\circ d_s\circ w^{-1}
\end{equation}


\end{proposition}

\begin{proof}
Prove (a). It suffices to verify that $\overline{\cdot}$ preserves the defining relations of $\hat {\bf H}_{\chi,\sigma}(W)$. Indeed,
$\overline {w^{-1}}\cdot \overline D_s \cdot \overline w = w D_{s^{-1}}w^{-1}
=\chi_{w,s^{-1}}D_{ws^{-1}w^{-1}}+\sigma_{w,s^{-1}}(1-ws^{-1}w^{-1})$
$$=\overline \chi_{w,s}\overline D_{wsw^{-1}}+\overline \sigma_{w,s}(1-ws^{-1}w^{-1})=\overline{\chi_{w,s}D_{wsw^{-1}}+\sigma_{w,s}(1-wsw^{-1})}=\overline {wD_sw^{-1}}$$
for all $w\in W$, $s\in {\mathcal S}$, i.e., \eqref{eq:relations hat H} is $\overline{\cdot}$-invariant. 
Clearly, applying $\overline {\cdot}$ to \eqref{eq:relations hat H} for $D_s$, we obtain \eqref{eq:relations hat H} for $D_{s^{-1}}=\overline D_s$, because $|s^{-1}|=|s|$, which verifies that \eqref{eq:relations hat H taft} is also $\overline{}$-invariant.  This proves (a).

Prove (b). First, show that $\theta$ is an endomorphism of $\hat {\bf H}_{\chi,\sigma}(W)$, i.e., that $\theta$ preserves the defining relations. 
Indeed, for $w\in W$ one has $\theta(w)=\tau_w\cdot w$, where $\tau_w\in R^\times$ such that $\tau_w\tau_{w^{-1}}=1$. Therefore, abbreviating $s'=wsw^{-1}$, we obtain $\theta(w)\theta(D_s)\theta(w^{-1})=w\theta(D_s)w^{-1}=w(\chi_{s,s}D_s+\sigma_{s,s})w^{-1}$
$$=\chi_{s,s}(\chi_{w,s}D_{s'}+\sigma_{w,s}(1-s'))+\sigma_{s,s}=\chi_{s',s'}(\chi_{w,s}D_{s'}+\sigma_{w,s}(1-s'))+\sigma_{s,s}$$
$$=\chi_{w,s}(\sigma_{s',s'}+\chi_{s',s'}D_{s'})+\sigma_{w,s}(1-\chi_{s',s'}s')=\theta(\chi_{w,s}D_{s'}+\sigma_{w,s}(1-s'))=\theta(wD_sw^{-1})$$
for $w\in W$, $s\in {\mathcal S}$ by Lemma \ref{le:basic properties of chi sigma}(b) and the assumption of part (b). Finally, let us verify that the  relations \eqref{eq:relations hat H taft primitive} are invariant under $\theta$. Indeed, applying $\theta$ to the defining functional relation \eqref{eq:fstDs} for $f_s$ defined by \eqref{eq:actual fs} and using \eqref{eq:fab symmetry}, we obtain (abbreviating $a_s=\chi_{s,s}$, $b_s=\sigma_{s,s}$):
$$\theta(f_s(ts+D_s)-f(t))=f_s(a_sts+a_sD_s+b_s)-f(t)=f_s(ts+D_s)-f(t)=0$$
This proves that $\theta$ is an $R$-linear endomorphism 
of $\hat {\bf H}_{\chi,\sigma}(W)$. It is easy to see that $\theta$ is invertible and the inverse is given by $\theta^{-1}(w)=\tau_w^{-1}w$ for $w\in W$ and $\theta^{-1}(D_s)=\sigma_{s^{-1},s^{-1}}+\chi_{s^{-1},s^{-1}}D_s$ for $s\in {\mathcal S}$.
This proves the first assertion of part (b). Prove the second assertion. 
Indeed, we obtain for all $w\in W$:
$\theta({w\cdot \tilde  {\bf D}_{\chi,\sigma}(W)\cdot w^{-1}})=\theta(w)\cdot \theta({\tilde {\bf D}_{\chi,\sigma}(W)})\cdot \theta(w^{-1})\subseteq w\cdot {\tilde {\bf D}_{\chi,\sigma}(W)}\cdot w^{-1}$
therefore, $\theta({\bf K}_{\chi,\sigma}(W))\subset {\bf K}_{\chi,\sigma}(W)$.
Part (b) is proved.

Prove (c)(i). We need the following fact.
\begin{lemma} 
\label{le:symmetris hat D(W) gen} For each $\chi,\sigma$ satisfying \eqref{eq:2 cocycle},  $\tilde V=\oplus_{s\in {\mathcal S}} R\cdot D_s$  is a $W$-module via $w(1)=1$ and 
$w(D_s)= \sigma_{w,s}+\chi_{w,s}D_{wsw^{-1}}$  
for $w\in W$, $s\in {\mathcal S}$. 

\end{lemma}

\begin{proof} Indeed, $w_1(w_2(D_s))= w_1(\sigma_{w_2,s}+\chi_{w_2,s}D_{w_2sw_2^{-1}})=\sigma_{w_2,s}+\chi_{w_2,s}w_1(D_{w_2sw_2^{-1}})$
$$=\sigma_{w_2,s}+\chi_{w_2,s}\chi_{w_1,w_2sw_2^{-1}}D_{w_1w_2sw_2^{-1}w_1^{-1}}+\chi_{w_2,s}\sigma_{w_1,w_2sw_2^{-1}}=\sigma_{w_1w_2,s}+\chi_{w_1w_2,s}D_{w_1w_2sw_2^{-1}w_1^{-1}}$$
for all $w\in W$, $s\in {\mathcal S}$, by \eqref{eq:2 cocycle}. Also $1(D_s)=D_s$ because $\chi_{1,s}=1$, $\sigma_{1,s}=0$.



The lemma is proved.
\end{proof}

That is, the $W$-action lifts to $T(V)$, where $V=\oplus_{s\in {\mathcal S}} R \cdot D_s$ by algebra automorphisms (because any $R$-linear map $V\to T(V)$ lifts to an endomorphism of the algebra $T(V)$). 
Thus, it remains to show that the defining relations \eqref{eq:relations hat H taft primitive} are preserved under the action. 
Indeed, since $a_s=\chi_{s,s}$ is a primitive $|s|$-th root of unity, i.e., $1+a_s+\cdots +a_s^{|s|-k}=-a^{-k}(1+a_s+\cdots a_s^{k-1})$ for $0\le k\le |s|$, the relation \eqref{eq:relations hat H taft primitive} with $s\in {\mathcal S}$ of finite order $|s|$ reads (in the notation \eqref{eq:actual fs}):
\begin{equation}
\label{eq:relations hat H taft primitive short}
f_s(D_s)=0\ .
\end{equation}

Then, applying $w$ to the left hand side of  the above relation, we obtain (in the notation \eqref{eq:actual fs}):
\begin{equation}
\label{eq:fwsw-11}
f_s(\sigma_{w,s}+\chi_{w,s}D_{s'})=f_{s'}(D_s)
\end{equation}
by Lemma \ref{le:fwsw-1}, where we abbreviated $s'=wsw^{-1}$. Finally, taking into account that $\chi_{s',s'}=\chi_{s,s}$ by Lemma \ref{le:basic properties of chi sigma}(b) and 
$\sigma_{s',s'}=\sigma_{s,s}$ by the assumption of part (b), we obtain $f_{s'}(D_s)=f_s(D_s)=0$. Part (c)(i) is proved.

Prove (c)(ii). We start with the following obvious general result. 

\begin{lemma} For any $R$-module $V$ and $R$-linear maps $f,s:V\to T(V)$  there is a unique $R$-linear map $d=d_{f,s}:T(V)\to T(V)$ such that 

$\bullet$ $d(1)=0$, $d(v)=f(v)$  for  $v\in V$.

$\bullet$ $d(xy)=d(x)y+s(x)d(y)$ for all $x,y\in T(V)$.

\end{lemma}

For $V=\oplus_{s\in {\mathcal S}} \ZZ \cdot D_s$, $s\in {\mathcal S}$ viewed as a $\ZZ$-linear map $V\to V\subset T(V)$ and 
$f:V\to \ZZ\subset T(V)$ given by $f(D_{s'})=\delta_{s,s'}$, we abbreviate $d_s:=d_{f,s}$. To prove the assertion, it suffices to show that $d_s$ preserves the defining relations \eqref{eq:relations hat H taft primitive short} of $\hat {\bf D}_{\chi,\sigma}(W)$, i.e., the relations of the form 
$f_s(D_s)=0$ for all $s\in {\mathcal S}$ of finite order $|s|$. Indeed, if $s\ne s'$, then, clearly,  $d_s(f_{s'}(D_{s'}))=0$. Suppose that $s=s'$. Then, in the notation \eqref{eq:actual fs} one has (similarly to the proof of Lemma \ref{le:fwsw-1}):
$d_s(f_s(D_s))=\sum\limits_{k=1}^{|s|} s\left(\prod\limits_{i=k+1}^{|s|}\left(D_s-a_s\frac{1-a_s^i}{1-a_s}\right)\right)\cdot d_s\left(D_s-b_s\frac{1-a_s^k}{1-a_s}\right)\cdot\prod\limits_{j=1}^{k-1}(D_s-b_s\frac{1-a_s^j}{1-a_s})$
$$=\sum_{k=1}^{|s|} \left(\prod_{i=k+1}^{|s|}\left(a_sD_s+b_s-b_s\frac{1-a_s^i}{1-a_s}\right)\right)\cdot a_s \cdot\prod_{j=1}^{k-1}(D_s-b_s\frac{1-a_s^j}{1-a_s})$$
$$=\sum_{k=1}^{|s|} \left(\prod_{i=k+1}^{|s|}\left(a_sD_s-b_s\frac{a_s-a_s^i}{1-a_s}\right)\right)\cdot a_s \cdot\prod_{j=1}^{k-1}(D_s-b_s\frac{1-a_s^j}{1-a_s})=\sum_{k=1}^{|s|} a_s^{|s|-k+1} \cdot\prod_{j=1}^{|s|-1}(D_s-b_s\frac{1-a_s^j}{1-a_s})=0$$
because $s(D_s)=a_sD_s+b_s$ and $a_s=\chi_{s,s}$ is a primitive $s$-th root of unity in $R^\times$.
This proves (c)(ii).

We prove the last assertion of (c) by showing that both sides of \eqref{eq:w-conjugate ds} are $wsw^{-1}$-derivations which agree on generators of $\hat {\bf D}_{\chi,\sigma}(W)$. Indeed, denote $d'_s=\chi_{w,s}w\circ d_s\circ w^{-1}$ and, first, substitute $x=D_{s'}$: 
$d'_s(D_{s'})=\chi_{w,s}w(d_s(\chi_{w^{-1},s'}D_{w^{-1}s'w}+\sigma_{w^{-1},s'}))=\chi_{w,s}w(\chi_{w^{-1},s'}\delta_{s,w^{-1}s'w})=\chi_{w,s}\chi_{w^{-1},wsw^{-1}}\delta_{wsw^{-1},s'}=d_{wsw^{-1}}(D_{s'})$ by \eqref{le:basic properties of chi sigma}.
Furthermore,  
$$d'_s(xy)=\chi_{w,s}w(d_s(w^{-1})(x)\cdot w^{-1}(y))=\chi_{w,s}w\left(d_s(w^{-1})(x))\cdot w^{-1}(y)+s(w^{-1}(x))\cdot d_s(w^{-1}(y))\right)$$
$$=\chi_{w,s}w(d_s(w^{-1})(x)))\cdot y+\chi_{w,s}\cdot (wsw^{-1})(x)\cdot w(d_s(w^{-1}(y)))=d'_s(x)\cdot y+(wsw^{-1})(x)\cdot d'_s(y)$$
for $x,y\in \hat {\bf D}_{\chi,\sigma}(W)$. This proves that $d'_s=d_s$. 

Proposition \ref{pr:symmetris hat D(W) gen} is proved.
\end{proof}

\noindent {\bf Proof of Theorem \ref{th:symmetris hat D(W) gen}}. 
Using Proposition \ref{pr:symmetris hat D(W) gen}(a), we obtain for all $w\in W$:
$$\overline {w\cdot \tilde  {\bf D}_{\chi,\sigma}(W)\cdot w^{-1}}
=\overline {w^{-1}}\cdot \overline {\tilde{\bf D}_{\chi,\sigma}(W)}\cdot \overline w \subseteq w\cdot {\tilde {\bf D}_{\chi,\sigma}(W)}\cdot w^{-1}\ ,$$
therefore, $\overline{{\bf K}_{\chi,\sigma}(W)}\subset {\bf K}_{\chi,\sigma}(W)$. 

Finally, we need the following fact.

\begin{lemma}
 $\varepsilon(\overline x)=\overline {\varepsilon(x)}$ for $x\in \hat {\bf H}_{\chi,\sigma}(W)$.
\end{lemma}

\begin{proof} Since both $\overline{\cdot}\circ \varepsilon$ and $\varepsilon\circ \overline{\cdot}$  are $R$-antilinear ring homomorphisms $\hat {\bf H}_{\chi,\sigma}(W)\to R$, it suffices to prove the assertion only on generators of $\hat {\bf H}_{\chi,\sigma}(W)$.
Indeed, $\varepsilon(\overline D_s)=\overline{\varepsilon(D_s)}=0$ for all $s\in {\mathcal S}$ and 
$\varepsilon(\overline w)=\varepsilon(w^{-1})=1=\overline{\varepsilon(w)}$ for $w\in W$. 

The lemma is proved.
\end{proof}

Therefore,  the ideal ${\bf J}_{\chi,\sigma}(W)$ generated by ${\bf K}_{\chi,\sigma}(W)\cap Ker~\varepsilon$ is $\overline{\cdot}\,$-invariant  and $\overline{\cdot}$ is well-defined on 
the quotient ${\bf H}_{\chi,\sigma}(W)=\hat {\bf H}_{\chi,\sigma}(W)/{\bf K}_{\chi,\sigma}(W)$.

This proves Theorem \ref{th:symmetris hat D(W) gen}.\endproof


The following result correlates the automorphism $\theta$ with relations in ${\bf H}_{\chi,\sigma}(W)$.

\begin{proposition} 
\label{pr:antipode theta}
In the assumptions of Proposition \ref{pr:symmetris hat D(W) gen}(b) suppose that $\varepsilon(\theta(x))=\varepsilon(x)$ for all $x\in {\bf K}_{\chi,\sigma}(W)$. Then

(a) $\theta(x)=S^{-2}(x)$ for all $x\in {\bf K}_{\chi,\sigma}(W)$.

(b) Suppose that $\hat {\bf H}_{\chi,\sigma}(W)$ is a free $R$-module. Then ${\bf H}_{\chi,\sigma}(W)$ admits an $R$-linear automorphism $\theta$ such that the structural homomorphism  
$\hat {\bf H}_{\chi,\sigma}\twoheadrightarrow {\bf H}_{\chi,\sigma}$ is $\theta$-equivariant.
\end{proposition}

\begin{proof} Prove (a). We need the following result.
%
%
%

\begin{lemma} In the assumptions of Proposition \ref{pr:symmetris hat D(W) gen}(b) one has
\begin{equation}
\label{eq:delta circ theta}
\Delta\circ \theta=(S^{-2}\otimes \theta)\circ \Delta=(\theta\otimes 1)\circ \Delta \ .
\end{equation}

\end{lemma}

\begin{proof} Since both $\Delta$ and $\theta$ are algebra homomorphisms, hence so are $\Delta\circ \theta$, $(S^{-2}\otimes \theta)\circ \Delta$, and $(\theta\otimes 1)\circ \Delta$, it suffices to prove \eqref{eq:delta circ theta} only on 
generators of $\hat {\bf H}_{\chi,\sigma}(W)$.  

Indeed, for $w\in W$ one has $\theta(w)=\tau_w\cdot w$ for some $\tau_w\in R^\times$, therefore
$$\Delta(\theta(w))=\Delta(\tau_w\cdot w)=\tau_w\cdot w\otimes w=\theta(w)\otimes w=S^{-2}(w)\otimes \theta(w)\ .$$
Furthermore,  we obtain for $s\in {\mathcal S}$ (abbreviating $a_s=\chi_{s,s}$, $b_s=\sigma_{s,s}$):
$$\Delta(\theta(D_s))=\Delta(b_s+a_sD_s)=b_s\cdot 1\otimes 1+a_sD_s\otimes 1+a_s\cdot s\otimes D_s=\theta(D_s)\otimes 1+\theta(s)\otimes D_s$$
$$=S^{-2}(D_s)\otimes 1+S^{-2}(s)\otimes \theta(D_s)$$
because $\theta(D_s)=b_s+a_sD_s$, $\theta(s)=a_s\cdot s$ and $S^2(D_s)=S(-s^{-1}D_s)=s^{-1}\cdot D_s\cdot s$, therefore, 
$S^{-2}(D_s)=s\cdot D_s\cdot s^{-1}=a_sD_s+b_s(1-s)$.
This proves \eqref{eq:delta circ theta}.

The lemma is proved.\end{proof}

Finally, applying $1\otimes \varepsilon$ to \eqref{eq:delta circ theta}, we
obtain
$\theta(x)=S^{-2}(x_{(1)})\cdot \varepsilon(\theta(x_{(2)})$
for all $x\in {\bf K}_{\chi,\sigma}(W)$. Since ${\bf K}_{\chi,\sigma}(W)$ is a left coideal by the argument from the proof of Theorem \ref{th:Coideal Hopf gen}, then $x_{(2)}\in {\bf K}_{\chi,\sigma}(W)$, i.e., $\varepsilon(\theta(x_{(2)})=\varepsilon(x_{(2)})$ and  
$\theta(x)=S^{-2}(x_{(1)})\cdot \varepsilon(x_{(2)})=S^{-2}(x_{(1)}\varepsilon(x_{(2)}))=S^{-2}(x)$.
This proves (a).

Prove (b). The assumption of the proposition and  the second assertion of Proposition \ref{pr:symmetris hat D(W) gen}(b) imply that ${\bf K}_{\chi,\sigma}(W)^+={\bf K}_{\chi,\sigma}(W)\cap Ker~\varepsilon$ is $\theta$-invariant.  Therefore, the ideal ${\bf J}_{\chi,\sigma}(W)$ generated by 
${\bf K}_{\chi,\sigma}(W)^+$ is also $\theta$-invariant and 
$\theta$ is well-defined on 
the quotient ${\bf H}_{\chi,\sigma}(W)=\hat {\bf H}_{\chi,\sigma}(W)/{\bf K}_{\chi,\sigma}(W)$. This proves (b).

The proposition is proved.
\end{proof}

\noindent {\bf Proof of Theorem \ref{th:symmetris D(W)}}. Prove (a). Indeed, $\chi$ and $\sigma$ defined by \eqref{eq:chi sigma defined} satisfy the assumptions of Proposition \ref{pr:symmetris hat D(W) gen}(a) with the identity $\overline{\cdot}$ on $\ZZ$, therefore $\overline{\cdot}$ is a well-defined involutive anti-automorphism of 
$\hat {\bf H}(W)=\hat {\bf H}_{\chi,\sigma}(W)$ and it satisfies $\overline D_s=D_s$ for $s\in {\mathcal S}$. Copying the argument from the proof of Theorem \ref{th:symmetris hat D(W) gen}, we see that ${\bf K}(W)^+={\bf K}(W)\cap Ker~\varepsilon$ is $\overline{\cdot}$-invariant. 
Since all filtered components $\hat {\bf D}(W)_{\le d}$ are also $\overline{\cdot}$-invariant, replacing $W$ with $W_{\{i,j\}}$ and taking $d=m_{ij}$, we see that all ${\bf K}_{ij}(W)$ are $\overline{\cdot}$-invariant. 
Therefore, the (Hopf) ideals ${\bf J}(w)$ and $\underline {\bf J}(w)$ generated respectively by ${\bf K}(W)^+$ and $\underline {\bf K}=\sum\limits_{j\ne i} {\bf K}_{ij}(W)$ are also $\overline{\cdot}$-invariant. This proves Theorem \ref{th:symmetris D(W)}(a).

Prove Theorem \ref{th:symmetris D(W)}(b) now.  
We need the following result.

\begin{proposition} 
\label{pr:symmetris hat D(W)} For any Coxeter group $W$ one has:

\noindent (a)  For $s\in {\mathcal S}$,   $\hat {\bf D}(W)$ admits an $s$-derivation $d_s$ such that $d_s(D_{s'})=\delta_{s,s'}$, $s,s'\in {\mathcal S}$.

\noindent (b) $d_s({\bf K}(W))=\{0\}$ for all $s\in {\mathcal S}$  and  $w x w^{-1}=w(x)$ for all , $w\in W$. 
 
\end{proposition}

\begin{proof} 
%
Part (a) directly follows from Proposition \ref{pr:symmetris hat D(W) gen}(c)(ii).

Prove (b).
Since $s_i{\bf K}(W)s_i={\bf K}(W)$, Theorem \ref{th:ZWcross and free D}  and Lemma \ref{le:conjugation vs action}  imply that for each $x\in {\bf K}(W)$ one has $d_{s_i}(x)=0$ and $s_ixs_i=s_i(x)$. This, in particular,  proves the first assertion of part (b) for $s=s_i$ and the second assertion -- for $w=s_i$. Let us prove the second assertion for any $w$. Indeed, if $\ell(w)\le 1$, we have nothing to prove. Suppose that $\ell(w)\ge 2$, i.e., $w=s_iw'$ so that $\ell(w')=\ell(w)-1$. Then, using inductive hypothesis in the form $w'xw'^{-1}=w'(x)\in {\bf K}(W)$ for all $x\in {\bf K}(W)$, we obtain
$$wxw^{-1}=s_i\cdot (w'xw'^{-1})\cdot s_i=s_i\cdot (w'xw'^{-1})\cdot s_i=s_i\cdot (w'(x))\cdot s_i=s_i(w'(x))=(s_iw')(x)=w(x)$$
for all $x\in {\bf K}(W)$, which proves the second assertion. Prove the first assertion now. Let $s\in {\mathcal S}$, choose $w\in W$ and $i\in I$ such that $s=ws_iw^{-1}$. The last assertion of Proposition \ref{pr:symmetris hat D(W) gen} guarantees that $d_s=\chi_{w,s}\cdot w\circ d_{s_i}\circ w^{-1}$. Then 
$$d_s({\bf K}(W))=\chi_{w,s}\cdot (w\circ d_{s_i}\circ w^{-1})({\bf K}(W))=\chi_{w,s}\cdot (w)(d_{s_i}({\bf K}(W)))=\{0\}\ .$$  
This finishes the proof of (b).
The proposition is proved.
\end{proof} 

Therefore, the (Hopf) ideal  ${\bf I}(w)$ of $\hat {\bf D}(W)$ generated ${\bf K}(W)$  is invariant both under the $W$-action and under all $d_s$. This proves that the quotient  $\underline {\bf D}(W)=\hat {\bf D}(W)/{\bf I}(w)$ has a natural $W$-action and $s$-derivations $d_s$. Similarly,  let $\underline {\bf K}(W)\subset \hat {\bf D}(W)$ be as in the proof of Theorem \ref{th:HW factorization}. By definition, $\underline {\bf K}(W)\subset {\bf K}(W)$ is $W$-invariant and is annihilated by all $d_s$, Therefore, 
the ideal $\underline {\bf I}(W)$  generated by $\underline {\bf K}(W)$ is also invariant both under the $W$-action and under all $d_s$ hence  the quotient  ${\bf D}(W)=\hat {\bf D}(W)/\underline {\bf I}(W)$ has a natural $W$-action and $s$-derivations $d_s$.
This proves Theorem \ref{th:symmetris D(W)}(b).

Prove Theorem \ref{th:symmetris D(W)}(c) now. We need the following result.

\begin{proposition} 
\label{pr:eps theta}
Suppose that $W=\langle s_i\,|\,i\in I\rangle$ is a finite Coxeter group. 
Then $\varepsilon(\theta(x))=\varepsilon(x)$ for all $x\in {\bf K}(W)$.

\end{proposition}

\begin{proof} 
We need the following result.

\begin{lemma} 
\label{le:theta finite W}
Suppose that $W=\langle s_i\,|\,i\in I\rangle$ is a finite Coxeter group. Then in the notation of Lemma \ref{le:conjugation vs action}(a), one has
\begin{equation}
\label{eq:theta via w0}
\theta(x)=\hat \tau(w_0(x))
\end{equation}
for $x\in \hat {\bf D}(W)$, where $w_0$ is the longest element of $W$ and $\tau$ is an automorphism of $\hat {\bf H}(W)$ determined by $\hat \tau(s_i)=s_{\tau(i)}$, 
$\hat \tau(D_i)=D_{\tau(i)}$, where $\sigma$ is a certain permutation of $I$.

\end{lemma} 

\begin{proof} Lemma  \ref{le:conjugation vs action}(a) taken with $s=s_i$ immediately implies that 
\begin{equation}
\label{eq:w(Di)}
w(D_i)= \begin{cases}
D_{ws_iw^{-1}} & \text{if $\ell(ws_i)=\ell(w)+1$}\\
1-D_{ws_iw^{-1}} & \text{if $\ell(ws_i)=\ell(w)-1$}
\end{cases}
\end{equation}
for all $i\in I$, $w\in W$.

Furthermore, clearly, $w_0 s_iw_0^{-1}=s_{\tau(i)}$ for all $i\in I$ and some permutation $\tau$ of $I$ (which satisfies $m_{\tau(i),\tau(j)}=m_{ij}$ for all $i,j\in I$). It is also clear that the assignments $s_i\mapsto s_{\tau(i)}$, $D_i\mapsto D_{\tau(i)}$ define an automorphism $\hat \tau$ of $\hat {\bf H}(W)$. 
Since $\ell(w_0s_i)=\ell(w_0)-1$, we have by \eqref{eq:w(Di)}:
$$w_0(D_i)=1-D_{w_0s_iw_0^{-1}}=1-D_{\tau(i)}\ .$$ 

Since the  $\theta$ and another automorphism of $\hat {\bf D}(W)$ given by $x\mapsto \hat \tau (w_0(x))$ agree on generators $D_i$, $i\in I$, this proves \eqref{eq:theta via w0}.
The lemma is proved.
\end{proof}

Furthermore, Lemma \ref{le:theta finite W} and the immediate fact that $\varepsilon\circ \hat \tau=\varepsilon$ imply 
$$\varepsilon(\theta(x))=\varepsilon(\hat\tau(w_0(x)))=\varepsilon(w_0(x))$$
for all $x\in \hat {\bf D}(W)$. 
If $x\in {\bf K}(W)$, then $w_0(x)=w_0\cdot x\cdot w_0^{-1}$ by Proposition \ref{pr:symmetris hat D(W)}(b) 
and $\varepsilon(w_0(x))=\varepsilon(x)$.

The proposition is proved.
\end{proof}

Finally, we need the following result.

\begin{lemma} $\theta({\bf K}_{ij}(W))={\bf K}_{ij}(W)$ for any Coxeter group $W$ and any distinct $i,j\in I$.

\end{lemma} 

\begin{proof}
By definition, for any subset $J$ of $I$, $\theta$ preserves the subalgebra $\hat {\bf H}(W_J)\subset \hat {\bf H}(W)$, e.g., $\theta(\hat {\bf D}(W_J))=\hat {\bf D}(W_J)$. Also, by Proposition \ref{pr:symmetris hat D(W) gen}(b), $\theta({\bf K}(W_J))={\bf K}(W_J)$.
Since $\theta$ preserves each filtered component $\hat {\bf D}(W)_{\le d}$,  $\theta$ also preserves each filtered component ${\bf K}(W_J)_{\le d}\subset  {\bf K}(W)_{\le d}$. 
Note  that if $m_{ij}\ge 2$, then the subgroup $W_{\{i,j\}}$ of $W$ is finite. These arguments and Proposition \ref{pr:eps theta} guarantee that
$\varepsilon(\theta(x))=\varepsilon(x)$ 
for $x\in {\bf K}(W_{\{i,j\}})_{\le d}$ whenever $m_{ij}\ge 2$, $d\ge 0$. Taking $d=m_{ij}$ (and taking into account that ${\bf K}_{ij}(W)=\{0\}$ whenever $m_{ij}=0$), we finish the proof.

The lemma is proved.
\end{proof}

Therefore, the Hopf ideal $\underline {\bf J}(w)$ of $\hat {\bf H}(W)$ generated $\underline {\bf K}=\sum\limits_{j\ne i} {\bf K}_{ij}(W)$ is $\theta$-invariant. Hence $\theta$ is a well-defined automorphism of the quotient  ${\bf H}(W)=\hat {\bf H}(W)/\underline {\bf J}(w)$.
 
This proves Theorem \ref{th:symmetris D(W)}(c). \endproof

%
 %
%
%
 %
%
 %
%

\medskip

\noindent {\bf Proof of Theorem \ref{th:symmetris Dchisigma(W)}}. We need the following result.

\begin{proposition} 
\label{pr:bad derivatives}
In the assumptions of Proposition \ref{pr:symmetris hat D(W) gen}(c) and notation of Proposition \ref{pr:factored derivatives}:

\noindent (a) the condition 
\eqref{eq:acyclicity S} implies: 

(i) $\sigma_{w,s_1}\sigma_{ws_1,s_2}\cdots \sigma_{ws_1\cdots s_{k-1},s_k}\partial_{ws_1\cdots s_k,w}=0$
for any $w\in W$, $s_1,\ldots,s_k\in {\mathcal S}$, $k\ge 1$.

(ii) $w(x)=\partial_{w,w}(x)$ for all $w\in W$, 
$x\in \hat {\bf D}_{\chi,\sigma}(W)$.
 
\noindent (b) The condition \eqref{eq:alternating acyclicity S} for a given $s\in {\mathcal S}$ implies that

(i)  $\sigma_{s^{-1},s_1}\sigma_{s^{-1}s_1,s_2}\cdots \sigma_{s^{-1}s_1\cdots s_{k-1},s_k}\partial_{s^{-1}s_1\cdots s_k,1}=\delta_{k,1}\delta_{s,s_1}\sigma_{s^{-1},s}$ 
for all $s_1,\ldots,s_k\in {\mathcal S}$, $k\ge 1$.

(ii)  $\partial_{s^{-1},1}=-\sigma_{s^{-1},s}d_{s^{-1}}$ in the notation of Proposition \ref{pr:symmetris hat D(W) gen}(c)(i).

\end{proposition}

\begin{proof} 
Prove (a). Denote 
$\partial^w_{s_1,\ldots,s_k}=\sigma_{w,s_1}\sigma_{ws_1,s_2}\cdots \sigma_{ws_1\cdots s_{k-1},s_k}\partial_{ws_1\cdots s_k,w}$ for all $w\in W$, $s_1,\ldots,s_k\in {\mathcal S}$, $k\ge 0$ (with $k=0$ this is just $\partial_{w,w}$). 


\begin{lemma}
\label{le:partial s w recursion}
$\partial^w_{s_1,\ldots,s_k}(D_{s_{k+1}}x)=ws_1\cdots s_k(D_s)\partial^w_{s_1,\ldots,s_k}(x)
-\partial^w_{s_1,\ldots,s_{k+1}}(x)$
for each $w\in W$, $s_1,\ldots,s_{k+1}\in {\mathcal S}$, $k\ge 0$, $x\in \hat {\bf D}_{\chi,\sigma}(W)$.

\end{lemma}

\begin{proof} Clearly, $\partial_{w,w''}(D_s)=\delta_{w,w''}w(D_s)-\delta_{w'',ws}\sigma_{w,s}$
for all $w,w''\in W$, $s\in {\mathcal S}$. This and Proposition \ref{pr:factored derivatives} imply 
\begin{equation}
\label{eq:partial leibniz rule}
\partial_{w,w'}(D_sx)=w(D_s)\partial_{w,w'}(x)-\sigma_{w,s}\partial_{ws,w'}(x)
\end{equation}
for all $w,w'\in W$, $s\in {\mathcal S}$, $x\in \hat {\bf D}_{\chi,\sigma}(W)$.
Furthermore, \eqref{eq:partial leibniz rule} implies that 
$$\partial_{ws_1\cdots s_k,w}(D_{s_{k+1}}x)=ws_1\cdots s_k(D_s)\partial_{w,ws_1\cdots s_k}(x)-\sigma_{ws_1\cdots s_k,s_{k+1}}\partial_{ws_1\cdots s_{k+1},w}(x)\ .$$
The lemma is proved.
\end{proof}

Furthermore, we will show that $\partial^w_{s_1,\ldots,s_k}(y)=0$   for   all $s_1,\ldots,s_k\in {\mathcal S}$, $k\ge 1$  
and $\partial_{w,w}(y)=w(y)$ (i.e., when $k=0$)
by induction in the filtered degree of $y\in \hat {\bf D}_{\chi,\sigma}(W)$ (the algebra is naturally filtered by $\ZZ_{\ge 0}$ via 
$\deg D_s=1$ for $s\in {\mathcal S}$).
Indeed, if $y\in \hat {\bf D}_{\chi,\sigma}(W)_{\le 0}=R$, then $\partial_{w,w'}(y)=\delta_{w,w'}y$, therefore, 
$\partial^w_{s_1,\ldots,s_k}(y)=
\sigma_{w,s_1}\sigma_{ws_1,s_2}\cdots \sigma_{ws_1\cdots s_{k-1},s_k}\delta_{s_1\cdots s_k,1}y=0$
by the condition \eqref{eq:acyclicity S} and $\partial_{w,w}(y)=w(y)=y$.
If $y\in \hat {\bf D}_{\chi,\sigma}(W)_{\le r}$, $r>0$, then $y$ is an $R$-linear combination of the elements of the form $D_sx$, where $x\in \hat {\bf D}_{\chi,\sigma}(W)_{\le r-1}$. By $R$-linearity it suffices to prove the assertion only for $y=D_sx$. Then, using the inductive hypothesis in the form: $\partial^w_{s_1,\ldots,s_k}(x)=0$, $\partial^w_{s_1,\ldots,s_k,s}(x)=0$, $\partial_{w,w}(x)=w(x)$
Lemma \ref{le:partial s w recursion} guarantees that $\partial^w_{s_1,\ldots,s_k}(y)=0$ and same lemma taken with $k=0$ implies that
$\partial_{w,w}(D_sx)=w(D_s)\partial_{w,w}(x)-\partial^w_s(x)=w(D_s)\partial_{w,w}(x)=w(D_s)w(x)=w(y)$.

This proves (a).

Prove (b) now. Denote 
$\tilde \partial^s_{s_1,\ldots,s_k}=\sigma_{s^{-1},s_1}\sigma_{s^{-1}s_1,s_2}\cdots \sigma_{s^{-1}s_1\cdots s_{k-1},s_k}\partial_{s^{-1}s_1\cdots s_k,1}$ for  $s,s_1,\ldots,s_k\in {\mathcal S}$, $k\ge 0$ (if $k=0$, this is just $\partial_{s,1}$). 


\begin{lemma}
\label{le:partial s w recursion 2}
$\tilde \partial^s_{s_1,\ldots,s_k}(D_{s_{k+1}}x)=s^{-1}s_1\cdots s_k(D_s)\tilde \partial^s_{s_1,\ldots,s_k}(x)
-\tilde \partial^s_{s_1,\ldots,s_{k+1}}(x)$
for  $s,s_1,\ldots,s_{k+1}\in {\mathcal S}$, $x\in \hat {\bf D}_{\chi,\sigma}(W)$.

\end{lemma}

\begin{proof}  Let $u:=s^{-1}s_1\cdots s_k$. Then 
$\partial_{u,1}(D_{s_{k+1}}x)=u(D_s)\partial_{u,1}(x)-\sigma_{u,s_{k+1}}\partial_{u,1}(x)$ by \eqref{eq:partial leibniz rule}.

The lemma is proved.
\end{proof}

Furthermore, similarly to the proof of part (a), we will show that $\tilde \partial^s_{s_1,\ldots,s_k}(y)=0$ for $k\ge 1$ and all $s_1,\ldots,s_k\in {\mathcal S}$ and $\partial_{s^{-1},1}(y)=-\sigma_{s^{-1},s}d_{s^{-1}}(y)$ (i.e., when $k=0$)
by induction in the filtered degree of $y\in \hat {\bf D}_{\chi,\sigma}(W)$.
Indeed, if $y\in \hat {\bf D}_{\chi,\sigma}(W)_{\le 0}=R$, then $\partial_{w,w'}(y)=\delta_{w,w'}y$, therefore, 
$$\tilde \partial^s_{s_1,\ldots,s_k}(y)=
\sigma_{s^{-1},s_1}\sigma_{s^{-1}s_1,s_2}\cdots \sigma_{s^{-1}s_1\cdots s_{k-1},s_k}\delta_{s_1\cdots s_k,s}y=0$$ 
by the condition \eqref{eq:alternating acyclicity S} and $\partial_{s^{-1},1}(y)=d_s(y)=0$.

If $y\in \hat {\bf D}_{\chi,\sigma}(W)_{\le r}$, $r>0$, then $y$ is an $R$-linear combination of the elements of the form $D_{s'}x$, where $x\in \hat {\bf D}_{\chi,\sigma}(W)_{\le r-1}$. By $R$-linearity it suffices to prove the assertion only for $y=D_{s'}x$. Then, using the inductive hypothesis in the form: $\partial^w_{s_1,\ldots,s_k}(x)=\delta_{k,1}\delta_{s,s_1}\sigma_{s^{-1},s}\cdot x$, $\partial^w_{s_1,\ldots,s_k,s'}(x)=0$, $\partial_{s^{-1},1}(x)
=\sigma_{s^{-1},s}d_{s^{-1}}(x)$
Lemma \ref{le:partial s w recursion 2} guarantees that 
$$\tilde \partial^s_{s_1,\ldots,s_k}(y)=\delta_{k,1}\delta_{s,s_1}\sigma_{s^{-1},s}s^{-1}s_1\cdots s_k(D_{s'})x=\delta_{k,1}\delta_{s,s_1}\sigma_{s^{-1},s}D_{s'}x=\delta_{k,1}\delta_{s,s_1}\sigma_{s^{-1},s}\cdot y$$ 
for all $k\ge 1$ and same lemma taken with $k=0$ implies that
$$\partial_{s^{-1},1}(y)=s^{-1}(D_{s'})\partial_{s^{-1},1}(x)-\tilde \partial^s_{s'}(x)=-s^{-1}(D_{s'})\sigma_{s^{-1},s}d_{s^{-1}}(x)-\delta_{s,s'}\sigma_{s^{-1},s}x=-\sigma_{s^{-1},s}d_{s^{-1}}(y)\ .$$
This proves (b).
The proposition is proved. \end{proof}

Finally, Lemma \ref{le:factored derivatives} and Proposition \ref{pr:bad derivatives}(a) imply that 
$w xw^{-1}=w(x),~\partial_{s^{-1},1}(x)=0$
for all $x\in {\bf K}_{\chi,\sigma}(W)$, $w\in W$, $s\in {\mathcal S}$. 
Therefore, the ideal ${\bf I}$ of $\hat {\bf D}_{\chi,\sigma}(W)$ generated by 
${\bf K}_{\chi,\sigma}(W)\cap Ker~\varepsilon$
is invariant under both $W$-action and the $s^{-1}$-derivation $\partial_s:=-\partial_{s^{-1},1}=\sigma_{s^{-1},s}d_s$. Hence ${\bf D}_{\chi,\sigma}(W)=\hat {\bf D}_{\chi,\sigma}(W)/{\bf I}$ is also invariant under these symmetries. 

Theorem \ref{th:symmetris Dchisigma(W)} is proved. \endproof

\subsection{Simply-laced Hecke-Hopf algebras and proof of Theorems \ref{th:hopf hat intro}, \ref{th:hecke hopf intro}, \ref{th:hopf hat intro W} and Propositions \ref{pr:HS_n factorization}, \ref{pr:D presentation Sn}, \ref{pr:presentation simply-laced D}}
\label{subsec:Simply-laced} 

We need the following result. 

\begin{proposition} 
\label{pr:K for S3}
For any Coxeter group $W$ and $i,j\in I$ one has:

(a) If $m_{ij}=2$, then ${\bf K}_{ij}(W)=\ZZ\cdot K_{ij}$, where $K_{ij}=D_iD_j-D_jD_i$.

(b) If $m_{ij}=3$, then 
$${\bf K}_{ij}(W)=\ZZ\cdot K_{ij}+\ZZ\cdot K_{ji}+\ZZ\cdot(K_{ij}D_i-D_iK_{ji})+\ZZ\cdot (K_{ji}D_i-D_iK_{ji})+\ZZ\cdot (D_jK_{ji}-K_{ij}D_j)$$

\end{proposition}

\begin{proof} Indeed, by Proposition \ref{pr:symmetris hat D(W)}(b),  
\begin{equation}
\label{eq:Kij as kernel}
\ZZ+{\bf K}_{ij}(W)\subseteq {\bf K}'_{ij}(W):=\{x\in \hat {\bf D}(W_{\{i,j\}})_{\le m_{ij}}\}\,|\,d_s(x)=0, s\in {\mathcal S} \cap W_{\{i,j\}}\} \ .
\end{equation}

Prove (a) now. Clearly, if $m_{ij}=2$, then each $x\in {\bf K}_{ij}(W)$ is of the form $x=a+bD_iD_j+cD_jD_i$ for some $a,b,c\in \ZZ$. Since $s_i(D_j)=D_j$, then, clearly, $d_i(x)=bD_j+cD_j$, $d_j(x)=bD_i+cD_i$, $\varepsilon(x)=a$. Thus, $d_i(x)=d_j(x)=\varepsilon(x)=0$ iff $a=0$, $b+c=0$.
This proves (a). 

Prove (b) now. Fix $i,j\in I$ with $m_{ij}=3$. Then, according to Theorem \ref{th:ZWcross and free D}, $\hat {\bf D}(W_{\{i,j\}})$ is an algebra generated by $D_1:=D_i$, $D_2:=D_{ij}$, $D_3:=D_j$ subject to relations $D_k^2=D_k$ for $k=1,2,3$. Denote also $d_1:=d_i$, $d_2=s_id_js_i=s_jd_is_j$, $d_3=d_j$ so that 
$d_k(xy)=d_k(x)y+s_k(x)d_k(y)$
for all $x,y\in \hat {\bf D}(W_{\{i,j\}})$, where $s_1:=s_i$, $s_2:=s_is_js_i=s_js_is_j$, $s_3:=s_j$.

In particular, $K_{ij}=D_1D_3-D_2D_1-D_3D_2+D_2$, $K_{ji}=D_3D_1-D_1D_2-D_2D_3+D_2$.



\begin{lemma} 
\label{le:drop from 3 to 2} In the assumptions of Proposition \ref{pr:K for S3}(b), one has:

(a)  ${\bf K}'_{ij}(W)={\bf K}'_{ij}(W)_{\le 2}+\ZZ\cdot(K_{ij}D_1-D_1K_{ji})+\ZZ\cdot (K_{ji}D_1-D_1K_{ji})+\ZZ\cdot (D_3K_{ji}-K_{ij}D_3)$.

(b) ${\bf K}'_{ij}(W)_{\le 2}=\ZZ+\ZZ\cdot K_{ij}+\ZZ\cdot K_{ji}$.

\end{lemma}

\begin{proof} Since $\hat {\bf D}(W_{\{ij\}})$ is the free product of three copies of $\hat {\bf D}(W_{\{i\}})$, it is a free $\ZZ$-module (this also follows from by Lemmas \ref{le:hats are equal} and \ref{le:free hat H}). In particular, 
$\hat {\bf D}(W_{\{i,j\}})_{\le 3}$  is a free $\ZZ$-module with a basis $1$, $D_1,D_2,D_3$, $D_aD_b$,
 $D_aD_bD_{6-a-b}$, $D_aD_bD_a$ for all distinct $a,b\in \{1,2,3\}$, that is, each $x\in \hat{\bf D}(W_{\{i,j\}})_{\le 3}$ can be uniquely  written as  
\begin{equation}
\label{eq:x in D<=3}
x=a_0+\sum_{\ell=1}^3a_\ell D_\ell+\sum_{a,b\in \{1,2,3\},a\ne b} (f_{a,b}D_aD_b+g_{a,b}D_aD_bD_{6-a-b}+h_{a,b}D_aD_bD_a)\ ,
\end{equation} 
where all $a_k$,  $f_{a,b}$, $g_{a,b}$, $h_{a,b}$ are integers.

Let us show first that $d_k(x)=0$ for some $x$  in \eqref{eq:x in D<=3} and some $k\in \{1,2,3\}$ implies that
then $h_{k,b}=0$ for $b\in \{1,2,3\}\setminus \{k\}$.
Indeed, $d_k(D_aD_b)=\delta_{k,a}D_b+\delta_{k,b}s_k(D_a)$, 
$d_k(D_aD_bD_c)=d_k(D_aD_b)D_c+\delta_{k,c}s_k(D_aD_b)=\delta_{k,a}D_bD_c+\delta_{k,b}s_k(D_a)D_c+\delta_{k,c}s_k(D_aD_b)$.
Therefore, $d_k(x)=a_k+\sum\limits_{a\ne b} (f_{a,b}(\delta_{k,a}D_b+\delta_{k,b}s_k(D_a))+\sum\limits_{a\ne b} g_{a,b}(\delta_{k,a}D_bD_{6-k-b}+\delta_{k,b}s_k(D_a)D_{6-a-k}+\delta_{k,6-a-b}s_k(D_aD_b))$
$$+\sum\limits_{a\ne b}h_{a,b}(\delta_{k,a}D_bD_k+s_k(D_kD_b)+\delta_{k,b}s_k(D_a)D_a)$$

Taking into account that 
\begin{equation}
\label{eq:skDa}
s_k(D_a)=\begin{cases}  
1-D_k & \text{if $a=k$}\\
D_{6-a-k} & \text{if $k\in \{1,3\}$, $a\ne k$}\\
1-D_{6-a-k} & \text{if $k=2$, $a\ne k$}\\
\end{cases}
\end{equation}
we see that  
$d_k(x)=\sum\limits_{b\ne k} h_{k,b}(D_bD_k+D_kD_{6-k-b}) +\cdots$,
where $\cdots$ stand for the linear combination of monomials not containing $D_k$. Thus, $d_k=0$ implies $h_{k,b}=0$ for 
$b\in \{1,2,3\}\setminus \{k\}$. In particular,  $d_k(x)=0$ for $k=1,2,3$ implies that $h_{a,b}=0$ for all distinct $a,b\in \{1,2,3\}$. 

Based on the above computations, using \eqref{eq:skDa} again we obtain for $k\in \{1,2,3\}$: $d_k(x)=a_k+\sum\limits_{a\ne b} (f_{a,b}(\delta_{k,a}D_b+\delta_{k,b}s_k(D_a))+\sum\limits_{a\ne b} g_{a,b}(\delta_{k,a}D_bD_{6-k-b}+\delta_{k,b}s_k(D_a)D_{6-a-k}+\delta_{k,6-a-b}s_k(D_aD_b))$
$$=z+\sum_{b\ne k} g_{k,b}D_bD_{6-k-b}+\sum_{a\ne k} g_{a,k} s_k(D_a)D_{6-a-k}+\sum_{a,b:6-a-b=k} g_{a,b}s_k(D_a)s_k(D_b)$$
$$=z'+\sum_{b\ne k} g_{k,b}D_bD_{6-k-b}+\sum_{a\ne k} g_{a,k} (1-\delta_{k,2})D_{6-a-k}D_{6-a-k}+\sum_{a,b:6-a-b=k} g_{a,b}D_{6-a-k}D_{6-b-k}$$
$$=z''+\sum_{b\ne k} g_{k,b}D_bD_{6-k-b}+\sum_{a,b:6-a-b=k} g_{a,b}D_bD_a$$
for some $z,z',z''\in \hat {\bf D}(W_{\{i,j\}})_{\le 1}$. Thus, fixing $k',k''$ such that $\{k,k',k''\}=\{1,2,3\}$, we obtain:
$$d_k(x)=g_{k,k'}D_{k'}D_{k''}+g_{k,k''}D_{k''}D_{k'}+g_{k',k''}D_{k''}D_{k'}+g_{k'',k'}D_{k'}D_{k''}+z'' \ .$$  
Since $z''\in \hat {\bf D}(W_{\{i,j\}})_{\le 1}$, 
the equations $d_k(x)=0$ for $k=1,2,3$ imply that $g_{k,k'}+g_{k'',k'}=0$ for each permutation $(k,k',k'')$ of $\{1,2,3\}$. 
 
Thus, \eqref{eq:x in D<=3} with $d_k(x)=0$ for $k=1,2,3$ becomes
\begin{equation}
\label{eq:kernel reduced}
x=y+ g_{1,2}(D_1D_2D_3-D_3D_2D_1)+g_{1,3}(D_1D_3D_2-D_2D_3D_1)+g_{2,1}(D_2D_1D_3-D_3D_1D_2)
\end{equation}
for some $y\in \hat {\bf D}(W_{\{i,j\}})_{\le 2}$. 
Note that:
$D_1K_{ji}=D_1D_3D_1-D_1D_2D_3, K_{ij}D_1=D_1D_3D_1-D_3D_2D_1$, 
$$D_1K_{ij}=D_1D_3-D_1D_2D_1-D_1D_3D_2+D_1D_2,~D_3K_{ji}=D_3D_1-D_3D_1D_2-D_3D_2D_3+D_3D_2\ ,$$
$$K_{ji}D_1=D_3D_1-D_1D_2D_1-D_2D_3D_1+D_2D_1,~K_{ij}D_3=D_1D_3-D_2D_1D_3-D_3D_2D_3+D_2D_3\ ,$$
which, in particular, imply that
$D_1D_2D_3-D_3D_2D_1=K_{ij}D_1-D_1K_{ji}=-(K_{ji}D_3-D_3K_{ij}$,
$$D_1D_3D_2-D_2D_3D_1=K_{ji}D_1-D_1K_{ji}+D_1D_2-D_2D_1+D_1D_3-D_3D_1\ ,$$
$$D_2D_1D_3-D_3D_1D_2=D_3K_{ji}-K_{ij}D_3+D_2D_3-D_3D_2+D_1D_3-D_3D_1 \ .$$


Thus, \eqref{eq:kernel reduced} becomes: 
$$x\in \hat {\bf D}(W_{\{i,j\}})_{\le 2}+\ZZ\cdot(K_{ij}D_1-D_1K_{ji})+\ZZ\cdot (K_{ji}D_1-D_1K_{ji})+\ZZ\cdot (D_3K_{ji}-K_{ij}D_3)$$
Note that
$K_{ij}D_1-D_1K_{ji}$, $K_{ji}D_1-D_1K_{ji}$ and $D_3K_{ji}-K_{ij}D_3$ belong to ${\bf K}'_{ij}(W)$.
This proves (a). 

Prove (b). Repeating the argument from the proof of (a), we see that 
$\hat {\bf D}(W_{\{i,j\}})_{\le 2}$  is a free $\ZZ$-module with a basis
$1$, $D_1,D_2,D_3$, $D_aD_b$, for all distinct $a,b\in \{1,2,3\}$, that is, each $x\in \hat{\bf D}(W_{\{i,j\}})_{\le 2}$ can be uniquely  written as  
\begin{equation}
\label{eq:x in D<=2}
x=a_0+\sum_{\ell=1}^3 a_\ell D_\ell+\sum_{a,b\in \{1,2,3\},a\ne b} f_{a,b}D_aD_b\ ,
\end{equation} 
where all $a_k$, $f_{a,b}$ are integers. 

Using the argument from the proof of (a) and \eqref{eq:skDa} we obtain for $k\in \{1,2,3\}$: 
$$d_k(x)=a_k+\sum\limits_{a\ne b} (f_{a,b}(\delta_{k,a}D_b+\delta_{k,b}s_k(D_a))=
a_k+\sum\limits_{b\ne k} f_{k,b} D_b + \sum\limits_{a\ne k} f_{a,k}s_k(D_a)$$
Thus, fixing $k',k''$ such that $\{k,k',k''\}=\{1,2,3\}$, we obtain
$$d_k(x)=a_k+
\begin{cases} 
f_{k,k'}D_{k'}+f_{k,k''}D_{k''}+f_{k',k}D_{k''}+f_{k'',k}D_{k'}
&\text{if $k\in \{1,3\}$}\\
f_{k,k'}D_{k'}+f_{k,k''}D_{k''}+f_{k',k}(1-D_{k''})+f_{k'',k}(1-D_{k'})
&\text{if $k=2$}\\
\end{cases}\ .$$

Thus, $d_k(x)=0$ for $k=1,2,3$ imply that $a_1=a_3=0$, $f_{12}+f_{31}=0$, $f_{13}+f_{21}=0$, $f_{12}+f_{31}=0$, $f_{31}+f_{23}=0$, $f_{32}+f_{13}=0$, $a_2+f_{12}+f_{32}=0$.

Therefore, $x=a_0+f_{13}K_{ij}+f_{31}K_{ji}$.  
This proves (b). 

The lemma is proved.
\end{proof}

To finish the proof of Proposition \ref{pr:K for S3}(b), it suffices to show that 
${\bf K}'_{ij}(W)=\ZZ+{\bf K}_{ij}(W)$ for $m_{ij}=3$. We already have the inclusion $\ZZ+{\bf K}_{ij}(W)\subset {\bf K}'_{ij}(W)$ by \eqref{eq:Kij as kernel}. To show the opposite inclusion note that Lemma \ref{le:drop from 3 to 2} implies that ${\bf K}'_{ij}(W)$ is a $\ZZ$-submodule of $\hat {\bf D}(W_{\{ij}\})$ generated by $1$, $K_{ij}$, $K_{ji}$, 
$\tilde K_{ij}=K_{ij}D_i-D_iK_{ji}=-(K_{ji}D_j-D_jK_{ij})=-\tilde K_{ji}$, $\tilde K'_{ij}=K_{ji}D_i-D_iK_{ji}$, and $\tilde K'_{ji}=K_{ij}D_j-D_jK_{ij}$.

Thus, to prove the inclusion ${\bf K}'_{ij}(W)\subset \ZZ+{\bf K}_{ij}(W)$ it suffices to show that 
\begin{equation}
\label{eq:conjugation S_3}
\{w K_{ij} w^{-1},w \tilde K_{ij} w^{-1},w \tilde K_{ij} w^{-1}\}\subset \hat {\bf D}(W_{\{ij}\})
\end{equation}
for each $w\in W_{\{ij\}}$. 

Note that $Q_{ij}^{(1,1,1)}=K_{ij}$, $Q_{ji}^{(1,1,1)}=K_{ji}$ in the notation of Proposition \ref{pr:Dij2}(a). Thus, 
$$s_iK_{ij}s_i=s_jK_{ij}s_j=K_{ji}, ~s_iK_{ji}s_i=s_jK_{ji}s_j=K_{ij}$$
by Proposition \ref{pr:Dij2}(a), which implies that $K_{ij},K_{ji}\in {\bf K}_{ij}(W)$.
Also 
$$s_i\tilde K_{ij}s_i=K_{ji}(1-s_i-D_i)-(1-s_i-D_i)K_{ij}=K_{ji}(1-D_i)-(1-D_i)K_{ij}=K_{ji}-K_{ij}-\tilde K'_{ij}\in \hat {\bf D}(W_{\{ij\}})$$
because $K_{ji}s_i-s_iK_{ij}=0$. In particular, 
$s_i\tilde K'_{ij}s_i=K_{ij}-K_{ji}-\tilde K_{ij}\in \hat {\bf D}(W_{\{ij\}})$.

Furthermore, 
$$s_j\tilde K_{ij}s_j=K_{ji}s_jD_is_j-s_jD_is_jK_{ij}=\tilde K''_{ij}\in \hat {\bf D}(W_{\{ij}\})\ ,$$
$$s_j\tilde K'_{ij}s_j=K_{ij}s_jD_is_j-s_jD_is_jK_{ji}=\tilde K''_{ji}\in \hat {\bf D}(W_{\{ij}\})$$
where we abbreviated $\tilde K''_{ij}=K_{ji}D_{ij}-D_{ij}K_{ij}$. Finally, 
$$s_i\tilde K''_{ij}s_i=K_{ij}s_iD_{ij}s_i-s_iD_{ij}s_iK_{ji}=K_{ij}D_j-D_jK_{ji}=\tilde K'_{ij}\in \hat {\bf D}(W_{\{ij}\})$$
$$s_j\tilde K''_{ij}s_j=K_{ij}s_jD_{ij}s_j-s_jD_{ij}s_jK_{ji}=K_{ij}D_i-D_iK_{ji}=\tilde K_{ij}\in \hat {\bf D}(W_{\{ij}\})\ .$$

This proves the inclusions \eqref{eq:conjugation S_3}. Thus, ${\bf K}'_{ij}(W)= \ZZ+{\bf K}_{ij}(W)$. 
Together with Lemma \ref{le:drop from 3 to 2} this finishes the proof of Proposition \ref{pr:K for S3}(b). 

Proposition \ref{pr:K for S3} is proved.
\end{proof}

\noindent {\bf Proof of Theorem \ref{th:hopf hat intro W}}.
In the assumptions of Theorem \ref{th:hopf hat intro W}, suppose that $m_{ij}=3$ and let $K'_{ij}\in \hat {\bf H}(W)$
denote $K'_{ij}=-D_js_iD_j+s_iD_jD_i+D_iD_js_i+s_iD_js_i$. 
Clearly,
$$K'_{ij}s_i=-D_jD_{ij}+D_{ij}(1-D_i-s_i)+D_iD_j+s_iD_j=D_iD_j-D_jD_{ij}+D_iD_{ij}+D_{ij}=K_{ij}$$
in the notation of Proposition \ref{pr:K for S3}. We also abbreviate $K'_{ij}:=K_{ij}=D_iD_j-D_jD_i$ if $m_{ij}=2$.

Thus, ${\bf H}(W)$ is the quotient of $\hat {\bf H}(W)$ by the ideal generated by $K'_{ij}$ for all distinct $i,j\in I$ 
Theorem \ref{th:hopf hat intro W} is proved. \endproof

Therefore, Theorem \ref{th:hopf hat intro} is proved. \endproof

\noindent{\bf Proof of Proposition \ref{pr:HS_n factorization} and Theorem \ref{th:hecke hopf intro}}. Since $S_n$ is simply-laced, ${\bf H}(S_n)$ is covered by Theorem \ref{th:hopf hat intro W}. Then Theorem \ref{th:ZWcross and free D} guarantees the factorization of ${\bf H}(S_n)$.
Also the first assertion of Theorem \ref{th:upper Hecke in Hecke-Hopf} for $W=S_n$, $\kk=\ZZ[q,q^{-1}]$ coincides with the assertion of Theorem \ref{th:hecke hopf intro}.

Proposition \ref{pr:HS_n factorization} and Theorem \ref{th:hecke hopf intro} are proved.
\endproof

\noindent{\bf Proof of Proposition \ref{pr:presentation simply-laced D}}. In the proof of Proposition \ref{pr:levi K}, 
we established that ${\bf D}(W)=\hat {\bf D}(W)/\langle \underline {\bf K}(W)\rangle$ for any Coxeter group $W$, where 
\begin{equation}
\label{eq:underline K}
\underline {\bf K}(W)=\sum\limits_{w\in W, i,j\in I:i\ne j} w{\bf K}_{ij}(W)w^{-1}=\sum\limits_{i,j\in I:i\ne j, w\in W^{\{i,j\}}} w{\bf K}_{ij}(W)w^{-1}
\end{equation}
by Lemma \ref{le:W^J conjugation}, where $W^{\{i,j\}}=\{w\in W\,|\,\ell(ws_i)=\ell(w)+1,\ell(ws_j)=\ell(w)+1\}$. 


Now suppose that $W$ is simply-laced, i.e., $m_{ij}\in \{0,2,3\}$. Then,  in view Proposition \ref{pr:K for S3}, the equation \eqref{eq:underline K} reads
\begin{equation}
\label{eq:underline K modified}
\underline {\bf K}(W)=\sum\limits_{i,j\in I:i\ne j, w\in W^{\{i,j\}}} wK_{ij}(W)w^{-1}
\end{equation}
For each compatible pair $(s,s')\in {\mathcal S}\times {\mathcal S}$ define an element $K_{s,s'}\in \hat {\bf H}(W)$ by
$$K_{s,s'}:=\begin{cases} 
0 & \text{if $m_{s,s'}=0$}\\
D_sD_{s'}-D_{s'}D_s & \text{if $m_{s,s'}=2$}\\
D_sD_{s'}-D_{ss's}D_s-D_{s'}D_{ss's}+D_{ss's} & \text{if $m_{s,s'}=3$}\\
\end{cases}$$
in the notation of Proposition \ref{pr:presentation simply-laced D}. Since $wD_iw^{-1}=D_{ws_iw^{-1}}$ whenever $\ell(ws_i)=\ell(w)+1$ by \eqref{eq:si conj Dsi w}, in each of these cases, one has, in the notation of Proposition \ref{pr:K for S3}, $K_{s,s'}=wK_{ij}w^{-1}$ for some distinct $i,j\in I$, $w\in W^{\{i,j\}}$.  

Therefore, for each simply-laced Coxeter group $W$, \eqref{eq:underline K modified} reads:
$\underline {\bf K}(W)=\sum\limits_{\substack{i,j\in I:\\i\ne j, w\in W^{\{i,j\}}}} wK_{s,s'}(W)$, 
where the summation is over all compatible pairs $(s,s')\in {\mathcal S}\times {\mathcal S}$.

The proposition is proved.
\endproof

\noindent{\bf Proof of Proposition \ref{pr:D presentation Sn}}.
Let $W=S_n$ and let $s=(i,j)$, $s'=(k,\ell)$, $1\le i<j\le n$, $1\le k<\ell\le n$ be distinct transpositions in $S_n$.  

$\bullet$ Clearly, $m_{s,s'}=2$, i.e., $s's=s's$ iff $\{i,j\}\cap \{k,\ell\}=\emptyset$; then $(s,s')$ is compatible. 

$\bullet$ Clearly, $m_{s,s'}=3$ iff either $i=k$ or $j=\ell$ or $j=k$ or $i=\ell$; then  $(s,s')$ is compatible precisely in the last two cases.

Finally, this characterization of compatible pairs in $S_n$ and Proposition \ref{pr:presentation simply-laced D} finish the proof. 
\endproof

\subsection{Action on Laurent polynomials and verification of Conjecture \ref{conj:laurent action}}
\label{subsec:Action on Laurent polynomials and verification of Conjecture}

Let ${\mathcal Q}_I$ be the field of fractions of the Laurent polynomial ring  ${\mathcal L}_I=\ZZ[t_i^{\pm 1},i\in I]$. So ${\mathcal Q}_I$ is a purely transcendental field generated by $t_i$, $i\in I$. Since ${\mathcal L}_I$ is a group ring of $\ZZ^I=\oplus_{i\in I}\ZZ\alpha_i$, then the natural reflection action of $W$ on $\ZZ^I$ ($s_i(\alpha_j)=\alpha_j-a_{ij}\alpha_i$) extends to a $W$-action on ${\mathcal Q}_I$ by automorphisms.

\begin{proposition} 
\label{pr:hat H action on LI and vanishing of K_ij} 

For any Coxeter group $W$ the assignments $D_i\mapsto \frac{1}{1-t_i}(1-s_i)$, $s_i\mapsto s_i$, $i\in I$, define a homomorphism of algebras 
$\hat p_W:\hat {\bf H}(W)\to {\mathcal Q}_I\rtimes \ZZ W$. Under this homomorphism, $\hat p_W({\bf K}_{ij}(W))=\{0\}$ whenever $m_{ij}\in \{0,2,3\}$.

\end{proposition}

\begin{proof} It suffices to verify only relations involving $D_i$'s. Indeed, let us abbreviate $\tau_i=\frac{1}{1-t_i}$ and $\underline D_i:=\tau_i(1-s_i)\in {\mathcal Q}_I\rtimes \ZZ W$. Taking into account that $s_i\tau_is_i=\frac{1}{1-t_i^{-1}}=1-\tau_i$, we obtain
$$\underline D_i^2=\tau_i(1-s_i)\tau_i(1-s_i)=(\tau_i^2-\tau_is_i\tau_i)(1-s_i)=(\tau_i^2-\tau_i(1-\tau_i)\cdot s_i)(1-s_i)=(\tau_i^2+\tau_i(1-\tau_i))(1-s_i)=\underline D_i$$
for $i\in I$. Furthermore, let us verify linear braid relations in $\hat {\bf H}(W)$, which we write in the form $wD_i w^{-1}=D_{i'}$ whenever $ws_iw^{-1}=s_{i'}$ and $\ell(ws_i)=\ell(w)+1$. Indeed, for such $i,i'$ and $w$ one has $wt_i w^{-1}=w(t_i)=t_{i'}$ therefore, 
$w\underline D_iw^{-1}=w\tau_i(1-s_i)w^{-1}=\tau_{i'}(1-s_{i'})=\underline D_{i'}$.

This proves the first assertion of the proposition.

Let us prove the second assertion.

Indeed, if $m_{ij}=0$, then ${\bf K}_{ij}(W)=\{0\}$ and we have nothing to prove.

If $m_{ij}=2$, then, according to Proposition \ref{pr:K for S3}(a), ${\bf K}_{ij}(W)=\ZZ \cdot K_{ij}$, where $K_{ij}=D_iD_j-D_jD_i$. Clearly, in this case, $s_is_j=s_js_i$,  $s_it_j=t_js_i$ hence $s_i\tau_j=\tau_js_i$, therefore, 
$\underline D_i\underline D_j=\tau_i(1-s_i)\tau_j(1-s_j)=\tau_j(1-s_j)\tau_i(1-s_i)=\underline D_j\underline D_i$, 
i.e., $\hat p_W({\bf K}_{ij}(W))=0$. 

Let now $m_{ij}=3$. Then, according to Proposition \ref{pr:K for S3}(b),  ${\bf K}_{ij}(W)=\ZZ \cdot K_{ij}+\ZZ \cdot K_{ji}$, where $K_{ij}=D_iD_j-D_jD_{ij}-D_{ij}D_i+D_{ij}$, and $D_{ij}=s_iD_js_i=s_jD_is_j$. Thus, it suffices to show that $\hat p_W(K_{ij})=0$. 
Indeed, $\hat p_W(K_{ij})=\underline D_i\underline D_j-\underline D_j\underline D_{ij}-\underline D_{ij}\underline D_i+\underline D_{ij}$, where $\underline D_{ij}=\tau_{ij}(1-s_{ij})$, and $\tau_{ij}=s_i\tau_js_i=s_j\tau_is_j=\frac{1}{1-t_it_j}$, $s_{ij}=s_is_js_i=s_j\tau_is_j$. 
Let us compute:
$$\underline D_i\underline D_j=\tau_i(1-s_i)\tau_j(1-s_j)=\tau_i(\tau_j-\tau_{ij}s_i)(1-s_j)=\tau_i\tau_j(1-s_j)-\tau_i\tau_{ij}(s_i-s_is_j)\ ,$$
$$\underline D_j\underline D_{ij}=\tau_j(1-s_j)\tau_{ij}(1-s_{ij})=\tau_j(\tau_{ij}-\tau_is_j)(1-s_{ij})=\tau_j\tau_{ij}(1-s_{ij})-\tau_i\tau_j(s_j-s_is_j)\ ,$$
$$\underline D_{ij}\underline D_i=\tau_{ij}(1-s_{ij})\tau_i(1-s_i)=\tau_{ij}(\tau_i-(1-\tau_j)s_{ij})(1-s_i)=\tau_{ij}\tau_i(1-s_i)-\tau_{ij}(1-\tau_j)(s_{ij}-s_is_j)\ .$$

Therefore, 
$\hat p_W(K_{ij})=\tau_i(\tau_j-\tau_{ij}s_i)(1-s_j)=\tau_i\tau_j(1-s_j)-\tau_i\tau_{ij}(s_i-s_is_j)-\tau_j\tau_{ij}(1-s_{ij})$
$$+\tau_i\tau_j(s_j-s_is_j)-\tau_{ij}\tau_i(1-s_i)+\tau_{ij}(1-\tau_j)(s_{ij}-s_is_j)+\tau_{ij}(1-s_{ij})$$
$$=k_{ij}-\tau_i\tau_js_j-\tau_i\tau_{ij}(s_i-s_is_j)+\tau_j\tau_{ij}s_{ij}+\tau_i\tau_j(s_j-s_is_j)+\tau_{ij}\tau_is_i+\tau_{ij}(1-\tau_j)(s_{ij}-s_is_j)-\tau_{ij}s_{ij}$$
$$=k_{ij}-k_{ij}s_is_j\ ,$$
where $k_{ij}=\tau_i\tau_j-\tau_j\tau_{ij}-\tau_{ij}\tau_i+\tau_{ij}$. Thus, $\hat p_W(K_{ij})=0$ because $k_{ij}=0$.
This finishes the proof of the second assertion of the proposition. 

The proposition is proved.
\end{proof}

\noindent {\bf Verification of Conjecture \ref{conj:laurent action} in the simply-laced case}. The following is an immediate corollary of Proposition \ref{pr:hat H action on LI and vanishing of K_ij}.

\begin{corollary} 
\label{cor:H(W) action on QI}
Suppose that $W$ is a simply-laced Coxeter group, i.e.,  $m_{ij}\in \{0,2,3\}$ for $i,j\in I$. 
Then, in the notation of Proposition \ref{pr:hat H action on LI and vanishing of K_ij}, the assignments $D_i\mapsto \frac{1}{1-t_i}(1-s_i)$, 
$s_i\mapsto s_i$, $i\in I$, define a homomorphism of algebras $p_W:{\bf H}(W)\to {\mathcal Q}_I\rtimes \ZZ W$.  
 
\end{corollary}

Note that ${\mathcal Q}_I\rtimes \ZZ W$ naturally acts on ${\mathcal Q}_I$ via $(t w)(t')=t\cdot w(t')$ 
for $t,t'\in {\mathcal Q}_I$, $w\in W$. Composing this with $p_W$ gives an action of ${\bf H}(W)$ on ${\mathcal Q}_I$, under which ${\mathcal L}_I$ is invariant, thus both ${\mathcal Q}_I$ and ${\mathcal L}_I$ are module algebras over ${\bf H}(W)$.
For $W$ simply-laced this, taken together with Corollary \ref{cor:H(W) action on QI} turns ${\mathcal Q}_I$ into a module algebra over ${\bf H}(W)$, so that ${\mathcal L}_I$ is a module subalgebra. Since the above action of ${\bf H}(W)$ on ${\mathcal L}_I$ coincides with \eqref{eq:demazure W}, this verifies Conjecture \ref{conj:laurent action} for all simply-laced  $W$.
\endproof

\noindent{\bf Proof of Proposition \ref{pr:demazure}}. Let $W=S_n$, so that $I=\{1,\ldots,n-1\}$ and ${\mathcal L}_I=\ZZ[t_1^{\pm 1},\ldots,t_{n-1}^{\pm 1}]$. Also denote ${\mathcal P}_n:=\ZZ[x_1,\ldots,x_n]$ and let ${\mathcal Q}_n$ be the field of fractions of ${\mathcal P}_n$. We identify ${\mathcal Q}_I$ with the subfield of ${\mathcal Q}_n$ generated by  $t_i=\frac{x_i}{x_{i+1}}$, $i=1,\ldots,n-1$. We have a natural $S_n$-action on ${\mathcal Q}_n$ by permutations so that its restriction to ${\mathcal Q}_I$ coincides with the natural $S_n$-action on ${\mathcal Q}_I$. In particular, 
this defines a natural action of ${\mathcal Q}_I\rtimes \ZZ S_n$ on ${\mathcal Q}_n$ via $(tw)(x)=t\cdot w(x)$ for $t\in {\mathcal Q}_I$, $x\in {\mathcal Q}_n$, $w\in S_n$.

For $W=S_n$ this, taken together with Corollary \ref{cor:H(W) action on QI} defines a structure of a module algebra over ${\bf H}(S_n)$
on ${\mathcal Q}_n$, so that ${\mathcal P}_n$ is a module subalgebra. This proves Proposition \ref{pr:demazure} because the above action coincides with the one given by \eqref{eq:demazure W}. \endproof

\section{Appendix: deformed semidirect products}

For readers' convenience, in this section we state relevant results about deformations of cross products, see also \cite{shepler} and the forthcoming joint paper of Yury Bazlov with the first author \cite{BBHcross}.

Throughout this section, we fix a commutative  ring $R$. Let $A$ and $B$ be  unital $R$-algebras and let $\Psi:B\otimes A\to A\otimes B$ be an $R$-linear map (all tensor products are over $R$).

Define a (possibly non-associative) multiplication on $A\otimes B$ by:
$(a'\otimes b)(a\otimes b')=a'\Psi(b\otimes a)b'$
for all $a,a'\in A$, $b,b'\in B$
and denote the resulting algebra by  $A\otimes_\Psi B$. 

Note that $1\otimes 1$ is a unit of $A\otimes_\Psi B$ iff
\begin{equation}
\label{eq:unit psi}
\Psi(1\otimes a)=a\otimes 1,~\Psi(b\otimes 1)=1\otimes b
\end{equation} 
for all $a\in A$, $b\in B$ (in that case, $A\otimes 1$ and $1\otimes B$ are  subalgebras of $A\otimes_\Psi B$).

We need the following result from \cite{M} (due to its importance, we provide a proof).

\begin{proposition} [\text{\cite[Proposition 21.4]{M}}]
\label{pr:Majid}
Let $A$ and $B$ be associative unital $R$-algebras and $\Psi:B\otimes A\to A\otimes B$ be an $R$-linear map satisfying \eqref{eq:unit psi}. Then the $R$-algebra $A\otimes_\Psi B$ is  associative  iff  
the following diagrams are commutative:
\begin{equation}
\label{eq:commutative diagram Psi 3}
\begin{CD}
B \otimes A\otimes A@>(1\otimes \Psi)\circ (\Psi\otimes 1)>>A\otimes A\otimes B\\
@VV1\otimes {\bf m}_AV@V{\bf m}_A\otimes 1 VV\\
B \otimes A @>\Psi>>A\otimes B
\end{CD}\ ,\quad 
\begin{CD}
B \otimes B\otimes A@>(\Psi \otimes 1)\circ (1\otimes \Psi)>>A\otimes B\otimes B\\
@VV{\bf m}_B\otimes 1V@V1\otimes {\bf m}_B VV\\
B \otimes A @>\Psi>>A\otimes B
\end{CD}. 
\end{equation}
where ${\bf m}_A$ (resp. ${\bf m}_B$) is the multiplication map $A\otimes A\to A$ (resp. $B\otimes B\to B$). 
\end{proposition}

\begin{proof} Indeed, suppose that $A\otimes_\Psi B$ is an associative algebra. Clearly, the associativity equations 
\begin{equation}
\label{eq:weak associativity}
(1\otimes b)(aa'\otimes 1)=((1\otimes b)(a\otimes 1))(a'\otimes 1),~(1\otimes b'b)(a\otimes 1)=(1\otimes b')(1\otimes b)(a\otimes 1)
\end{equation}
for all $a,a'\in A$, $b,b'\in B$ 
are respectively equivalent to the commutativity of the diagrams \eqref{eq:commutative diagram Psi 3}. 

Conversely, suppose that the diagrams \eqref{eq:commutative diagram Psi 3} are commutative, that is, \eqref{eq:weak associativity} hold. These, taken together with obvious relations $(a\otimes b)=(a\otimes 1)(1\otimes b)$ for $a\in A$, $b\in B$ and:
\begin{equation}
\label{eq:weak associativity2}
(a''\otimes 1)(zz')=((a''\otimes 1)z)z',~(zz')(1\otimes b'')=z(z'(1\otimes b''))
\end{equation}
for any $a''\in A$, $b''\in B$, $z,z'\in A\otimes B$, imply 
\begin{equation}
\label{eq:weak associativity3}
(a''\otimes b)(aa'\otimes b')=((a''\otimes b)(a\otimes 1))(a'\otimes b'),~(a'\otimes b'b)(a\otimes b'')=(a'\otimes b')((1\otimes b)(a\otimes b''))
\end{equation}
for all $a,a',a''\in A$, $b,b',b''\in B$.
In view of the obvious relations $aa'\otimes b'=(a\otimes 1)(a'\otimes b')$ and $a'\otimes b'b=(a'\otimes b')(1\otimes b)$, the relations \eqref{eq:weak associativity3} are equivalent to 
\begin{equation}
\label{eq:weak associativity4}
z((a\otimes 1)z')=(z(a\otimes 1))z',~z((1\otimes b)z')=(z(1\otimes b))z'
\end{equation}
for all $a,a'\in A$, $b,b'\in B$, $z\in A\otimes B$. 

Finally, using  by \eqref{eq:weak associativity4} repeatedly, we obtain
$$(a_1\otimes b_1)((a_2\otimes b_2)(a_3\otimes b_3))=(a_1\otimes b_1)(a_2\Psi(b_2\otimes a_3)b_3)=(a_1\otimes b_1)((a_2\otimes 1)(\Psi(b_2\otimes a_3)b_3)$$
$$=((a_1\otimes b_1)(a_2\otimes 1))(\Psi(b_2\otimes a_3)b_3)=(a_1\Psi(b_1\otimes a_2))((1\otimes b_2)(a_3\otimes b_3))=(a_1\Psi(b_1\otimes a_2))(1\otimes b_2))(a_3\otimes b_3)$$
$$=(a_1\Psi(b_1\otimes a_2)b_2)(a_3\otimes b_3)=((a_1\otimes b_1)(a_2\otimes a_2))(a_3\otimes b_3)\ .$$

This proves associativity of $A\otimes_\Psi B$.

The proposition is proved.  
\end{proof}

We say that $A\otimes_\Psi B$ is {\it left} associative (resp. {\it right} associative) if the first (resp. the second) diagram \eqref{eq:commutative diagram Psi 3} is commutative. According to Proposition \ref{pr:Majid}, $A\otimes_\Psi B$ is an associative $R$-algebra iff it is both left and right associative.
In particular, taking $B=R W$, where $W$ is a monoid acting on $A$ by $R$-linear endomorphisms and $\Psi:R W\otimes A\to A\otimes R W$ given by $\Psi(w\otimes a)=w(a)\otimes w$, $w\in W$, $a\in A$, we recover the following  well-known result.

\begin{corollary} (semidirect product) Let $W$ be a monoid and $A$ be an $R W$-module algebra (i.e., $W$ acts on $A$ by $R$-linear algebra endomorphisms). Then the space $A\otimes R W$ is an associative $R$-algebra with the product given by
$(a'\otimes w)(a\otimes w')=a'\cdot w(a)\otimes ww'$
for all $a,a'\in A$, $w,w'\in W$. 

\end{corollary}


For an $R$-module $V$ denote by $T(V)$ its tensor algebra 
$\oplus_{n\ge 0} V^{\otimes n}$ of $V$.

\begin{proposition} 
\label{pr:left associative}
In the assumptions of Proposition \ref{pr:Majid} suppose that $A=T(V)$ for some $R$-module $V$. Then for any $R$-linear map $\mu:B\otimes V\to T(V)\otimes B$ satisfying
\begin{equation}
\label{eq:unit mu}
\mu(1\otimes x)=x\otimes 1
\end{equation} 
for all $x\in T(V)$, 
there exists a unique $\Psi^\mu:B\otimes T(V)\to T(V)\otimes B$ 
such that $T(V)\otimes_{\Psi^{\mu}} B$ is left associative with unit $1\otimes 1$
and $\Psi^\mu|_{B\otimes V}=\mu$. 
\end{proposition}

\begin{proof} Define $\Psi^\mu:=\oplus_{n\ge 0} \Psi^{(n)}$, where $\Psi^{(n)}$ is an $R$-linear map $B\otimes V^{\otimes n}\to T(V)\otimes B$ given by

$\bullet$ $\Psi^{(0)}(b\otimes 1)=1\otimes b$ for all $b\in B$.

$\bullet$ $\Psi^{(n)}=\mu_n\circ \cdots \circ \mu_1$ for $n\ge 1$, where $\mu_i:T(V)^{\otimes i-1}\otimes B\otimes V^{\otimes n+1-i}
\to T(V)^{\otimes i}\otimes B\otimes V^{\otimes n-i}$ is given by 
$\mu_i=1\otimes \cdots \otimes 1\otimes \mu\otimes 1\otimes \cdots \otimes 1$.   
 
Taking into account that $V^{\otimes m}\otimes V^{\otimes n}=V^{\otimes m+n}$ for $m,n\ge 0$, we immediately obtain
$$\Psi^{(m+n)}=(\mu_{m+n}\circ \cdots \circ \mu_{m+1})\circ (\mu_m\circ \cdots \circ \mu_1)=(1\otimes \Psi^{(n)})\circ (\Psi^{(m)}\otimes 1)$$
which implies that $\Psi^\mu=({\bf m}_{T(V)}\otimes 1)\circ (1\otimes \Psi^\mu)\circ (\Psi^\mu\otimes 1)$
i.e., the following diagram is commutative. 
$$\begin{CD}
B \otimes T(V)\otimes T(V)@>(1\otimes \Psi^\mu)\circ (\Psi^\mu\otimes 1)>>T(V)\otimes T(V)\otimes B\\
@VV1\otimes {\bf m}_{T(V)} V@V{\bf m}_{T(V)}\otimes 1 VV\\
B \otimes T(V) @>\Psi^\mu>>T(V)\otimes B
\end{CD}.
$$

The above diagram is the first diagram \eqref{eq:commutative diagram Psi 3} for $A=T(V)$, hence,  $T(V)\otimes_{\Psi^{\mu}} B$ is left associative. By the construction, ${\Psi^{\mu}}$ satisfies \eqref{eq:unit psi}.

Clearly, $\Psi^\mu$ is uniquely determined by the assumptions of the proposition.

The proposition is proved.
\end{proof}

\begin{proposition} 
\label{pr:right associative}
Let $V$ be an $R$-module, $B$ be an $R$-algebra, and $\mu:B\otimes V\to T(V)\otimes B$ be an $R$-linear map satisfying \eqref{eq:unit mu}. Then $T(V)\otimes_{\Psi^\mu} B$ 
is an associative $R$-algebra  iff the following diagram is commutative:
\begin{equation}
\label{eq:commutative diagram Psi 4}
\begin{CD}
B \otimes B\otimes V@>(\Psi^\mu \otimes 1)\circ (1\otimes \mu)>>T(V)\otimes B\otimes B\\
@VV{\bf m}_B\otimes  1V@V1\otimes {\bf m}_B VV\\
B \otimes T(V) @>\Psi^\mu>>T(V)\otimes B
\end{CD}. 
\end{equation}

\end{proposition}

\begin{proof}  We need the following result.

\begin{lemma} In the assumptions of Proposition \ref{pr:right associative}, commutativity of \eqref{eq:commutative diagram Psi 4} implies that the following diagram is commutative for all $n\ge 0$.
\begin{equation}
\label{eq:commutative diagram Psi 2p}
\begin{CD}
B \otimes B\otimes V^{\otimes n}@>(\Psi^\mu\otimes 1)\circ (1\otimes \Psi^\mu)>>T(V)\otimes B\otimes B\\
@VV{\bf m}_B\otimes 1V@V 1\otimes {\bf m}_B VV\\
B \otimes V^{\otimes n} @>\Psi^\mu>>T(V)\otimes B
\end{CD}.
\end{equation}
\end{lemma} 

\begin{proof} We proceed by induction in $n$. If $n=0,1$, the assertion is obvious. Suppose that $n\ge 2$. Tensoring the commutative diagram \eqref{eq:commutative diagram Psi 2p} for $V^{\otimes n-1}$ with $V$ from the right and then horizontally composing with the commutative diagram \eqref{eq:commutative diagram Psi 4} (which is tensored with $T(V)$ from the left), followed by the multiplications $T(V)\otimes T(V)\to T(V)$ and $B\otimes B\to B$, we obtain a commutative diagram:
$$\begin{CD}
B \otimes B\otimes V^{\otimes n}@>>>T(V)\otimes B\otimes B\otimes V @>>>T(V)\otimes T(V)\otimes B\otimes B@>{\bf m}_{T(V)}>>T(V)\otimes B\otimes B\\
@VV{\bf m}_B\otimes 1V@V1\otimes {\bf m}_B\otimes 1VV@VV1\otimes 1\otimes {\bf m}_B V @VV1\otimes {\bf m}_B V\\
B \otimes V^{\otimes n} @>>>T(V)\otimes B\otimes V @>>>T(V)\otimes T(V)\otimes B @>{\bf m}_{T(V)}>>T(V)\otimes B
\end{CD}.
$$
Finally,  left associativity of $T(V)\otimes_{\Psi^\mu} B$, i.e., commutativity of the first diagram \eqref{eq:commutative diagram Psi 3} established in Proposition \ref{pr:left associative} for $A=T(V)$ implies that the composition of top (resp. bottom) horizontal arrows in the above  diagram is $(\Psi^\mu\otimes 1)\circ (1\otimes \Psi^\mu)$ (resp. $\Psi^\mu$). This finishes the proof of the lemma. 
\end{proof}

Clearly, commutativity of \eqref{eq:commutative diagram Psi 2p} for all $n\ge 0$ is equivalent to commutativity of the second diagram \eqref{eq:commutative diagram Psi 3} with $A=T(V)$, i.e., to the right associativity of $T(V)\otimes_{\Psi^\mu} B$. 

The proposition is proved.
\end{proof}


For each $R$-linear map $\mu:B\otimes V\to T(V)\otimes B$  consider the category ${\mathcal C}_\mu$ whose objects are
associative $R$-algebras $A$ generated by $B$ and $V$ such that:

$\bullet$ $b\cdot v={\bf m}_A\circ \mu(v\otimes b)$ for all $b\in  B$, $v\in V$;

$\bullet$ The assignment $b\mapsto 1\cdot b$ is a (not necessarily injective) algebra homomorphism $\iota_A:B\to A$;

\noindent morphisms are surjective algebra homomorphisms $f:A\twoheadrightarrow A'$ such that $\iota_{A'}=f\circ \iota_A$, $\iota_{A',V}=f\circ \iota_{A,V}$, where $\iota_{A'',V}$ stands for the natural (not necessarily injective) $R$-linear map $V\to A''$.  

Clearly, ${\mathcal C}_\mu$ is a partially ordered set with a unique maximal element $A_\mu$, i.e., for any $A\in {\mathcal C}_\mu$ one has a surjective algebra homomorphism $A_\mu \twoheadrightarrow A$. It is also clear that  $A_\mu$ is the quotient of the free product $T(V)*B$ by the ideal $I_\mu$ generated by all elements of the form 
\begin{equation}
\label{eq:Agamma}
b*v-{\bf j}(\mu(b\otimes v))
\end{equation}
for all $b\in B$, $v\in V$, where ${\bf j}:V\otimes B\hookrightarrow T(V)*B$ is a natural embedding given by ${\bf j}(v'\otimes b')=v'*b'$ for all $b'\in B$, $v'\in V$.

For any (associative or not) ring $A$ denote by $J_A$ the left ideal generated by all elements of the form $r_{a,b,c}=a(bc)-(ab)c$, $a,b,c\in A$.
The  identity 
$a\cdot r_{b,c,d}+r_{a,b,c}\cdot d=r_{ab,c,d}-r_{a,bc,d}+r_{a,b,cd}$
for $a,b,c,d\in R$
implies that $J_A$ is also a right ideal. Then denote $\underline A:=A/J_A$. Clearly, $\underline A$ is associative and is universal in 
the sense that for any surjective homomorphism $A\twoheadrightarrow A'$ where $A'$ is an associative ring there is a surjective homomorphism 
$\underline A\twoheadrightarrow A'$.

\begin{theorem} 
\label{th:from nonassociative to associative}
For any (unital associative) $R$-algebra $B$, an $R$-module $V$ and an $R$-linear map $\mu:B\otimes V\to T(V)\otimes B$ satisfying \eqref{eq:unit mu} one has
$\underline{T(V)\otimes_{\Psi^\mu} B}= A_\mu$.
In particular,  $A_\mu=T(V)\otimes_{\Psi^\mu} B$ iff $T(V)\otimes_{\Psi^\mu} B$ is associative, i.e., iff the diagram \eqref{eq:commutative diagram Psi 4} is commutative.



\end{theorem}

\begin{proof} Denote $A'_\mu=\underline{T(V)\otimes_{\Psi^\mu} B}$ and by $\pi_\mu$ the structural homomorphism $T(V)\otimes_{\Psi^\mu} B\twoheadrightarrow A'_\mu$. Clearly, $A'_\mu$ is an $R$-algebra and:

$\bullet$ $A'_\mu$ is generated by $V$ and $B$.

$\bullet$ $b\cdot v={\bf m}_{A'_\mu}\circ \mu(v\otimes b)$ for all $b\in B$, $v\in V$.

$\bullet$ The assignment $b\mapsto \pi_\mu(1\otimes b)$ is an algebra homomorphism $B\to A'_\mu$.

Therefore, $A'_\mu$ is an object of the category ${\mathcal C}_\mu$ and thus one has a canonical surjective algebra homomorphism $\pi'_\mu:A_\mu\twoheadrightarrow A'_\mu$. On the other hand, universality of $A'_\mu$ implies that there is a canonical surjective $R$-algebra algebra homomorphism $A'_\mu\twoheadrightarrow A_\mu$.   Thus, $\pi'_\mu$ is an isomorphism, hence it is the identity, i.e., $A'_\mu=A_\mu$.

The theorem is proved. \end{proof}

In some cases  conditions \eqref{eq:unit mu} and \eqref{eq:commutative diagram Psi 4} can be simplified. The following is immediate consequence of Proposition \ref{pr:right associative}.

\begin{corollary} 
\label{cor:mu nu gamma} Let $V$ be an $R$-module, $B$ be an $R$-algebra, and  $\mu:B\otimes V\to  T(V)\otimes B$ be given by 
$\mu=\nu+\beta$, 
where $\nu:B\otimes V\to V\otimes B$ and $\beta:B\otimes V\to B$ are $R$-linear maps such that $\nu(1\otimes v)=v\otimes 1$  
for all $v\in V$. 
Then  $A_\mu=T(V)\otimes B$ as an $R$-module iff the following conditions hold.

$\bullet$ $\nu\circ ({\bf m}_B\otimes Id_V)=(Id_B\otimes {\bf m}_B)\circ (\nu\otimes Id_B)\circ (Id_B\otimes \nu)$
in $Hom_R(B\otimes B\otimes V,V\otimes B)$.

$\bullet$ $\beta\circ ({\bf m}_B\otimes Id_V)={\bf m}_B\circ (Id_B\otimes \beta)+{\bf m}_B\circ (\beta\otimes Id_B)\circ (Id_B\otimes \nu)$
in $Hom_R(B\otimes B\otimes V,B)$.

\end{corollary} 

%

We conclude with the discussion of factorizable (in the sense of Proposition \ref{pr:Majid}) algebras with $B=R W$, the {\it linearization} of a monoid $W$, so that $RW$ is naturally an algebra over $R$.

\begin{proposition} 
\label{pr:factored derivatives}
Given an $R$-algebra ${\bf H}$, suppose that it factors as 
${\bf H}={\bf D}\cdot R W$ over $R$, where $W$ is a monoid (i.e., 
the multiplication map defines an isomorphism of  $R$-modules
${\bf D}\otimes RW \widetilde \longrightarrow {\bf H}$)  and both ${\bf D}$  and $RW$ are subalgebras of ${\bf H}$. 
Then for any $g,h\in W$ there exists a unique $R$-linear map $\partial_{g,h}:{\bf D}\to {\bf D}$ such that:
\begin{equation}
\label{eq:factored derivatives}
g x=\sum_{w\in W} \partial_{g,w}(x) w
\end{equation}
for all $g\in W$, $x\in {\bf D}$. 

Moreover, the family $\{\partial_{g,h}\}$ satisfies: $\partial_{g,h}(xy)=\sum\limits_{w\in W}
\partial_{g,w}(x)\partial_{w,h}(y)$
for all $g,h\in W$, $x,y\in {\bf D}$ and 
$\partial_{gh,w}(x)=\sum\limits_{w_1,w_2\in W:w_1w_2=w}
\partial_{g,w_1}(\partial_{h,w_2}(x))$
for all $g,h,w\in W$, $x\in {\bf D}$.

\end{proposition}

\begin{proof} Indeed, the existence and uniqueness of follows from the factorization of ${\bf H}$, i.e., that ${\bf H}$ is a free left ${\bf D}$-module with the basis $W$. To prove the second assertion, note that 
$$gxy=\sum_{w\in W} \partial_{g,w}(x) wy=\sum_{w\in W} \partial_{g,w}(x) \left(\sum_{h\in W} \partial_{w,h}(y)h\right)=\sum_{h\in W} \left(\sum_{w\in W} \partial_{g,w}(x) \partial_{w,h}(y)\right)h$$
for $g\in W$, $x,y\in {\bf D}$ and
$$ghx=\sum_{w_2\in W} g\partial_{h,w_2}(x) w_2=\sum_{w_2\in W} \left(\sum_{w_1\in W} \partial_{g,w_1}(\partial_{h,w_2}(x))w_1\right) w_2=\sum_{w_1,w_2\in W} \partial_{g,w_1}(\partial_{h,w_2}(x))w_1w_2$$
for $g,h\in W$, $x\in {\bf D}$.
The proposition is proved.
\end{proof}

\end{document}